%% file: main.tex
\documentclass{article}

\usepackage{geometry}
\usepackage{amsthm}
\usepackage{amssymb}
\usepackage{amsfonts}
\usepackage{amsmath}
\usepackage[shortlabels]{enumitem}
\usepackage{adjustbox}
\usepackage[T1]{fontenc}
\usepackage{adjustbox}

\usepackage{appendix}

\usepackage{algorithm}
 \usepackage{algpseudocode}
 \usepackage{algorithmicx}
 \algdef{SE}[DOWHILE]{Do}{doWhile}{\algorithmicdo}[1]{\algorithmicwhile\ #1}%

\algrenewcommand\algorithmicrequire{\textbf{Input:}}
\algrenewcommand\algorithmicensure{\textbf{Output:}}

\makeatletter
\newenvironment{breakablealgorithm}
  {% \begin{breakablealgorithm}
   \begin{center}
     \refstepcounter{algorithm}% New algorithm
     \hrule height.8pt depth0pt \kern2pt% \@fs@pre for \@fs@ruled
     \renewcommand{\caption}[2][\relax]{% Make a new \caption
       {\raggedright\textbf{\fname@algorithm~\thealgorithm} ##2\par}%
       \ifx\relax##1\relax % #1 is \relax
         \addcontentsline{loa}{algorithm}{\protect\numberline{\thealgorithm}##2}%
       \else % #1 is not \relax
         \addcontentsline{loa}{algorithm}{\protect\numberline{\thealgorithm}##1}%
       \fi
       \kern2pt\hrule\kern2pt
     }
  }{% \end{breakablealgorithm}
     \kern2pt\hrule\relax% \@fs@post for \@fs@ruled
   \end{center}
  }
\makeatother
\DeclareFontFamily{U}{mathx}{}
\DeclareFontShape{U}{mathx}{m}{n}{<-> mathx10}{}
\DeclareSymbolFont{mathx}{U}{mathx}{m}{n}
\DeclareMathAccent{\widehat}{0}{mathx}{"70}
\DeclareMathAccent{\widecheck}{0}{mathx}{"71}

\DeclareMathOperator{\Mod}{Mod}
\DeclareMathOperator{\SL}{SL}
\DeclareMathOperator{\PSL}{PSL}
\DeclareMathOperator{\Aut}{Aut}
\DeclareMathOperator{\Hom}{Hom}
\DeclareMathOperator{\Epi}{Epi}
\DeclareMathOperator{\Out}{Out}

\DeclareMathOperator{\Homeo}{Homeo}
\DeclareMathOperator{\trace}{trace}
\DeclareMathOperator{\sgn}{sgn}
\DeclareMathOperator{\cxty}{cxty}

\usepackage{tikz} 
\usetikzlibrary{positioning}

\usepackage{diagbox}

\theoremstyle{definition}
\newtheorem*{defx*}{Definition}
\newtheorem{definition}{Definition}[subsection]
\newtheorem{example}[definition]{Example}

\theoremstyle{plain}
\newtheorem{thmx}{Theorem}
\newtheorem{colx}{Corollary}
\newtheorem{theorem}[definition]{Theorem}
\newtheorem{lemma}[definition]{Lemma}
\newtheorem{corollary}[definition]{Corollary}
\newtheorem{proposition}[definition]{Proposition}

\newtheorem{question}[definition]{Question}

\theoremstyle{remark}
\newtheorem{remark}[definition]{Remark}

\newcommand{\RNum}[1]{\lowercase\expandafter{\romannumeral #1\relax}}
\newcommand{\URNum}[1]{\uppercase\expandafter{\romannumeral #1\relax}}

\makeatletter
\newcommand\xleftrightarrow[2][]{%
  \ext@arrow 9999{\longleftrightarrowfill@}{#1}{#2}}
\newcommand\longleftrightarrowfill@{%
  \arrowfill@\leftarrow\relbar\rightarrow}
\makeatother

\usepackage{url}

\usepackage{hyperref}

\usepackage[backend=biber,style=alphabetic,sorting=nyt,maxbibnames=10,maxalphanames=10,url=true,doi=false,giveninits=true]{biblatex}
\usepackage{xurl}
\AtEveryBibitem{\clearlist{language}}
\title{Classification of Torus Fibrations Over $S^2$\\Up to Fibre Sum Stabilisation}
\author{
  \Large{Yibo Zhang}
  \footnote{
    This work is supported by the French National Research Agency in the framework of the \guillemotleft France 2030 \guillemotright program (ANR-15-IDEX-0002) and by the LabEx PERSYVAL-Lab (ANR-11-LABX-0025-01).
  }\\
\small{\emph{Institut Fourier, UMR 5582, Laboratoire de Mathématiques}}\\
\small{\emph{Université Grenoble Alpes, CS 40700, 38058 Grenoble cedex 9, France}}\\
\small{\emph{email:}} \tt{yibo.zhang@univ-grenoble-alpes.fr}
}

\date{}

\begin{document}

%%%%%%%%%%%%%%%%%%%%%%%%%%%%%%%%%%%%%%%%%%%%%%%%%%%%%%%%%%%%%%%%%%%%%%%%%%%%%%

\maketitle

\begin{abstract}
    We study torus fibrations over the $2$-sphere and Hurwitz equivalence of their monodromies.
    We show that, if two torus fibrations over $S^2$ have the same type of
    singularities, then their global monodromies are Hurwitz equivalent after performing direct sums with a certain torus Lefschetz fibration.
    The additional torus Lefschetz fibration is universal when the type of singularities is ``simple''. 
\end{abstract}

\tableofcontents

\input{introduction}

\input{Section2/elem_trans}
\input{Section2/contraction}
\input{Section2/direct_sum}
\input{Section2/equivalence}

\input{Section3/swappability}
\input{Section3/local_global}
\input{Section3/singular_fibr}

\input{Section4/short}
\input{Section4/conjugate}
\input{Section4/classification}
\input{Section4/almost}

\input{Section5/fails}
\input{Section5/achiralLef}

\sloppy
\printbibliography[
    heading=bibintoc,
    title={References}
]

\appendixtitleon
\appendixtitletocon
\begin{appendices}
\input{computability}
\end{appendices}

\end{document}

%% file: introduction.tex
\section{Introduction}

A generalised \emph{torus fibration} over the $2$-sphere is a continuous map $f:M^4\rightarrow S^2$ from a closed oriented $4$-manifold $M^4$ to the $2$-sphere $S^2$, for which there exists some finite set $\mathcal{B}\subset S^2$,
called the \emph{branch set},
so that the restriction of $f$ to $M^4\setminus f^{-1}(\mathcal{B})$ is a locally trivial fibration over $S^2\setminus\mathcal{B}$
with fibre a torus.
In this paper, we will study torus fibrations over $S^2$ and provide an algebraic classification of their monodromies.

Given a torus fibration $f:M^4\rightarrow S^2$, we suppose that the branch set for $f$ is $\mathcal{B}=\{p_1,\ldots,p_n\}$ and choose a \emph{base point} $p\in S^2\setminus\mathcal{B}$.
As in \cites[p.176]{moishezon1977complex}{kas1980handlebody}{funar2022singular}, the locally trivial fibration $f:M^4\setminus f^{-1}(\mathcal{B})\rightarrow S^2\setminus \mathcal{B}$ determines a \emph{monodromy homomorphism}
\begin{equation*}
    \Phi_{f,p}:\pi_1(S^2\setminus \mathcal{B},p)\rightarrow \Mod(\mathbb{T}^2),
\end{equation*}
by identifying $f^{-1}(p)$ with $\mathbb{T}^2$,
where $\Mod(\mathbb{T}^2)=\pi_0 \Homeo^+(\mathbb{T}^2)$ is the mapping class group of torus and hence isomorphic to $\SL(2,\mathbb{Z})$.

Choose homotopy classes of loops $\gamma_1,\ldots,\gamma_n\subset S^2$ based at $p$ such that each loop $\gamma_j$ goes around some branch point $p_i$ exactly once clockwise
and the fundamental group $\pi_1(S^2\setminus\mathcal{B},p)$ is generated by $\gamma_1,\ldots,\gamma_n$ with the relation $\gamma_1\cdots\gamma_n=1$. Therefore, the group $\pi_1(S^2\setminus\mathcal{B} ,p)$ is isomorphic to $\mathbb{F}_{n-1}$.
The monodromy $\phi_j = \Phi_{f,p}(\gamma_j) \in \SL(2,\mathbb{Z})$ is called the \emph{fibre monodromy} around the singular fibre $f^{-1}(p_i)$.
Note that $\phi_1\cdots\phi_n=1$.

The $n$-tuple $(\phi_1,\ldots,\phi_n)$ is called a \emph{global monodromy} of $f$ and the homomorphism $\Phi_{f,p}$ is uniquely determined by $p$, $\gamma_1,\ldots,\gamma_n$ and $(\phi_1,\ldots,\phi_n)$.
However, a different choice of the base point $p$ amounts to changing $(\phi_1,\ldots,\phi_n)$ by a diagonal (or simultaneous) conjugacy.
Also a different choice of homotopy classes of $\gamma_1,\ldots,\gamma_n$ may change $(\phi_1,\ldots,\phi_n)$ by a sequence of elementary transformations (or Hurwitz moves; see Subsection \ref{subsection::elementary transformations} for more details):
\begin{equation*}
    (\ldots,\phi_{i}\phi_{i+1}\phi_i^{-1},\phi_i,\ldots)
    \xleftarrow{R_i^{-1}}
    (\ldots,\phi_i,\phi_{i+1},\ldots)
    \xrightarrow{R_i}(\ldots,\phi_{i+1},\phi_{i+1}^{-1}\phi_i\phi_{i+1},\ldots),
    1\le i\le n-1.
\end{equation*}

We will denote a multi-set by $[x_1, x_2, x_2, x_3, x_3, x_3, \ldots]$ and denote the conjugacy class of an element $g$ in a group $G$ by $Cl_G(g)$ (or $Cl(g)$ if we do not specify $G$).

\begin{defx*}
Let $f:M^4\rightarrow S^2$ be a torus fibration over $S^2$ with $n$ branch points and $(\phi_1,\ldots,\phi_n)$ be a global monodromy of $f$.
The \emph{type} (of singularities) of $f$
is defined to be the multi-set
\begin{equation*}
\mathcal{O}(f)=[Cl_{\SL(2,\mathbb Z)}(\phi_1), \ldots, Cl_{\SL(2,\mathbb Z)}(\phi_n)],
\end{equation*}
which does not depend on the choice of its global monodromy.
\end{defx*}

%The elementary transformations are used for the Artin's representation which represents the braid group $B_n$ as a subgroup of $Aut(\mathbb{F}_n)$.
%Besides, the action of $Aut(\mathbb{F}_n)$ on $\Hom(\mathbb{F}_n,SL(2,\mathbb{Z}))$ by precomposition has a $B_n$-invariant subset identified with $\Hom(\mathbb{F}_{n-1},SL(2,\mathbb{Z}))$.
%Since the group $\pi_1(S^2\setminus\mathcal{B})$ is isomorphic to $\mathbb{F}_{n-1}$, we then consider the orbit space

There is a left action of $\Aut(\mathbb F_n)$ on $\Hom(\mathbb F_n, \SL(2,\mathbb{Z}))$ 
by precomposition with the inverse.
Suppose that $\{\alpha_1,\alpha_2,\ldots,\alpha_n\}$ is a generating set of $\mathbb F_n$.
Artin's representation embeds the braid group $B_n$ on $n$ strands as a subgroup of $\Aut(\mathbb F_n)$.
The subset $\Hom(\mathbb F_n/\langle \alpha_1\cdots\alpha_n\rangle, \SL(2,\mathbb Z))$ of $\Hom(\mathbb F_n, \SL(2,\mathbb Z))$ is $B_n$-invariant and identified with $\Hom(\pi_1(S^2\setminus\mathcal{B}), \SL(2,\mathbb{Z}))$, which inherits the action of $B_n$.
We then consider the orbit space
\begin{equation*}
    \mathcal{M}_n=
    B_n\backslash \Hom(\mathbb{F}_{n-1},\SL(2,\mathbb{Z})) / \SL(2,\mathbb{Z})
\end{equation*}
where the action of $\SL(2,\mathbb{Z})$ on the right is by conjugation.
Note that the $B_n$ action induces an action of the sphere braid group $B_n(S^2)$ on
$\Hom(\mathbb F_{n-1}, \SL(2,\mathbb Z))/\SL(2,\mathbb Z)$, coming from the natural mapping class group action on this set.
A torus fibration with $n$ branch points determines an element in $\mathcal{M}_n$ and therefore the study of torus fibrations by means of their monodromy addresses two independent questions.

\begin{question}
How does an orbit in $\mathcal{M}_n$ limit the corresponding torus fibration?
\end{question}

\begin{question}
How to characterise or classify the elements in $\mathcal{M}_n$?
\end{question}

A torus \emph{Lefschetz fibration} $f:M^4\rightarrow S^2$ is the simplest torus fibration,
which is a smooth torus fibration and contains only one singularity in each singular fibre, each singularity admitting complex local coordinates $(z_1,z_2)$ compatible with the orientation of $M^4$ such that the fibration is locally given by $f(z_1,z_2)=z_1^2+z_2^2$.
The type of singularities then depends only on the number of branch points, every $\phi_i$ being a positive Dehn twist around some simple loop.

For torus Lefschetz fibrations, answers to both questions are given by Moishezon and Livn\'e.
On the one hand, an orbit in $\mathcal{M}_n$, if it does correspond to a torus Lefschetz fibration, determines the unique one up to fibre-preserving diffeomorphism
% isomorphism of smooth fibrations
%, i.e., pair of orientation-preserving and fiber-preserving diffeomorphisms 
(see Part \URNum{2}, Lemma 7a in \cite{moishezon1977complex}).
On the other hand, for torus Lefschetz fibrations with the same number of branch points, the action of $B_n$ on the set of their monodromy homomorphisms is transitive (see Part \URNum{2}, Lemma 8 in \cite{moishezon1977complex}).
This result was generalized by Orevkov (see \cite{orevkov2004braid}).

If one relaxes the requirement of the orientation for Lefschetz fibrations, the fibrations are \emph{achiral} Lefschetz fibrations.
We say that the orientation is still preserved for a type $I_1^+$ singular fibre but not for a type $I_1^-$ singular fibre.
The global monodromy was first investigated by Matsumoto in \cite{matsumoto1985torus} (see also \cite[Section 8.4]{gompf19994}).
An inspirational result in his study introduces a representative of the global monodromy using elementary transformations which is, however, not unique.
In particular, one cannot readily classify those achiral Lefschetz fibrations (or their corresponding elements in $\mathcal{M}_n$) whose singular fibres of type $I_1^+$ and $I_1^-$ occur in pairs.

In general, it is extremely difficult to classify the orbits in $\mathcal{M}_n$. 
An algebraic understanding of $\mathcal{M}_n$ is related the study of Wiegold (see \cite{Lubotzky2011}) who conjectured that
\begin{equation*}
    |\Out(\mathbb{F}_{n-1})  \backslash \Epi(\mathbb{F}_{n-1},G) / \Aut(G)| = 1
\end{equation*}
for any finite simple group $G$ and $n\ge 4$, where $\Epi(\mathbb{F}_{n-1},G)$ denotes the set of epimorphisms $\mathbb{F}_{n-1}\rightarrow G$. For the study of its extension to surface groups, we refer to \cite[Theorem 1.4]{Funar2018}.

As in \cites{auroux2005stable}{Catanese2015}{Samperton_2020}, we discuss the stable equivalence of algebraic objects by relating them to the direct sum construction.
When the $2$-sphere is replaced by an arbitrary surface, another notion of stabilisation corresponds to pinching a hole (see \cites{Catanese2016}{funar2020braidedsurfaces}). For more interesting problems on the orbit space $\mathcal{M}_n$ and its variations,
not related to the concept of stabilisation,
we refer to \cites{auroux2006problems}{auroux2015factorizations}.

\subsection*{Global monodromies with stabilisation}

Suppose that $f_1:M_1\rightarrow S^2$ and $f_2:M_2\rightarrow S^2$ are two torus fibrations.
Choosing a pair of $2$-disks $D_1, D_2\subset S^2$ that do not contain any branch points of $f_1$, $f_2$ respectively, 
gluing $M_1\setminus f_1^{-1}(D_1)$ and $M_2\setminus f_2^{-1}(D_2)$ along 
some orientation reversing fibrewise homeomorphism
%diffeomorphism
$\beta:\partial f_1^{-1}(D_1)\rightarrow \partial f_2^{-1}(D_2)$, we obtain a fibre-connected sum $M_1\oplus_{\beta} M_2$ between $M_1$ and $M_2$. The fibration $f$ of $M_1\oplus_{\beta} M_2$ piecing together $f_1$ and $f_2$ is again a torus fibration over $S^2$, called a direct sum between $f_1$ and $f_2$ and written as $f=f_1\oplus f_2$ if we do not specify $\beta$.
In \cite{auroux2005stable} Auroux introduced the direct sum between a fibration and a fixed standard fibration, called stabilisation.
He then proceeded to give a classification of genus $g\ge 2$ Lefschetz fibrations, up to stabilisation.

\begin{defx*}
A conjugacy class of $\SL(2,\mathbb{Z})$ which either corresponds to elements of trace $0,\pm 1,\pm 3$ or else contains 
$\begin{bmatrix}1 & 1\\0 & 1\end{bmatrix}$,
$\begin{bmatrix}1 & -1\\0 & 1\end{bmatrix}$,
$\begin{bmatrix}-1 & 1\\0 & -1\end{bmatrix}$ or
$\begin{bmatrix}-1 & -1\\0 & -1\end{bmatrix}$
is called \emph{simple}.
\end{defx*}

The following result is a rather general extension of Auroux's stable classification in genus $1$ but for arbitrary singularities:

\begin{thmx}
\label{thmx::A}
Let $\mathcal{O}$ be a multi-set of conjugacy classes of $\SL(2,\mathbb{Z})$.
There exists a torus Lefschetz fibration $f_O^L$ over $S^2$ depending only on the non-simple conjugacy classes occurring in $\mathcal{O}$ that has the following property:
for $i=1,2$,
\begin{itemize}[-,topsep=0pt,itemsep=-1ex,partopsep=1ex,parsep=1ex]
    \item let $f_i$ be a torus fibration over $S^2$ with $\mathcal{O}(f_i)=\mathcal{O}$;
    \item let $\widetilde{f_i}$ be a direct sum between $f_i$ and $f_{\mathcal{O}}^L$;
    \item let $(g_1^{(i)},\ldots,g_n^{(i)})$ be a global monodromy of $\widetilde{f_i}$.
\end{itemize}
Then $(g_1^{(1)},\ldots,g_n^{(1)})$ and $(g_1^{(2)},\ldots,g_n^{(2)})$ are Hurwitz equivalent i.e. one can transform $(g_1^{(1)},\ldots,g_n^{(1)})$ into $(g_1^{(2)},\ldots,g_n^{(2)})$ using a finite sequence of elementary transformations.
\end{thmx}

In Theorem \ref{thmx::A}, the choices of direct sums $\widetilde{f_1}$, $\widetilde{f_2}$, base points and loops for the global monodromies are far from unique.
As such, we adopt the following convention: we will use the double plural to highlight the unlimited objects, say all global monodromies of all direct sums.

Theorem \ref{thmx::A} shows that, in particular, given a torus fibration $f$ over $S^2$, all global monodromies of all direct sums $f\oplus f_{\mathcal{O}(f)}^L$ are pairwise Hurwitz equivalent.
The additional fibration $f_{\mathcal{O}}^L$ in Theorem \ref{thmx::A} can be replaced by a torus fibration with fewer branch points but which is not a Lefschetz fibration (see Theorem \ref{theorem::global monodromy fibre sum Hurwitz equivalent} for a more detailed reformulation). In both cases, the number of branch points in the additional fibration depends on the number of non-simple elements in $\mathcal{O}$. In particular, we have the following results:

\begin{thmx}
\label{thmx::B}
There exists a torus Lefschetz fibration $f_{12}^L$ over $S^2$ with $12$ branch points such that, for any multi-set $\mathcal{O}$ of simple conjugacy classes of $\SL(2,\mathbb{Z})$ corresponding to elements of trace $0,\pm 1$ or $\pm 2$,
all global monodromies of all direct sums $f\oplus f_{12}^L$ with $f$ a torus fibration over $S^2$ satisfying $\mathcal{O}(f)=\mathcal{O}$ are pairwise Hurwitz equivalent.
\end{thmx}

\begin{thmx}
\label{thmx::C}
There exists a torus Lefschetz fibration $f_{60}^L$ over $S^2$ with $60$ branch points such that, for any multi-set $\mathcal{O}$ of simple conjugacy classes of $\SL(2,\mathbb{Z})$,
all global monodromies of all direct sums $f\oplus f_{60}^L$ with $f$ a torus fibration over $S^2$ satisfying $\mathcal{O}(f)=\mathcal{O}$ are pairwise Hurwitz equivalent.
\end{thmx}

In Theorem \ref{thmx::B} each of $-2,-1,0,1$ and $2$ might occur as the trace of some element in $\mathcal{O}$. We emphasise that the ``or'' is always inclusive in this paper.
The fibration $f_{12}^L$ in Theorem \ref{thmx::B} can be replaced by a non-Lefschetz fibration with only $6$ branch points and the fibration $f_{60}^L$ in Theorem \ref{thmx::C} can be replaced by a fibration with only $19$ branch points.

The stated Hurwitz equivalence in Theorem \ref{thmx::A}, Theorem \ref{thmx::B} and Theorem \ref{thmx::C}
is obtained with a specific normal form (see Theorem \ref{theorem::global monodromy fibre sum Hurwitz equivalent}) which satisfies a remarkable property, called swappability (see Subsection \ref{subsection::swappability}).
The normal form is computable:
one can compute the finite sequence of elementary transformations with algorithms (see Appendix \ref{section::computability}).

The Hurwitz equivalence fails without stabilisation or with an unreasonable stabilisation, see Subsection \ref{subsection::fails}.
The following theorem compares the (unstable) Hurwitz equivalence and the stable equivalence
between global monodromies of torus achiral Lefschetz fibrations.

\begin{thmx}
\label{thmx::achiralLef}
For torus achiral Lefshetz fibrations with $p$ singular fibres of type $I_1^+$ and $q$ singular fibres of type $I_1^-$,
we have the following statements.
\begin{enumerate}
[(i),topsep=0pt,itemsep=-1ex,partopsep=1ex,parsep=1ex]
\item After performing direct sums with $f_{12}^L$,
all global monodromies are Hurwitz equivalent.
\item If $p\neq q$, then all global monodromies are Hurwitz equivalent.
\item If $p=q\ge 1$, then the global monodromies have infinitely many Hurwitz equivalent classes
and there exists an explicit combinatorial classification.
\end{enumerate}
\end{thmx}

As a consequence of Theorem \ref{thmx::B}, we partially extend Kas' classification of elliptic surfaces up to diffeomorphism \cite{kas1977deformation} to a stable classification of their global monodromies.
\emph{Elliptic surfaces} over $\mathbb{CP}^1$ are proper holomorphic maps $f:S\rightarrow \mathbb{CP}^1$ between a complex surface $S$ and $\mathbb{CP}^1$ such that the generic fibre is an elliptic curve.
An elliptic surface is certainly a torus fibration whose singular fibres were classified by Kodaira in \cites{Kodaira1964}{Kodaira1966}; the fibre monodromies are described in \cite{miranda1989basic}.

\begin{colx}
\label{corollary::elliptic surfaces}
Let $f_1:S_1\rightarrow\mathbb{CP}^1$ and $f_2:S_2\rightarrow\mathbb{CP}^1$ be elliptic surfaces without multiple singular fibres, without singular fibres of type $\URNum{1}_v$ or $\URNum{1}_v^*$, $v\ge 2$ in Kodaira's classification.
Suppose that $\mathcal{O}(f_1)=\mathcal{O}(f_2)$.
Then, all global monodromies of all direct sums $f_1\oplus f_{12}^L$ and $f_2\oplus f_{12}^L$ are pairwise Hurwitz equivalent.
\end{colx}

%%%%%%%%%%%%%%%%%%%%%%%%%%%%%%
%%%%%%%%%%%%%%%%%%%%%%%%%%%%%%

\subsection*{Fibre-preserving homeomorphisms}

An element in $\mathcal{M}_n$ does not provide all the data about the fibration.
In most cases, a torus fibration cannot be determined by its monodromy in any way.
Additional restrictions and data on the local models at singularities are essential.

One remarkable encoding for the local model comes from King’s classification in \cite{king1978topological,king1997topology} of isolated singularities and the local study of singularities by Church and
Timourian in \cite{church1972differentiable,church1974differentiable}, using this we study the so-called \emph{singular fibrations}.

Roughly speaking by singular
fibration we mean a smooth fibration with only finitely many singularities each having a ``nice'' neighbourhood
(see Subsection \ref{subsection::singular fibration} for a precise definition).
Each singularity is then characterised by a local Milnor fibre which is a sub-surface of the generic fibre, a binding link K and an open book decomposition.
Singular fibrations have been studied in \cite{looijenga1971note,funar2011global,funar2022singular}. The local properties of their singularities are related to the corresponding fibred knots (see e.g. \cite{burde2008knots}).

As an improvement of Proposition 2.1 in \cite{funar2022singular} as well as a consequence of Theorem \ref{thmx::C} and the swappability of the corresponding normal form, we have the following stable classification of singular fibrations based on the type of singularities up to fibre-preserving homeomorphism:

\begin{colx}
\label{corollary::main_result_singular_fibrations}
Let $f_1:M_1\rightarrow S^2$ and $f_2:M_2\rightarrow S^2$ be torus singular fibrations with a single singularity in each singular fibre
%, with surjective monodromy homomorphisms
and with $\mathcal{O}(f_1)=\mathcal{O}(f_2)$.
Suppose that each local Milnor fibre of singularities in $f_1$ and $f_2$ is either
\begin{itemize}[-,topsep=0pt,itemsep=-1ex,partopsep=1ex,parsep=1ex]
\item a surface of genus $0$ with $\le 2$ boundary components, or
\item a surface of genus $1$ with only $1$ boundary component.
\end{itemize}
Let $\widetilde{f_1}=f_1\oplus f_{60}^L: \widetilde{M_1}\rightarrow S^2$ and
$\widetilde{f_2}=f_2\oplus f_{60}^L: \widetilde{M_2}\rightarrow S^2$ be direct sums.
Then $(\widetilde{M_1},\widetilde{f_1})$ and $(\widetilde{M_2},\widetilde{f_2})$ are fibre-preserving homeomorphic.
\end{colx}

Throughout this paper, all fibrations which we consider will be over $S^2$,
unless otherwise stated.

\subsection*{Acknowledgements}
This work is part of my PhD thesis.
I would like to thank my advisors, Louis Funar and Greg McShane, for their helpful discussions and guidance. I would also like to thank Stepan Orevkov for useful suggestions.%, especially on Subsection \ref{subsection::contractions-restorations}.

%% file: Section2/elem_trans.tex
\section{Connected sums and Hurwitz equivalence}

\subsection{Elementary transformations}
\label{subsection::elementary transformations}
We first define the elementary transformations.
Throughout this subsection, $G$ is an arbitrary group and $Z(G)$ is the center of $G$.
An $n$-tuple in $G$ is a sequence $(g_1,\ldots,g_n)$ of elements in $G$, each $g_i$ is called a component of the tuple.
Let $\mathcal{T}_{G,n}$ be the set of $n$-tuples $(g_1,\ldots,g_n)$ in $G$ satisfying $g_1\cdots g_n\in Z(G)$.

\begin{definition}
For $1\le i\le n-1$, the \emph{elementary transformations} (or \emph{Hurwitz moves}) $R_i$ is a bijection on the set of $n$-tuples in $G$ defined by:
\begin{equation*}
R_i (g_1,\ldots,g_n)
=
(g_1,\ldots,g_{i-1},g_{i+1},g_{i+1}^{-1}g_ig_{i+1},g_{i+2},\ldots,g_n).    
\end{equation*}
Both $R_i$ and its inverse $R_i^{-1}$ are elementary transformations.
A pair of tuples $(g_1,\ldots,g_n)$ and $(h_1,\ldots,h_n)$ which can be transformed into each other by a finite sequence of elementary transformations are called \emph{Hurwitz equivalent}, written as:
\begin{equation*}
    (g_1,\ldots,g_n) \sim (h_1,\ldots,h_n).
\end{equation*}
\end{definition}

We emphasise that the set of all $n$-tuples in $G$ can also be interpreted as $\Hom(\mathbb{F}_n,G)$ and the subset $\mathcal{T}_{G,n}$ is invariant under the elementary transformations.

\begin{lemma}
\label{lemma::cyclic shift} For $(g_1,\ldots,g_n)\in \mathcal{T}_{G,n}$ and any $1\le k\le n$, the tuple $(g_1,\ldots,g_n)$ is Hurwitz equivalent to $(g_k,g_{k+1},\ldots,g_n,g_1,g_2,\ldots,g_{k-1})$.
\end{lemma}
\begin{proof}
Applying $R_{n-1}\circ\ldots\circ R_1$ on the $n$-tuple $(g_1,\ldots,g_n)$ we get $(g_2,\ldots,g_n,g_1)$.
\end{proof}

Let $\bullet$ denote the concatenation of tuples: $(g_1,\ldots,g_n)\bullet (h_1,\ldots,h_m)=(g_1,\ldots,g_n,h_1,\ldots,h_m)$.
The power of a tuple corresponds to a repeated concatenation with itself.
The symbol $\prod$ represents the concatenation of a family of tuples.

\begin{lemma}
\label{lemma::move subtuples}
Let $(g_1,\ldots,g_n,h_1,\ldots,h_m,g_{n+1},\ldots,g_{n+n'})$ be an $(n+m+n')$-tuple in $G$ satisfying $h_i\cdots h_m\in Z(G)$. For $0\le k\le n+n'$, this $(n+m+n')$-tuple is Hurwitz equivalent to
\begin{equation}
    (g_1,\ldots,g_k,h_1,\ldots,h_m,g_{k+1},\ldots,g_{n+n'}).
\end{equation}

In particular, let $(g_{1,1},\ldots,g_{1,n_1})$, $(g_{2,1},\ldots,g_{2,n_2})$, \ldots, $(g_{k,1},\ldots,g_{k,n_k})$ be tuples in $G$ satisfying $g_{j,1}\cdots g_{j,n_j}\in Z(G)$ for each of $j=1,\ldots,k$. Then their concatenations in any order are pairwise Hurwitz equivalent.
\end{lemma}
\begin{proof}
Applying $R_{n+m-1}\circ\ldots\circ R_{n}$ if $n'>0$ and applying $R_{n+1}^{-1}\circ\ldots\circ R_{n+m}^{-1}$ if $n>0$ on the $(n+m+n')$-tuple we transform the tuple into
$(g_1,\ldots,g_{n+1},h_1,\ldots,h_m,g_{n+2},\ldots,g_{n+n'})$
and
$(g_1,\ldots,g_{n-1},h_1,\ldots,h_m,g_{n},\ldots,g_{n+n'})$
respectively.
\end{proof}

\begin{lemma}
\label{lemma::switch places of components with the same value}
Let $(g_1,\ldots,g_n)$ be an $n$-tuple in $G$ satisfying $g_i=g_jh$ with some $1\le i<j\le n$ and $h$ in $Z(G)$. Then
$(g_1,\ldots,g_n) \sim (g_1,\ldots,g_{i-1},g_j,g_{i+1},\ldots,g_{j-1},g_i,g_{j+1},\ldots,g_n)$.

\end{lemma}
\begin{proof}
Applying $R_{j-1}^{-1}\circ\ldots\circ R_{i+1}^{-1}\circ R_i \ldots\circ R_{j-1}$ on $(g_1,\ldots,g_n)$ we get the tuple
\[
(g_1,\ldots,g_{i-1},g_j,g_ig_j^{-1}g_{i+1}g_jg_i^{-1},\ldots,g_ig_j^{-1}g_{j-1}g_jg_i^{-1},g_i,g_{j+1},\ldots,g_n),
\]
which is equal to $(g_1,\ldots,g_{i-1},g_j,g_{i+1},\ldots,g_{j-1},g_i,g_{j+1},\ldots,g_n)$, as desired.
\end{proof}

\begin{definition}
An $n$-tuple in $G$ is said to \emph{contain a generating set} if its components form a generating set of the group $G$.
\end{definition}

For instance, the modular group $\PSL(2,\mathbb{Z})=\SL(2,\mathbb{Z}) / \{+I,-I\}$ has the presentation
\begin{equation*}
    \PSL(2,\mathbb{Z})=\langle a,b\mid a^3=b^2=1\rangle;
\end{equation*}
both $(a^2b,ba^2,a^2b,ba^2,a^2b,ba^2)$ and $(ba,ab,ba,ab,ba,ab)$ contain generating sets.

\begin{lemma}
\label{lemma::tuples containing a generating set}
Suppose that $(g_1,\ldots,g_n)$ and $(h_1,\ldots,h_m)$ are tuples in $G$ such that
$(h_1,\ldots,h_m)$ contains a generating set. Let $Q$ be an arbitrary element in $G$.
If there exists a sub-tuple of $(g_1,\ldots,g_n)$, say $(g_l,\ldots,g_r)$ with $1\le l\le r\le n$, such that $\prod_{i=l}^r g_i\in Z(G)$, then the concatenation $(g_1,\ldots,g_n)\bullet (h_1,\ldots,h_m)$ is Hurwitz equivalent to
\[
    (g_1,\ldots,g_{l-1},Q^{-1}g_lQ,\ldots,Q^{-1}g_rQ,g_{r+1},\ldots,g_n)\bullet (h_1,\ldots,h_m).
\]
\end{lemma}
\begin{proof}
We express a given element $Q$ in $G$ as $q_1\cdots q_u$ such that $q_i\in \{h_1,h_1^{-1},\ldots,h_m,h_m^{-1}\}$, $i=1,\ldots,u$. The lemma follows from Lemma \ref{lemma::move subtuples} and the following substitutions via elementary transformations for each of $j=1,\ldots,m$:
\begin{align*}
&(g_1,\ldots,g_l,\ldots,g_r,\ldots,g_n)\bullet(h_1,\ldots,h_j,\ldots,h_m)\\
&\rightarrow
(g_1,\ldots,g_{l-1},g_{r+1},\ldots,g_n,h_1,\ldots,h_{j-1},g_l,\ldots,g_r,h_j,\ldots,h_m)\\
&\rightarrow
(g_1,\ldots,g_{l-1},g_{r+1},\ldots,g_n,h_1,\ldots,h_j,h_j^{-1}g_lh_j,\ldots,h_j^{-1}g_rh_j,h_{j+1},\ldots,h_m)\\
&\rightarrow
(g_1,\ldots,g_{l-1},h_j^{-1}g_lh_j,\ldots,h_j^{-1}g_rh_j,g_{r+1},\ldots,g_n)\bullet(h_1,\ldots,h_j,\ldots,h_m);\\
&(g_1,\ldots,g_l,\ldots,g_r,\ldots,g_n)\bullet(h_1,\ldots,h_j,\ldots,h_m)\\
&\rightarrow
(g_1,\ldots,g_{l-1},g_{r+1},\ldots,g_n,h_1,\ldots,h_{j},g_l,\ldots,g_r,h_{j+1},\ldots,h_m)\\
&\rightarrow
(g_1,\ldots,g_{l-1},g_{r+1},\ldots,g_n,h_1,\ldots,h_{j-1},h_jg_lh_j^{-1},\ldots,h_jg_rh_j^{-1},h_j,\ldots,h_m)\\
&\rightarrow
(g_1,\ldots,g_{l-1},h_jg_lh_j^{-1},\ldots,h_jg_rh_j^{-1},g_{r+1},\ldots,g_n)\bullet(h_1,\ldots,h_j,\ldots,h_m).
\end{align*}
\end{proof}

%% file: Section2/contraction.tex
\subsection{Contraction and restoration on tuple}
\label{subsection::contractions-restorations}
%% tech

In this subsection, we introduce the notions of \emph{contraction} and \emph{restoration} on tuples.
We move on to a procedure that involves a series of operations, including contractions, restorations, and elementary transformations.
The procedure behaves like a self-consistent machine, maintaining data about the given tuple and operations.
Our study repeatedly utilises this procedure.
To make it clear and easy to visualise, thus we start with the following definition.

\begin{definition}
An \emph{iterated tuple} of height $0$ in $G$ is an element $g\in G$;
for $h\ge 1$, an iterated tuple of height $h$ in $G$ is a tuple whose components are iterated tuples of height smaller than $h$ such that at least one component is of height $h-1$.

Take $g\in G$ and $H=(H_1,\ldots,H_n)$ an iterated tuple of height $h\ge 1$.
The \emph{evaluation} on an iterated tuple is defined by $ev(g)=g$ and $ev(H)=\prod_{i=1}^n ev(H_i)$.
With $v\in G$, using the notation $g^v=v^{-1}gv$ we define $H^v$ as
\begin{equation*}
    H^v = (H_1,\ldots,H_n)^v = (H_1^v,\ldots,H_n^v).
\end{equation*}

The elementary transformation $R_i$ acts on the set of iterated tuples with $n\ge i+1$ components by taking the conjugation of each element in $H_i$ with $ev(H_{i+1})$ and swapping the positions, to wit
\begin{equation*}
    R_i(H_1,\ldots,H_n)=(H_1,\ldots,H_{i-1},H_{i+1},H_i^{ev(H_{i+1})},H_{i+2},\ldots,H_n).
\end{equation*}
\end{definition}

Given an $n$-tuple $(g_1,\ldots,g_n)$ in $G$, we keep hold of the following data:
\begin{itemize}[-]
    \item $(h_1,\ldots,h_m)$: an tuple in $G$;
    
    \item $(H_1,\ldots,H_m)$: an iterated tuple in $G$ such that $(ev(H_1),\ldots,ev(H_m))=(h_1,\ldots,h_m)$;
    
    \item $\mathcal{F}$: an ordered list such that each element is either
    \begin{itemize}[-,topsep=0pt,itemsep=-1ex,partopsep=1ex,parsep=1ex]
    \item a pair $(\mu,\sigma)$ with $\mu\in \mathbb{Z}$ and $\sigma$ an elementary transformation on (iterated) $\mu$-tuples, or
    \item a pair of integers $(l,r)$ with $1\le l< r$.
    \end{itemize}
\end{itemize}

At the beginning, both $(h_1,\ldots,h_m)$ and $(H_1,\ldots,H_m)$ are copies of $(g_1,\ldots,g_n)$, the ordered list $\mathcal{F}$ is empty.
We apply the following operations successively on the data.

\begin{enumerate}[(i)]
    \item \textbf{Elementary transformation:} Apply an elementary transformation, say $R_i^{\epsilon}$ with $1\le i\le m-1$ and $\epsilon=\pm 1$, on the $m$-tuple $(h_1,\ldots,h_m)$ and the iterated $m$-tuple $(H_1,\ldots,H_m)$.
    Append $(m,R_i^{\epsilon})$ to $\mathcal{F}$.
    
    \item \textbf{Contraction:}
    For a pair of integers $1\le l< r\le n$, we replace the tuple $(h_1,\ldots,h_m)$ with
    \begin{equation*}
        (h_1,\ldots,h_{l-1},h_l\cdots h_r,h_{r+1},\ldots,h_m)
    \end{equation*}
    and replace the iterated tuple $(H_1,\ldots,H_m)$ with
    \begin{equation*}
        (H_1,\ldots,H_{l-1},(H_l,\ldots,H_r),H_{r+1},\ldots,H_m).
    \end{equation*}
    Append $(l,r)$ to $\mathcal{F}$.
    
    \item \textbf{Restoration:} 
    Take the last pair of the form $(l,r)$ in $\mathcal{F}$, still denoted by $(l,r)$.
    Let $\mathcal{F}'$ be the sub-list of $\mathcal{F}$  which consists of the elements after $(l,r)$.
    Remove $(l,r)$ and all the elements after $(l,r)$ from $\mathcal{F}$.
    
    Set $k=l$ and $m'=m+(r-l)$.
    We consider each pair $(\mu,\sigma)=(m,R_i^{\epsilon})$ in $\mathcal{F}'$ with the order.
    %Let $(\mathcal{H}_1,\ldots,\mathcal{H}_m)$ be an iterated $m$-tuple and $(\mathcal{H}'_1,\ldots,\mathcal{H}'_{r-l+1})$ be an iterated $(r-l)$-tuple such that $\mathcal{H}_k=(\mathcal{H}'_1,\ldots,\mathcal{H}'_{r-l+1})$.
    
    \begin{itemize}[-,topsep=0pt,itemsep=-1ex,partopsep=1ex,parsep=1ex]
    \item If $1\le i\le k-2$, then append $(m',R_i^{\epsilon})$ to $\mathcal{F}$.
    
    \item If $k+1\le i\le m$, then append $(m',R_{i+(r-l)}^{\epsilon})$ to $\mathcal{F}$.
    
    \item If $\sigma=R_{k-1}$, then append the pairs
    $(m',R_{k-1}),\ldots,(m',R_{k-1+(r-l)})$
    to $\mathcal{F}$ and replace $k$ with $k-1$.
    
    In this case, the elementary transformation $\sigma$ acts on an iterated $m$-tuple of the form $(\mathcal{H}_1,\ldots,\mathcal{H}_{k-1},(\mathcal{H}'_1,\ldots,\mathcal{H}'_{r-l+1}),\mathcal{H}_{k+1},\ldots,\mathcal{H}_m)$ via
    \begin{equation*}
        (\ldots,\mathcal{H}_{k-1},(\mathcal{H}'_1,\ldots,\mathcal{H}'_{r-l+1}),\ldots)\xrightarrow{R_{k-1}} (\ldots,(\mathcal{H}'_1,\ldots,\mathcal{H}'_{r-l+1}),\mathcal{H}_{k-1}^{ev(\mathcal{H}'_1,\ldots,\mathcal{H}'_{r-l+1})},\ldots).
    \end{equation*}
    The new pairs
    $(m',R_{k-1}),\ldots,(m',R_{k-1+(r-l)})$ in $\mathcal{F}$
    act on an iterated $m'$-tuple of the form $(\mathcal{H}_1,\ldots,\mathcal{H}_{k-1},\mathcal{H}'_1,\ldots,\mathcal{H}'_{r-l+1},\mathcal{H}_{k+1},\ldots,\mathcal{H}_m)$
    via
    \begin{align*}
        (\ldots,\mathcal{H}_{k-1},\mathcal{H}'_1,\ldots,\mathcal{H}'_{r-l+1},\ldots)
        \xrightarrow{R_{k-1}}&
        (\ldots,\mathcal{H}'_1,\mathcal{H}_{k-1}^{ev(\mathcal{H}'_1)},\mathcal{H}'_2\ldots,\mathcal{H}'_{r-l+1},\ldots)\\
        \xrightarrow{R_{k}}&
        (\ldots,\mathcal{H}'_1,\mathcal{H}'_2,\mathcal{H}_{k-1}^{ev(\mathcal{H}'_1)ev(\mathcal{H}'_2)},\mathcal{H}'_3\ldots,\mathcal{H}'_{r-l+1},\ldots)\\
        \rightarrow& \ldots\\
        \xrightarrow{R_{k-1+(r-k)}}&
        (\ldots,\mathcal{H}'_1,\ldots,\mathcal{H}'_{r-l+1},\mathcal{H}_{k-1}^{ev(\mathcal{H}'_1)\cdots ev(\mathcal{H}'_{r-l+1})},\ldots).
    \end{align*}
    
    \item If $\sigma=R_k$, then append the pairs $(m',R_{k+(r-l)}),\ldots,(m',R_k)$ to $\mathcal{F}$ and replace $k$ with $k+1$.
    
    \item If $\sigma=R_{k-1}^{-1}$, then append the pairs $(m',R_{k-1}^{-1}),\ldots,(m',R_{k-1+(r-l)}^{-1})$ to $\mathcal{F}$ and replace $k$ with $k-1$.
    
    \item If $\sigma=R_k^{-1}$, then append the pairs $(m',R_{k+(r-l)}^{-1}),\ldots,(m',R_k^{-1})$ to $\mathcal{F}$ and replace $k$ with $k+1$.
    \end{itemize}
    
    Finally, suppose that $H_k = (H'_1,\ldots,H'_{r-l+1})$. We replace $(h_1,\ldots,h_m)$ with
    \begin{equation*}
    (h_1,\ldots,h_{k-1},
    ev(H'_1),\ldots,ev(H'_{r-l+1}),
    h_{k+1},\ldots,h_m)
    \end{equation*}
    and replace $(H_1,\ldots,H_m)$ with
    \begin{equation*}
        (H_1,\ldots,H_{k-1},
        H'_1,\ldots,H'_{r-l+1},
        H_{k+1},\ldots,H_m).
    \end{equation*}

\end{enumerate}

Note that above operations can be applied in any order, possibly each appears many times and different operations may alternate with each other. However, we apply operations only finitely many times.
The following lemma shows the main property of these operations.

\begin{lemma}
\label{lemma::technique}
If $m=n$, then $(h_1,\ldots,h_m)$ coincides with the resulting tuple of $(g_1,\ldots,g_n)$ after applying all elementary transformations $\sigma$ occurring in $\mathcal{F}$ with the order.
\end{lemma}

\begin{proof}
Let $(l,r)$ be the last pair of integers in $\mathcal{F}$ which indicates the last contraction operation and replaces $h_l,\ldots,h_r$ with $h_l\cdots h_r$. The product exactly corresponds to the $k$-th component of the $m$-tuple after each of the subsequent elementary transformations, where $k$ is introduced in the restoration operation.
Therefore, the restoration cancels the contraction and constructs the corresponding elementary transformations on the $m'$-tuple.
We conclude the lemma by induction.
\end{proof}

A direct application of the above operations requires us to maintain a lot of data, which would be a massive and tedious project. To simplify the application, our usage only focuses on the replacement
\begin{equation*}
(h_1,\ldots,h_m)
\dashrightarrow
(h_1,\ldots,h_{l-1},
h_l\cdots h_r,
h_{r+1},\ldots,
h_m)
\end{equation*}
of the contraction;
when applying the restoration, we enumerate all possible patterns of the corresponding contraction instead.
Therefore, the iterated tuple $(H_1,\ldots,H_m)$ and the ordered list $\mathcal{F}$ never appear in the argument.

%\textit{Given an $n$-tuple $(g_1,\ldots,g_n)$ in $G$, we sometimes glue adjacent components $g_l,\ldots,g_r$ together, with $1\le l\le r\le n$, and replace them by their product. Then the tuple $(g_1,\ldots,g_n)$ is compressed into an $(n-r+l)$-tuple}
%\begin{equation*}
%    (g_1,\ldots,g_{l-1},g_l\cdots g_r,g_{r+1},\ldots,g_n).
%\end{equation*}
%\textit{The tuple would be further processed by elementary transformations, while the product $g_l\cdots g_r$ would be sent to a conjugate $Q^{-1}g_l\cdots g_rQ$ with $Q\in G$. We always memorise what happened on $g_l\cdots g_r$ and, once we are happy, rewrite the product as adjacent components $Q^{-1}g_lQ,\ldots,Q^{-1}g_rQ$. Naturally, the elementary transformations applied on the $(n-r+l)$-tuple correspond to some elementary transformations applied on the original $n$-tuple which always adjust the $r-l+1$ adjacent components along the same way in quick succession.}

More delicate operations for tuples in the modular group and their properties will be introduced in Proposition \ref{proposition::reduce_and_reloading} and Proposition \ref{proposition::reduce_and_reloading_hard}.
We need the following definition in the sequel:

\begin{definition}
Let $(g_1,\ldots,g_n)$ and $(h_1,\ldots,h_m)$ be tuples in $G$ with $n\ge m$. The tuple $(g_1,\ldots,g_n)$ is said to be an \emph{$(h_1,\ldots,h_m)$-expansion} (or an \emph{expansion} of $(h_1,\ldots,h_m)$) if there exist integers $0=i_0<i_1<i_2<\ldots<l_m=n$ such that
$g_{i_{j-1}+1}\cdots g_{i_j}=h_j$ for each of $j=1,\ldots,m$.
\end{definition}

Suppose that $(g_1,\ldots,g_n)$ is an expansion of $(h_1,\ldots,h_m)$. Then the associated contraction operations consist of $m$ contractions that replace $(g_1,\ldots,g_n)$ with $(h_1,\ldots,h_m)$.

%%%%

%% file: Section2/direct_sum.tex
\subsection{Direct sums of fibrations and their global monodromies}

Recall that the \emph{type} $\mathcal{O}(f)$ of singularities of a torus fibration $f$ is a multi-set of fibre monodromies counted with multiplicity.
% TODO: Is the following is true? Why?
%Two global monodromies of $f$ with the same base point can be obtained from each other by a sequence of elementary transformations.
Let $f_1:M_1\rightarrow S^2$, $f_2:M_2\rightarrow S^2$ be torus fibrations,
possibly with different numbers of singular fibres. Let $f_1\oplus f_2$ be a direct sum of $f_1$ and $f_2$.
The global monodromy of $f_1\oplus f_2$ depends on the 
fibre-connected sum $M_1\oplus_{\beta} M_2$, the base point $p$ on $S^2$ and the set of generators for the fundamental group $\pi_1(S^2\setminus B)$.
To be precise, a global monodromy of $f_1\oplus f_2$ is a concatenation of two sub-tuples, say
\[
(\psi_1^{-1}\phi_{1,1}\psi_1,\ldots,\psi_1^{-1}\phi_{1,n_1}\psi_1)
\bullet
(\psi_2^{-1}\phi_{2,1}\psi_2,\ldots,\psi_2^{-1}\phi_{2,n_2}\psi_2),
\]
such that $(\phi_{1,1},\ldots,\phi_{1,n_1})$ and $(\phi_{2,1},\ldots,\phi_{2,n_2})$ are global monodromies of $f_1$ and $f_2$ respectively, $\psi_1,\psi_2\in \SL(2,\mathbb Z)$ and at least one of $\psi_1,\psi_2$ is $1$.
In general, global monodromies of different direct sums or of the same direct sum but with different base points are not Hurwitz equivalent.
%For convenience, we will not specify $\beta$ in the notation for a direct sum of fibrations.

For any $n$-tuple $(\phi_1,\ldots,\phi_n)$ in $\SL(2,\mathbb Z)$ with $\phi_1\cdots\phi_n=1$, we use $f_{(\phi_1,\ldots,\phi_n)}$ to denote a torus fibration that has a global monodromy equal to $(\phi_1,\ldots,\phi_n)$, if it exists.
We use the notation $f_{(\phi_1,\ldots,\phi_n)}^L$ for such a fibration that is also a Lefschetz fibration.
Lemma \ref{lemma::contruct_lefschetz_fibrations} will point out that we can always work with such a Lefschetz fibration up to expansion.
Let us first recall some facts about $\SL(2,\mathbb{Z})$ and Lefschetz fibrations.

Set $A=\begin{bmatrix}0 & -1\\1 & 1\end{bmatrix}$ and
$B=\begin{bmatrix}-1 & -2\\1 & 1\end{bmatrix}\in \SL(2,\mathbb{Z})$.
Let $L=-ABA=\begin{bmatrix}1 & 0\\1 & 1\end{bmatrix}$ and $R=-AB=\begin{bmatrix}1 & 1\\0 & 1\end{bmatrix}$.
The conjugacy classes of $\SL(2,\mathbb{Z})$ have been described using the geometry of continued fractions (see \cite{series1985geometry,karpenkov2013geometry,mathoverflow236162}).
They are classified according to the trace, which is conjugacy invariant, as follows.
\begin{enumerate}[(0)]
    \item For trace $0$, there are two conjugacy classes represented by $B$ and $-B$.
\end{enumerate}
For nonzero trace, the conjugacy classes come in opposite pairs, represented by a matrix $M$ and its opposite $-M$ with $tr(M)>0$ and $tr(-M)<0$.
\begin{enumerate}[(1)]
    \item For trace $1$, there are two conjugacy classes represented by $A$ and $-A^2$.\\
    For trace $-1$, there are two conjugacy classes represented by $-A$ and $A^2$.
    
    \item For trace $2$, there is a $\mathbb{Z}$-indexed families of conjugacy classes represented by $L^r$ with $r\in \mathbb{Z}$.\\
    For trace $-2$, there is a $\mathbb{Z}$-indexed families of conjugacy classes represented by $-L^r$ with $r\in \mathbb{Z}$.
    
    \item For trace $3$, there is only one conjugacy class represented by $LR$.\\
    For trace $-3$, there is only one conjugacy class represented by $-LR$.
    
    \item[($\ge 3$)] In general, for trace of absolute value $\ge 3$, the words of the form $\pm R^{j_1}L^{k_1}R^{j_2}L^{k_2}\cdots R^{j_m}L^{k_m}$ with $m\ge 1$, $j_1,\ldots,j_m,k_1,\ldots,k_m\ge 1$ represent all conjugacy classes.
    Conversely, different words of this form up to cyclic conjugacy belong to different conjugacy classes.
\end{enumerate}

%Let $f:M^4\rightarrow S^2$ be a torus fibration of connected compact oriented differential $4$-manifolds, each of whose singular fibres contains a single singularity. Suppose that each singularity admits a complex local coordinates $z_1$, $z_2$ and $\lambda$ with the same orientations as global orientations of $M^4$ and $S^2$ such that the fibration is locally given by $\lambda=z_1^2+z_2^2$. Such a fibration is called a Lefschetz fibration.
Recall that the fibre monodromies of torus Lefschetz fibrations are conjugates of $L$.

For convenience, we set $\widecheck{L}=-A^2B$ and $\widehat{L}=-BA^2$, which are conjugates of $L$.

\begin{lemma}
\label{lemma::contruct_lefschetz_fibrations}
Let $(\phi_1,\ldots,\phi_n)$ be an $n$-tuple in $\SL(2,\mathbb Z)$ with $\phi_1\cdots\phi_n=1$. There exists a torus Lefschetz fibration $f^L$, one of whose global monodromy is an expansion of $(\phi_1,\ldots,\phi_n)$.
\end{lemma}
\begin{proof}
It suffices to show that the semigroup generated by $\widecheck{L}$ and $\widehat{L}$ is exactly $\SL(2,\mathbb{Z})$. It follows from that $\widecheck{L}\widehat{L}=A$ and $\widecheck{L}\widehat{L}\widecheck{L}=B$ whose inverses are $A^5$ and $B^3$ respectively.
\end{proof}

Now we describe the following tuples with respect to a multi-set $\mathcal{O}$ of fibre monodromies. Their induced fibrations $f_{T_{\mathcal{O},0}}$, $f_{T_{\mathcal{O},1}}$, $f_{T_{\mathcal{O},2}}$ and $f_{T_{\mathcal{O}}}$ will stabilise torus fibrations.

\begin{definition}
\label{definition::additional tuple}
Suppose that $\mathcal{O}$ is a multi-set of conjugacy classes of $\SL(2,\mathbb{Z})$.
\begin{enumerate}[1)]
    \item We define $T_{\mathcal{O},0}$ as $(\widecheck{L},\widehat{L},A^2,\widecheck{L},\widehat{L},A^2)$.
    
    \item We define $T_{\mathcal{O},1}$ as an empty tuple if there does not exist a conjugacy class of trace $\pm 3$ in $\mathcal{O}$, otherwise
    \begin{equation*}
    T_{\mathcal{O},1}=(B,B,B,B)\bullet (-A^2BAB,BA,-ABA)^3.
    \end{equation*}
    
    \item We define $T_{\mathcal{O},2}$ as the concatenation of the following tuples.
    \begin{enumerate}[i),topsep=0pt,itemsep=-1ex,partopsep=1ex,parsep=1ex]
        \item If the conjugacy class represented by $\epsilon L^r$ with $r\ge 2$ and $\epsilon\in\{1,-1\}$ occurs $m\ge 1$ times in $\mathcal{O}$, take $m$ copies of
        \begin{equation*}
        \underbrace{(L,\ldots,L,L^{-r})}_{r+1 \text{ components}}.
        \end{equation*}
        
        \item If the conjugacy class represented by $\epsilon R^2$ with $r\ge 2$ and $\epsilon\in\{1,-1\}$ occurs $m\ge 1$ times in $\mathcal{O}$, take $m$ copies of
        \begin{equation*}
        \underbrace{(R,\ldots,R,R^{-r})}_{r+1 \text{ components}}.
        \end{equation*}
        
        \item
        Suppose that a conjugacy class of elements with $|trace|\ge 4$ is represented by
        \[
        \epsilon R^{j_1}L^{k_1}R^{j_2}L^{k_2}\cdots R^{j_m}L^{k_m}
        \]
        with $\epsilon=\{1,-1\}$, $m\ge 1$, $j_1,\ldots,j_m$, $k_1,\ldots,k_m\ge 1$.
        If the conjugacy class occurs $m\ge 1$ times in $\mathcal{O}$, take $m$ copies of
        \begin{equation*}
        \bigl(
        \underbrace{R,\ldots,R}_{j_1 \text{ components}},
        \overbrace{L,\ldots,L}^{k_1 \text{ components}},
        \ldots,
        \underbrace{R,\ldots,R}_{j_m \text{ components}},
        \overbrace{L,\ldots,L}^{k_m \text{ components}},
        (R^{j_1}L^{k_1}\cdots R^{j_m}L^{k_m})^{-1}
        \bigr).
        \end{equation*}
    \end{enumerate}
    \end{enumerate}
Eventually, we define $T_{\mathcal{O}}$ as $T_{\mathcal{O},0}\bullet T_{\mathcal{O},1}\bullet T_{\mathcal{O},2}$.
\end{definition}

%% file: Section2/equivalence.tex
\subsection{Hurwitz equivalence of global monodromies}
\label{subsection::Hurwitz equivalence of global monodromies}

If two global monodromies of torus fibrations are Hurwitz equivalent, then they must have the same number of branch points and the same type of singularities.
The following theorem shows that the global monodromies of torus fibrations with the same type of singularities become Hurwitz equivalent up to fibre sum stabilisations.

\begin{theorem}
\label{theorem::global monodromy fibre sum Hurwitz equivalent}
Given a torus fibration, let $\mathcal{O}$ be the type of singularities.
Suppose that $f_0$ is one of the following:
\begin{enumerate}[i),topsep=0pt,itemsep=-1ex,partopsep=1ex,parsep=1ex]
    \item a torus fibration, one of whose global monodromy is $(h_1,\ldots,h_m)=T_{\mathcal{O}}$;
    \item a torus Lefschetz fibration, one of whose global monodromy $(h_1,\ldots,h_m)$ is a $T_{\mathcal{O}}$-expansion.
\end{enumerate}
Then all global monodromies of all direct sums $f\oplus f_0$ are Hurwitz equivalent for all torus fibrations $f$ with $\mathcal{O}(f)=\mathcal{O}$.
Moreover, these global monodromies have a specific normal form determined by $\mathcal{O}$ and $(h_1,\ldots,h_m)$ as follows:
\begin{equation*}
(g_1,\ldots,g_l)
\bullet
\prod_i (\phi_{i,1},\ldots,\phi_{i,{n_i}})
\end{equation*}
where
$g_1\cdots g_l=I$,
$(g_1,\ldots,g_l)$ is the sub-tuple of $(h_1,\ldots,h_m)$ either equal to $T_{\mathcal{O},0}$
or corresponding to $T_{\mathcal{O},0}$,
$\phi_{i,1}\cdots \phi_{i,n_i} = \pm I$ for each $i$
and
each 
$(\phi_{i,1},\ldots,\phi_{i,n_i})$
is either
\begin{itemize}[-,topsep=0pt,itemsep=-1ex,partopsep=1ex,parsep=1ex]
\item a tuple of the form $(X,Y)$ with $XY=\pm I$, or
\item a tuple of $\pm A$, $\pm A^2$, $\pm B$, $\pm \widecheck{L}$, $\pm L$, $\pm \widehat{L}$, $\pm \widecheck{L}^{-1}$, $\pm L^{-1}$, $\pm \widehat{L}^{-1}$, except for at most $1$ component.
\end{itemize}
\end{theorem}

\begin{remark}
Theorem \ref{theorem::global monodromy fibre sum Hurwitz equivalent} and Lemma \ref{lemma::contruct_lefschetz_fibrations} imply the main result Theorem \ref{thmx::A}. Let
$H_{12}=(\widecheck{L},\widehat{L})^6$
and
\begin{equation*}
    H_{60}=
    (\widecheck{L},\widehat{L})^6
    \bullet
    (\widecheck{L},\widehat{L},\widecheck{L})^4
    \bullet
    (\widecheck{L},L,\widecheck{L},\widehat{L},\widecheck{L},\widehat{L},\widecheck{L},\widehat{L},\widecheck{L},\widehat{L},L,L)^3
\end{equation*}
be tuples in $\SL(2,\mathbb{Z})$.
Theorem \ref{thmx::B} and \ref{thmx::C} again follow from Theorem \ref{theorem::global monodromy fibre sum Hurwitz equivalent}, where the torus Lefschetz fibrations $f_{12}^L$ and $f_{60}^L$ are $f_{H_{12}}^L$ and $f_{H_{60}}^L$ respectively. 
\end{remark}

\begin{remark}
The normal form given in Theorem \ref{theorem::global monodromy fibre sum Hurwitz equivalent}, though its precise form is not given, satisfies a remarkable property, called \emph{swappability}.
We explain the swappability in Subsection \ref{subsection::swappability} but as a consequence, we have Proposition \ref{proposition::swappability_along_iota}.
\end{remark}

Let $\iota:\SL(2,\mathbb{Z})\rightarrow \PSL(2,\mathbb{Z})$ be the natural group homomorphism.

\begin{proposition}
\label{proposition::swappability_along_iota}
For $i=1,2$,
let $\prod_j (\phi_{j,1}^{(i)},\ldots,\phi_{j,n_j}^{(i)})$ be a tuple in $\SL(2,\mathbb{Z})$ such that
\begin{equation*}
\phi_{j,1}^{(i)}\cdots \phi_{j,n_j}^{(i)} = \pm I
\end{equation*}
for each $j$
and
each sub-tuple
$(\phi_{j,1}^{(i)},\ldots,\phi_{j,n_j}^{(i)})$
is either
\begin{itemize}[-,topsep=0pt,itemsep=-1ex,partopsep=1ex,parsep=1ex]
\item a tuple of the form $(X,Y)$ with $XY=\pm I$, or
\item a tuple of $\pm A$, $\pm A^2$, $\pm B$, $\pm \widecheck{L}$, $\pm L$, $\pm \widehat{L}$, $\pm \widecheck{L}^{-1}$, $\pm L^{-1}$, $\pm \widehat{L}^{-1}$, except for at most $1$ component.
\end{itemize}
Let $(g_1,\ldots,g_l)$ be a tuple in $\SL(2,\mathbb Z)$ that is either equal to $T_{\mathcal{O},0}$ or a $T_{\mathcal{O},0}$ expansion.
Suppose that
\begin{equation*}
[Cl(\phi_{j,1}^{(1)}),\ldots,Cl(\phi_{j,n_j}^{(1)}) \mid j]
=
[Cl(\phi_{j,1}^{(2)}),\ldots,Cl(\phi_{j,n_j}^{(2)}) \mid j]
\end{equation*}
and
\begin{equation*}
\prod_j (\iota(\phi_{j,1}^{(1)}),\ldots,\iota(\phi_{j,n_j}^{(1)}))
=
\prod_j (\iota(\phi_{j,1}^{(2)}),\ldots,\iota(\phi_{j,n_j}^{(2)})).
\end{equation*}
Then,
\begin{equation*}
(g_1,\ldots,g_l) \bullet
\prod_j (\phi_{j,1}^{(1)},\ldots,\phi_{j,n_j}^{(1)})
\sim
(g_1,\ldots,g_l) \bullet
\prod_j (\phi_{j,1}^{(2)},\ldots,\phi_{j,n_j}^{(2)}).
\end{equation*}
\end{proposition}

We need a deeper understanding of tuples in $\PSL(2,\mathbb{Z})$.
Set $a=\iota(A)$, $b=\iota(B)$.
Recall that $\PSL(2,\mathbb{Z})$ is generated by $a$ and $b$ with the relation $a^3=b^2=1$. Some other elements are marked as follows:
\begin{equation*}
    s_0=a^2b, s_1=aba, s_2=ba^2, t_0=ba, t_1=a^2ba^2, t_2=ab
\end{equation*}
Here $\iota(L)=s_1$, $\iota(R)=t_2$, $\iota(\widecheck{L})=s_0$ and $\iota(\widehat{L})=s_2$.
We further emphasise that $s_it_i=1$ for $i=0,1,2$.
Elements $s_0$, $s_1$, $s_2$ are conjugate to each other and $t_0$, $t_1$, $t_2$ are conjugate to each other.

Elements in $\mathcal{S}=\{a,a^2,b,s_0,s_1,s_2,t_0,t_1,t_2\}$ are called ``short'' and elements in
\begin{equation*}
    \mathcal{S}_2=\mathcal{S}\cup 
    \{
    bab, ba^2b, a^2ba, aba^2,
    a^2bab, ababa, baba^2, ba^2ba, a^2ba^2ba^2, aba^2b
    \}
\end{equation*}
are called ``almost short''.

The following improves and extends Theorem 3.6 in \cite{matsumoto1985torus}, which divides the tuples of elements in $\PSL(2,\mathbb{Z})$ conjugate to $a$, $a^2$, $b$, $s_0$ or $t_0$ into two categories and, for tuples in the second category, presents the normal forms.

\begin{theorem}
\label{theorem::Livne mathcalS}
Let $g_1,\ldots,g_n \in \PSL(2,\mathbb{Z})$ be conjugates of $a$, $a^2$, $b$, $aba$ or $a^2ba^2$ satisfying $g_1\cdots g_n=1$.
Suppose that $p_a$ of them are conjugates of $a$, $q_a$ of them are conjugates of $a^2$, $n_b$ of them are conjugates of $b$, $p$ of them are conjugates of $s_0$ and $q$ of them are conjugates of $t_0$ with $p_a$, $q_a$, $n_b$, $p$, $q\ge 0$ and $p_a+q_a+n_b+p+q=n$. Then,
\begin{enumerate}[topsep=0pt,itemsep=-1ex,partopsep=1ex,parsep=1ex]
    \item if $p=q$, $|p_a-q_a|\equiv 0\pmod{3}$ and $n_b$ is even, then the $n$-tuple $(g_1,\ldots,g_n)$ is Hurwitz equivalenet to
    \begin{equation*}
    (k_1,k_1^{-1},\ldots,k_{n'},k_{n'}^{-1},l_1,l_1,l_1,\ldots,l_{n''},l_{n''},l_{n''})    
    \end{equation*}
    with $n'+n''=n$, $k_i,l_j\in G$ and $l_j^3=1$, $i=1,\ldots,n'$, $j=1,\ldots,n''$;
    
    \item otherwise, the $n$-tuple $(g_1,\ldots,g_n)$ is Hurwitz equivalent to the concatenation of
    \begin{eqnarray*}
    (s_0,s_2,s_0,s_2,s_0,s_2)^{\lfloor \max\{p-q,0\}/6 \rfloor}\bullet
    (t_0,t_2,t_0,t_2,t_0,t_2)^{\lfloor \max\{q-p,0\}/6 \rfloor}\bullet
    (s_0,t_0)^{\min\{p,q\}}\bullet\\
    (a,a^2)^{\min\{p_a,q_a\}}\bullet
    (b,b)^{\lfloor n_b/2 \rfloor}\bullet
    (a,a,a)^{\lfloor \max\{p_a-q_a,0\}/3 \rfloor}\bullet
    (a^2,a^2,a^2)^{\lfloor \max\{q_a-p_a,0\}/3 \rfloor}
    \end{eqnarray*}
    and at most one of the following tuples:
    \begin{align*}
    &(a^2,s_0,s_2), (a,t_2,t_0), (a,s_0,s_0,s_2,s_0), (a^2,t_0,t_2,t_0,t_0), (b,s_0,s_2,s_0), (b,t_0,t_2,t_0),\\
    &(a,b,s_2),(a^2,b,t_0),(a,t_2,t_0,b,t_0,t_2,t_0),(a^2,s_0,s_2,b,s_0,s_2,s_0),\\
    &(a,a,s_0,s_2), (a^2,a^2,t_2,t_0), (a^2,a^2,s_0,s_0,s_2,s_0), (a,a,t_0,t_2,t_0,t_0),\\
    &(a^2,a^2,b,s_2),(a,a,b,t_0),(a^2,a^2,t_2,t_0,b,t_0,t_2,t_0),(a,a,s_0,s_2,b,s_0,s_2,s_0).
    \end{align*}
    The resulting $n$-tuple is called the normal form of $(g_1,\ldots,g_n)$
\end{enumerate}
\end{theorem}

As a supplement, we have Theorem \ref{theorem::Livne mathcalS2} and its modification.

\begin{theorem}
\label{theorem::Livne mathcalS2}
Let $g_1,\ldots,g_n\in \PSL(2,\mathbb{Z})$ be conjugates of $a$, $a^2$, $b$, $s_0$, $t_0$ or $ababa$ satisfying $g_1\cdots g_n=1$. Suppose that $m$ of them are conjugates of $ababa$. Take
\begin{equation*}
\mathcal{F}_{13}=(b,b,b,b,a^2bab,t_0,s_1,a^2bab,t_0,s_1,a^2bab,t_0,s_1).
\end{equation*}
Then $(g_1,\ldots,g_n)\bullet\mathcal{F}_{13}$ is Hurwitz equivalent to
\begin{equation*}
    (h_1,\ldots,h_{n-m-3-2\mu})\bullet
    (a^2bab,ba^2ba)^{(m+3-\mu)/2}\bullet
    (a^2bab,t_0,s_1)^{\mu}
\end{equation*}
where
\begin{itemize}[-,topsep=0pt,itemsep=-1ex,partopsep=1ex,parsep=1ex]
    \item each component of $(h_1,\ldots,h_{n-m-3-2\mu})$ is conjugate to one of $a,a^2,b,s_0,t_0$;
    \item $\mu=3-m$ if $m\le 3$ and $\mu=(m+1) \mod{2}$ otherwise.
\end{itemize}
\end{theorem}

\begin{theorem}[A modification of Theorem \ref{theorem::Livne mathcalS2}]
\label{theorem::Livne mathcalS2_modification}
Let $g_1,\ldots,g_n\in \PSL(2,\mathbb{Z})$ be conjugates of $a$, $a^2$, $b$, $s_0$, $t_0$ or $ababa$ satisfying $g_1\cdots g_n=1$. Suppose that $m$ of them are conjugates of $ababa$.
Take \begin{equation*}
\mathcal{F}_{13}=(b,b,b,b,a^2bab,t_0,s_1,a^2bab,t_0,s_1,a^2bab,t_0,s_1).
\end{equation*}
Let $\mathcal{F}^L$ be an $\mathcal{F}_{13}$-expansion of conjugates of $s_0$, written as
\begin{equation*}
    (u_{1,1},\ldots,u_{1,k_1},u_{2,1},\ldots,u_{2,k_2},\ldots,u_{13,1},\ldots,u_{13,k_{13}}),
\end{equation*}
such that $u_{i,1}\cdots u_{i,{k_i}}$ is equal to the $i$-th component of $\mathcal{F}_{13}$ for each of $i=1,\ldots,13$.
Then $(g_1,\ldots,g_n)\bullet\mathcal{F}^L$ is Hurwitz equivalent to
\begin{align*}
    (h_1,\ldots,h_{n'})
    &\bullet
    (a^2bab,ba^2ba)^{(m-3+\mu)/2}
    \bullet
    \prod_{i=1}^{3-\mu}
    (u_{3i+2,1},\ldots,u_{3i+2,k_{3i+2}},ba^2ba)\\
    &\bullet
    \prod_{i=3-\mu+1}^{3}
    (u_{3i+2,1},\ldots,u_{3i+2,k_{3i+2}},
    u_{3i+3,1},\ldots,u_{3i+3,k_{3i+3}},
    u_{3i+4,1},\ldots,u_{3i+4,k_{3i+4}})
\end{align*}
where:
\begin{itemize}[-,topsep=0pt,itemsep=-1ex,partopsep=1ex,parsep=1ex]
    \item each component of $(h_1,\ldots,h_{n'})$ is conjugate to one of $a,a^2,b,s_0,t_0$; 
    \item $\mu=3-m$ if $m\le 3$ and $\mu=(m+1) \mod{2}$ otherwise.
\end{itemize}
\end{theorem}

The proof of Theorem \ref{theorem::global monodromy fibre sum Hurwitz equivalent} relies on Theorem \ref{theorem::Livne mathcalS}, Theorem \ref{theorem::Livne mathcalS2} and the above modification of Theorem \ref{theorem::Livne mathcalS2}. We will prove Theorem \ref{theorem::Livne mathcalS}, \ref{theorem::Livne mathcalS2} and \ref{theorem::Livne mathcalS2_modification} in Section \ref{section::Livne}.

\begin{proof}[Proof of Theorem \ref{theorem::global monodromy fibre sum Hurwitz equivalent}]
Suppose that $(h_1,\ldots,h_m)=T_{\mathcal{O}}$ if $f_0$ is as in $\RNum{1})$, or $(h_1,\ldots,h_m)$ is a $T_{\mathcal{O}}$-expansion if $f_0$ is as in $\RNum{2})$, which is a global monodromy of $f_0$. We write it as a concatenation either
\begin{itemize}[-,topsep=0pt,itemsep=-1ex,partopsep=1ex,parsep=1ex]
    \item of all the following tuples, or
    \item of the following tuples labelled $(1)$, $(2)$ and $(4)$.
\end{itemize}
The list of tuples is as follows.
\begin{enumerate}[(1),topsep=0pt,itemsep=-1ex,partopsep=1ex,parsep=1ex]
    %\item $(\phi_1,\ldots,\phi_n)$, which is a global monodromy of $f$.
    \item The tuple $(h_{1,1},\ldots,h_{1,{m_1}})$ is either $(\widecheck{L},\widehat{L},A^2)$ or a $(\widecheck{L},\widehat{L},A^2)$-expansion of conjugates of $L$.
    
    \item The tuple $(h_{2,1},\ldots,h_{2,{m_2}})$ is either $(\widecheck{L},\widehat{L},A^2)$ or a $(\widecheck{L},\widehat{L},A^2)$-expansion of conjugates of $L$, which may be different from $(h_{1,1},\ldots,h_{1,{m_1}})$.
    
    \item The tuple $(h_{3,1},\ldots,h_{3,{m_3}})$ is either $T_{\mathcal{O},1}$ or a $T_{\mathcal{O},1}$-expansion of conjugates of $L$.
    
    \item The tuple $(h_{4,1},\ldots,h_{4,{m_4}})$ is either $T_{\mathcal{O},2}$ or a $T_{\mathcal{O},2}$-expansion of conjugates of $L$.
\end{enumerate}

\textbf{Step 1.} We first show that any global monodromy of a direct sum $f\oplus f_0$ can be transformed into
\begin{equation*}
    (\phi_1,\ldots,\phi_n)\bullet (h_1,\ldots,h_m)
\end{equation*}
where $(\phi_1,\ldots,\phi_n)$ is a global monodromy of $f$.

Given a base point $p$ of $f\oplus f_0$, a global monodromy of the direct sum with respect to $p$ is the concatenation
\begin{equation*}
    (\psi_1^{-1}\phi_1\psi_1,\ldots,\psi_1^{-1}\phi_n\psi_1)\bullet\prod_{i\in \mathcal{I}} (\psi_2^{-1}h_{i,1}\psi_2,\ldots,\psi_2^{-1}h_{i,m_i}\psi_2)
\end{equation*}
with the tuple $(\phi_1,\ldots,\phi_n)$ a global monodromy of $f$, elements $\psi_1$, $\psi_2\in \SL(2,\mathbb{Z})$ and the index set $\mathcal{I}$ either $\{1,2,3,4\}$ or $\{1,2,4\}$.
Each tuple in the concatenation has the product of components equal to $\pm I$ and both of the tuples $(\psi_2^{-1}h_{1,1}\psi_2,\ldots,\psi_2^{-1}h_{1,m_1}\psi_2)$ and $(\psi_2^{-1}h_{2,1}\psi_2,\ldots,\psi_2^{-1}h_{2,m_2}\psi_2)$ contain generating sets.
By Lemma \ref{lemma::tuples containing a generating set}, we can eliminate all the $\psi_1$, $\psi_2$ in the global monodromy using elementary transformations.
Rewrite the resulting tuple as
\begin{equation*}
    (\phi_1,\ldots,\phi_n)\bullet
    (g_1,\ldots,g_l)\bullet
    (h_{3,1},\ldots,h_{3,{m_3}})^{[3\in\mathcal{I}]}\bullet
    (h_{4,1},\ldots,h_{4,{m_4}})
\end{equation*}
where $[3\in\mathcal{I}]=1$ if $3\in\mathcal{I}$ and $[3\in\mathcal{I}]=0$ if $3\not\in\mathcal{I}$,
such that $(g_1,\ldots,g_l)$ is either $T_{\mathcal{O},0}$ or a $T_{\mathcal{O},0}$-expansion of conjugates of $L$.

\textbf{Step 2.} We show that the above resulting tuple is Hurwitz equivalent to
\begin{equation*}
    (\varphi_1,\ldots,\varphi_{n'})
    \bullet (g_1,\ldots,g_l)
    \bullet (h_{3,1},\ldots,h_{3,{m_3}})^{[3\in\mathcal{I}]}
    \bullet (h'_{4,1},\ldots,h'_{4,{m_4}})
\end{equation*}
such that
\begin{itemize}[-,topsep=0pt,itemsep=-1ex,partopsep=1ex,parsep=1ex]
    \item $(\varphi_1,\ldots,\varphi_{n'})$ is a tuple of elements in simple conjugacy classes such that $\varphi_1\cdots\varphi_{n'}=\pm I$;
    \item $(h'_{4,1},\ldots,h'_{4,{m_4}})$ depends only on $(h_{4,1},\ldots,h_{4,{m_4}})$ and components of $(\phi_1,\ldots,\phi_n)$ in non-simple conjugacy classes.
\end{itemize}
If $(h_{4,1},\ldots,h_{4,{m_4}})$ is a $T_{\mathcal{O},2}$-expansion, then using contractions on $(h_{4,1},\ldots,h_{4,{m_4}})$ as in Subsection \ref{subsection::contractions-restorations} we replace $(h_{4,1},\ldots,h_{4,{m_4}})$ with $T_{\mathcal{O},2}$.
The definition of $T_{\mathcal{O},2}$ states that it is the concatenation of several sub-tuples. These sub-tuples are in one-to-one correspondence with the singular fibres of $f$ whose fibre monodromies belong to non-simple conjugacy classes and they are further in one-to-one correspondence with the components of $(\phi_1,\ldots,\phi_n)$ excluding those of trace $0$, $\pm 1$, $\pm 3$ or conjugate to $\pm L$, $\pm R$.

Suppose that there exists an $(r+1)$-sub-tuple of the form $(L,\ldots,L,L^{-r})$ in $T_{\mathcal{O},2}$ with $r\ge 2$. We take the corresponding component, say $\phi_i$, which is equal to $\epsilon h^{-1}L^rh$ with $\epsilon=\pm 1$ and $h\in \SL(2,\mathbb{Z})$.
Since $(g_1,\ldots,g_l)$ contains a generating set, by Lemma \ref{lemma::tuples containing a generating set}, we replace the $(r+1)$-sub-tuple with
\begin{equation*}
    (h^{-1}Lh,\ldots,h^{-1}Lh,h^{-1}L^{-r}h).
\end{equation*}
By Lemma \ref{lemma::move subtuples}, we further replace $\phi_i$ with $(h^{-1}Lh,\ldots,h^{-1}Lh)$ and replace the above $(r+1)$-sub-tuple with $(\epsilon h^{-1}L^rh,h^{-1}L^{-r}h)$.
Again by Lemma \ref{lemma::tuples containing a generating set}, the pair $(\epsilon h^{-1}L^rh,h^{-1}L^{-r}h)$ can be transformed into $(\epsilon L^r,L^r)$.

We have similar arguments for sub-tuples of the form $(R,\ldots,R,R^{-r})$ or of the form
\begin{equation*}
    (R,\ldots,R,L,\ldots,L,\ldots,R,\ldots,R,L,\ldots,L,(R\ldots RL\ldots L\ldots R\ldots RL\ldots L)^{-1}) 
\end{equation*}
as in Definition \ref{definition::additional tuple}.
For the restoration, according to each component in $T_{\mathcal{O},2}$, we rewrite the corresponding component as a sub-tuple.
Notice that if such a component belongs to some simple conjugacy class,
then it is replaced by a sub-tuple of conjugates of $L$. Hence the resulting tuple is as desired.

We will not modify $(h'_{4,1},\ldots,h'_{4,{m_4}})$ anymore.

\textbf{Step 3.}
Suppose that $n_{+}$ components of $(\varphi_1,\ldots,\varphi_{n'})$ are of trace $3$ and $n_{-}$ components of $(\varphi_1,\ldots,\varphi_{n'})$ are of trace $-3$.
If $n_{+}+n_{-}=0$, then take $(\varphi'_1,\ldots,\varphi'_{n''})=(\varphi_1,\ldots,\varphi_{n'})$ and skip the step.
Otherwise, $[3\in\mathcal{I}]=1$.
We further show that, by elementary transformations, $(\varphi_1,\ldots,\varphi_{n'})\bullet
(g_1,\ldots,g_l) \bullet
(h_{3,1},\ldots,h_{3,{m_3}})$ can be transformed into
\begin{equation*}
    (\varphi'_1,\ldots,\varphi'_{n''})\bullet (g_1,\ldots,g_l) \bullet (\bar{h}_1,\ldots,\bar{h}_{\bar{m}})
\end{equation*}
such that
\begin{itemize}[-,topsep=0pt,itemsep=-1ex,partopsep=1ex,parsep=1ex]
    \item $(\varphi'_1,\ldots,\varphi'_{n''})$ is a tuple of elements either of trace $0,\pm 1$ or conjugate to $\pm L$ or $\pm R$,
    \item $(\bar{h}_1,\ldots,\bar{h}_{\bar{m}})$ depends only on $n_+$, $n_-$ and $(h_{3,1},\ldots,h_{3,{m_3}})$.
\end{itemize}
If $(h_{3,1},\ldots,h_{3,{m_3}})$ is an expansion of $T_{\mathcal{O},1}$, then using contractions on $(h_{3,1},\ldots,h_{3,{m_3}})$ as in Subsection \ref{subsection::contractions-restorations} we replace it with $T_{\mathcal{O},1}$.
By applying Theorem \ref{theorem::Livne mathcalS2} to $(\iota(\varphi_1),\ldots,\iota(\varphi_{n'}))\bullet\mathcal{F}_{13}$, the tuple is transformed into
\[
(\varphi'_1,\ldots,\varphi'_{n''})\bullet (g_1,\ldots,g_l) \bullet \prod_{i=1}^{k} (\psi_{i,0},\psi_{i,1}) \bullet (-A^2BAB,BA,-ABA)^{\mu}
\]
such that each of $\iota(\varphi'_i)$, $i=1,\ldots,n''$ is conjugate to $a$, $a^2$, $b$, $s_0$ or $t_0$, $\iota(\psi_{i,0})=a^2bab$, $\iota(\psi_{i,1})=ba^2ba$ and $\mu\le 3$.
The number $\mu$ is determined by $n_++n_-$, which further separates the cases.

Then, the restoration operations apply on the tuple.
Some components of $(\varphi'_1,\ldots,\varphi'_{n''})$ are replaced by sub-tuples of elements conjugate to $L$,
while keeping each component conjugate to some preimage of $a$, $a^2$, $b$, $s_0$ or $t_0$.
The remaining components that might be modified by the restoration are exactly the components of the last $\mu$ sub-triples and the last $3-\mu$ components denoted by $\psi_{i,0}$.
By Theorem \ref{theorem::Livne mathcalS2_modification} they are replaced by certain sub-tuples of $(h_{3,1},\ldots,h_{3,{m_3}})$.

The remaining components of $\prod_{i=1}^{k} (\psi_{i,0},\psi_{i,1})$ are of trace $\pm 3$ and they are either $\pm A^2BAB$ or $\pm BA^2BA$. To restrict their dependencies only on $n_+$ and $n_-$, we have to show that their signs can be rearranged to certain positions, but this follows from Proposition \ref{proposition::swappability_along_iota}.

\iffalse
Since $(g_1,\ldots,g_l)$ contains a generating set,
Lemma \ref{lemma::switch places of components with the same value} and the following substitutions realise what we desired:
\begin{align*}
    &(\epsilon_1 A^2BAB,\epsilon_2 BA^2BA)
    \bullet (g_1,\ldots,g_l)
    \longrightarrow
    (\epsilon_1 BA^2BA,\epsilon_2 B^2A^2BAB^{-1})
    \bullet (g_1,\ldots,g_l)\\
    &=
    (\epsilon_1 BA^2BA,\epsilon_2 A^2BAB)
    \bullet (g_1,\ldots,g_l)
    \longrightarrow
    (\epsilon_2 A^2BAB,\epsilon_1 BA^2BA)
    \bullet (g_1,\ldots,g_l).
\end{align*}
\fi

\textbf{Step 4.}
We conclude the proof of Theorem \ref{theorem::global monodromy fibre sum Hurwitz equivalent} by showing that $(\varphi'_1,\ldots,\varphi'_{n''})\bullet (g_1,\ldots,g_l)$ is Hurwitz equivalent to $(\varphi''_1,\ldots,\varphi''_{n''})\bullet (g_1,\ldots,g_l)$ such that $(\varphi''_1,\ldots,\varphi''_{n''})$ depends only on the multi-set $\mathcal{O}$.

Applying Theorem \ref{theorem::Livne mathcalS} to
$(\iota(\varphi'_1),\ldots,\iota(\varphi'_{n''}))$, we transform the tuple $(\varphi'_1,\ldots,\varphi'_{n''})$ into a new tuple, denoted by $(\varphi''_1,\ldots,\varphi''_{n''})$.
For the first case in Theorem \ref{theorem::Livne mathcalS}, as $(g_1,\ldots,g_l)$ contains a generating set, applying Lemma \ref{lemma::tuples containing a generating set} we further transform the concatenation into a resulting tuple, denoted by $(\varphi''_1,\ldots,\varphi''_{n''})\bullet (g_1,\ldots,g_l)$, satisfying
\begin{equation*}
    (\iota(\varphi''_1),\ldots,\iota(\varphi''_{n''})) =
    (s_0,t_0)^{\mu_1}\bullet (a,a^2)^{\mu_2}\bullet (b,b)^{\mu_3}\bullet (a,a,a)^{\mu_4}\bullet (a^2,a^2,a^2)^{\mu_5}
\end{equation*}
with $\mu_1,\ldots,\mu_5$ determined by $\mathcal{O}$.
The theorem follows from Proposition \ref{proposition::swappability_along_iota}.
%The theorem follows from Lemma \ref{lemma::swap components with the same projection in PSL(2,Z)}.
\end{proof}

\begin{remark}
Alternatively, instead of using Proposition \ref{proposition::swappability_along_iota}, one may apply the substitutions 
\begin{align*}
    &(\epsilon_1 A^2BAB,\epsilon_2 BA^2BA)
    \bullet (g_1,\ldots,g_l)
    \longrightarrow
    (\epsilon_1 BA^2BA,\epsilon_2 B^2A^2BAB^{-1})
    \bullet (g_1,\ldots,g_l)\\
    &=
    (\epsilon_1 BA^2BA,\epsilon_2 A^2BAB)
    \bullet (g_1,\ldots,g_l)
    \longrightarrow
    (\epsilon_2 A^2BAB,\epsilon_1 BA^2BA)
    \bullet (g_1,\ldots,g_l).
\end{align*}
at the end of Step 3 and
\begin{align*}
    (\tau_1 A^2B, \tau_2 BA^2, \tau_3 A^2B)
    \longleftrightarrow
    (\tau_2 BA^2, \tau_1 ABA, \tau_3 A^2B)
    \longleftrightarrow
    (\tau_2 BA^2, \tau_3 A^2B, \tau_1 BA^2).
\end{align*}
at the end of Step 4, where $\epsilon_1$, $\epsilon_2$, $\tau_1$, $\tau_2$, $\tau_3\in\{-I,+I\}$ are arbitrary.
They appeared in an earlier version of this paper.
\end{remark}

We end with the proof of Corollary \ref{corollary::elliptic surfaces}.
\begin{proof}[Proof of Corollary \ref{corollary::elliptic surfaces}]
The fibre monodromy of a singular fibre distinguishes the type in the Kodaira classification. The corollary follows from Theorem \ref{thmx::B}.
\end{proof}

%% file: Section3/swappability.tex
\section{Swappability and local models}

\subsection{Swappability 
of the normal form}
\label{subsection::swappability}

This subsection introduces the notion of \emph{swappability} for a tuple in $\SL(2,\mathbb Z)$ with a stabilisation.

\begin{definition}
Let $(\phi_1,\ldots,\phi_n)$ and $(g_1,\ldots,g_l)$ be tuples in $\SL(2,\mathbb Z)$.
Suppose that, for any abelian group $G$ and tuples $(\epsilon_1,\ldots,\epsilon_n)$, $(\epsilon'_1,\ldots,\epsilon'_n)$, $(\sigma_1,\dots,\sigma_l)$ in $G$ such that the following multi-sets of conjugacy classes in $G\times \SL(2,\mathbb Z)$ coincide:
\begin{equation*}
[Cl\big((\epsilon_1,\phi_1)\big),\ldots,Cl\big((\epsilon_n,\phi_n)\big)]
=
[Cl\big((\epsilon'_1,\phi_1)\big),\ldots,Cl\big((\epsilon'_n,\phi_n)\big)],
\end{equation*}
we have
\begin{equation*}
\big((\sigma_1,g_1),\ldots,(\sigma_l,g_l)\big)
\bullet
\big((\epsilon_1,\phi_i),\ldots,(\epsilon_n,\phi_n)\big)
\sim
\big((\sigma_1,g_1),\ldots,(\sigma_l,g_l)\big)
\bullet
\big((\epsilon'_1,\phi_i),\ldots,(\epsilon'_n,\phi_n)\big).
\end{equation*}
In this case,
we say that $(\phi_1,\ldots,\phi_n)$ is
$(g_1,\ldots,g_l)$-stabilised \emph{swappable}.
\end{definition}

\begin{remark}
The normal form given in Theorem \ref{theorem::global monodromy fibre sum Hurwitz equivalent}
is an example of swappable tuples,
which is guaranteed by Proposition \ref{proposition::swappability}.
\end{remark}

\begin{proposition}
\label{proposition::swappability}
Let $\prod_i (\phi_{i,1},\ldots, \phi_{i,n_{i}})$
be a tuple in $\SL(2,\mathbb Z)$ such that
$\phi_{i,1}\cdots \phi_{i,n_i}=\pm I$ for each $i$ and each 
$(\phi_{i,1},\ldots,\phi_{i,n_i})$
is either
\begin{itemize}[-,topsep=0pt,itemsep=-1ex,partopsep=1ex,parsep=1ex]
\item a tuple of the form $(X,Y)$ with $XY=\pm I$, or
\item a tuple of $\pm A$, $\pm A^2$, $\pm B$, $\pm \widecheck{L}$, $\pm L$, $\pm \widehat{L}$, $\pm \widecheck{L}^{-1}$, $\pm L^{-1}$, $\pm \widehat{L}^{-1}$, except for at most $1$ component.
\end{itemize}
Let $(g_1,\ldots,g_l)$ be a tuple in $\SL(2,\mathbb Z)$ either equal to $T_{\mathcal{O},0}$ or a $T_{\mathcal{O},0}$ expansion.
Then the tuple $\prod_i (\phi_{i,1},\ldots, \phi_{i,n_{i}})$ is $(g_1,\ldots,g_l)$-stabilised swappable.
\end{proposition}

\begin{proof}
We need only prove the proposition for the case $(g_1,\ldots,g_l)=T_{\mathcal{O},0}$.

Set $(\phi_1,\ldots,\phi_n)=\prod_i (\phi_{i,1},\ldots, \phi_{i,n_{i}})$
and consider the tuple
\begin{equation*}
\big((\sigma_1,g_1),\ldots,(\sigma_l,g_l)\big)
\bullet
\big((\epsilon_1,\phi_i),\ldots,(\epsilon_n,\phi_n)\big)
\end{equation*}
in $G\times \SL(2,\mathbb Z)$.
It suffices to show that, for any two components $(\epsilon_i,\phi_i)$ and $(\epsilon_j,\phi_j)$ such that $\phi_i$ is conjugate to $\phi_j$, one can interchange $\epsilon_i$ and $\epsilon_j$
using elementary transformations.

When $\phi_i=\phi_j$, the swapping follows from Lemma \ref{lemma::switch places of components with the same value}.

When $\phi_i\neq \phi_j$ but $\phi_i$, $\phi_j$ belong to different sub-tuples, using Lemma \ref{lemma::tuples containing a generating set} for the sub-tuple containing $(\epsilon_i,\phi_i)$, we transform the component $(\epsilon_i,\phi_i)$ into $(\epsilon_i,\phi_j)$.
After swapping $(\epsilon_i,\phi_j)$ and $(\epsilon_j,\phi_j)$, we apply Lemma \ref{lemma::tuples containing a generating set} again to make other components unchanged.

When $\phi_1\neq \phi_j$ and $\phi_i$, $\phi_j$ belong to a sub-tuple not of the form $(X,Y)$ with $XY=\pm I$,
we must have $\phi_i$ conjugate to one of $L$, $-L$, $L^{-1}$ and $-L^{-1}$.
Recall
$T_{\mathcal{O},0}=(\widecheck{L}, \widehat{L}, A^2)\bullet (\widecheck{L}, \widehat{L}, A^2)$
and notice that the (first) $(\widecheck{L}, \widehat{L}, A^2)$ contains generating sets.
If $\phi_i$ and $\phi_j$ are conjugate to $\pm L$,
then take $Q_i$ and $Q_j$ be such that
$Q_i^{-1}\phi_i Q_i=\pm L$ and $Q_j^{-1}\phi_j Q_j=\pm L$, therefore
$Q_j Q_i^{-1} \phi_i Q_i Q_j^{-1} = \phi_j$
and
$Q_i Q_j^{-1} \phi_j Q_j Q_i^{-1} = \phi_i$.
Using Lemma \ref{lemma::tuples containing a generating set} and Lemma \ref{lemma::switch places of components with the same value}, we have the following substitutions
\begin{align*}
&
(\ldots, (\epsilon_i,\phi_i), \ldots, (\epsilon_j,\phi_j), \ldots)
\bullet
((\delta_1,\widecheck{L}), (\delta_2,\widehat{L}), \ldots, (\delta_6,A^2))
\\
\longrightarrow&
(\ldots, (\epsilon_i,Q_i^{-1}\phi_iQ_i), \ldots, (\epsilon_j,Q_i^{-1}\phi_jQ_i), \ldots)
\bullet
((\delta_1,\widecheck{L}), (\delta_2,\widehat{L}), \ldots, (\delta_6,A^2))
\\
\longrightarrow&
(\ldots, (\delta_1,\widecheck{L}), \ldots, (\epsilon_j,Q_i^{-1}\phi_jQ_i), \ldots)
\bullet
((\epsilon_i,Q_i^{-1}\phi_iQ_i), (\delta_2,\widehat{L}), \ldots, (\delta_6,A^2))
\\
\longrightarrow&
(\ldots, (\delta_1,Q_i\widecheck{L}Q_i^{-1}), \ldots, (\epsilon_j,\phi_j), \ldots)
\bullet
((\epsilon_i,Q_i^{-1}\phi_iQ_i), (\delta_2,\widehat{L}), \ldots, (\delta_6,A^2))
\\
\longrightarrow&
(\ldots, (\delta_1,Q_j^{-1}Q_i\widecheck{L}Q_i^{-1}Q_j), \ldots, (\epsilon_j,Q_j^{-1}\phi_jQ_j), \ldots)
\bullet
((\epsilon_i,Q_i^{-1}\phi_iQ_i), (\delta_2,\widehat{L}), \ldots, (\delta_6,A^2))
\\
\longrightarrow&
(\ldots, (\delta_1,Q_j^{-1}Q_i\widecheck{L}Q_i^{-1}Q_j), \ldots, (\epsilon_i,Q_i^{-1}\phi_iQ_i), \ldots)
\bullet
((\epsilon_j,Q_j^{-1}\phi_jQ_j), (\delta_2,\widehat{L}), \ldots, (\delta_6,A^2))
\\
\longrightarrow&
(\ldots, (\delta_1,Q_i\widecheck{L}Q_i^{-1}), \ldots, (\epsilon_i,Q_jQ_i^{-1}\phi_iQ_iQ_j^{-1}), \ldots)
\bullet
((\epsilon_j,Q_j^{-1}\phi_jQ_j), (\delta_2,\widehat{L}), \ldots, (\delta_6,A^2))
\\
\longrightarrow&
(\ldots, (\delta_1,\widecheck{L}), \ldots, (\epsilon_i,Q_i^{-1}Q_jQ_i^{-1}\phi_iQ_iQ_j^{-1}Q_i), \ldots)
\bullet
((\epsilon_j,Q_j^{-1}\phi_jQ_j), (\delta_2,\widehat{L}), \ldots, (\delta_6,A^2))
\\
\longrightarrow&
(\ldots, (\epsilon_j,Q_j^{-1}\phi_jQ_j), \ldots, (\epsilon_i,Q_i^{-1}Q_jQ_i^{-1}\phi_iQ_iQ_j^{-1}Q_i), \ldots)
\bullet
((\delta_1,\widecheck{L}), (\delta_2,\widehat{L}), \ldots, (\delta_6,A^2))
\\
\longrightarrow&
(\ldots, (\epsilon_j,Q_iQ_j^{-1}\phi_jQ_jQ_i^{-1}), \ldots, (\epsilon_i,Q_jQ_i^{-1}\phi_iQ_iQ_j^{-1}), \ldots)
\bullet
((\delta_1,\widecheck{L}), (\delta_2,\widehat{L}), \ldots, (\delta_6,A^2))
\\
=&
(\ldots, (\epsilon_j,\phi_i), \ldots, (\epsilon_i,\phi_j), \ldots)
\bullet
((\delta_1,\widecheck{L}), (\delta_2,\widehat{L}), \ldots, (\delta_6,A^2)).
\end{align*}
If $\phi_i$ and $\phi_j$ are conjugate to $\pm L^{-1}$, then using the contraction on $((\delta_2,\widehat{L}),(\delta_3,A^2))$ we have a similar sequence of substitutions.

When $\phi_i\neq \phi_j$ and $\phi_i$, $\phi_j$ form a sub-tuple of the form $(X,Y)$ with $XY=\pm I$,
there exists $Q\in \SL(2,\mathbb Z)$ such that $Q^{-1} \phi_i Q = \phi_j$ and therefore $Q^{-1} \phi_j Q = \phi_i$.
By Lemma \ref{lemma::tuples containing a generating set}, the sub-tuple $((\epsilon_i,\phi_i),(\epsilon_j,\phi_j))$ can be transformed into $((\epsilon_i,Q^{-1}\phi_iQ ),(\epsilon_j,Q^{-1}\phi_jQ))$.
\end{proof}

As a consequence, we prove Proposition \ref{proposition::swappability_along_iota}.

\begin{proof}[Proof of Proposition \ref{proposition::swappability_along_iota}]
Let $G=\{1, -1\}$ be the group under multiplication.
We define
\begin{equation*}
\SL(2,\mathbb Z) \ni \phi \mapsto
\sharp(\phi)=
(\epsilon,\psi) \in G\times \SL(2,\mathbb Z)
\end{equation*}
such that $\trace(\psi)\ge 0$, $\epsilon=\sgn(\trace(\phi))$ if $\trace(\phi)\neq 0$ and $\epsilon=1$ otherwise.
This map is well-defined, injective and conjugacy-preserving, but not a group homomorphism.

Consider the tuple
\begin{equation*}
\prod_j 
((\epsilon_{j,1}^{(1)},\psi_{j,1}^{(1)}),\ldots,(\epsilon_{j,n_j}^{(1)},\psi_{j,n_j}^{(1)}))
=
\prod_j (\sharp(\phi_{j,1}^{(1)}),\ldots,\sharp(\phi_{j,n_j}^{(1)}))
\end{equation*}
and the tuple
\begin{equation*}
\prod_j 
((\epsilon_{j,1}^{(2)},\psi_{j,1}^{(2)}),\ldots,(\epsilon_{j,n_j}^{(2)},\psi_{j,n_j}^{(2)}))
=
\prod_j (\sharp(\phi_{j,1}^{(2)}),\ldots,\sharp(\phi_{j,n_j}^{(2)}))
\end{equation*}
in $\times \SL(2,\mathbb Z)$.
Their components present the same conjugacy classes counted with multiplicity and $\psi_{j,k}^{(1)}=\psi_{j,k}^{(2)}$ for all $j$ and $k$.
Besides, each sub-tuple of $\prod_j 
(\psi_{j,1}^{(1)},\ldots,\psi_{j,n_j}^{(1)})$
is either a tuple of the form $(X,Y)$ with $XY=\pm I$ or a tuple of $A$, $-A^2$, $\pm B$, $\widecheck{L}$, $L$, $\widehat{L}$, $\widecheck{L}^{-1}$, $L^{-1}$ and $\widehat{L}^{-1}$.
This proposition follows from Proposition \ref{proposition::swappability}.
\end{proof}

%% file: Section3/local_global.tex
\subsection{Fibre-preserving homeomorphisms: from local to global}

This subsection investigates fibre-preserving homeomorphisms between torus fibrations.
We start with the following definitions.
%first provide an additional restriction of singular fibres and introduce the local models.

\begin{definition}
Suppose that $f:M\rightarrow S^2$ is a torus fibration and $p_j\in S^2$ is a branch point.
The singular fibre $f^{-1}(p_j)$ may be \emph{locally symmetric} in the following sense.
Let $U\subset S^2$ be any sufficiently small neighbourhood of $p_j$ and $p\in\partial U$ be an arbitrary point.
Identifying $f^{-1}(p)$ with $\mathbb{T}^2$,
we use $\phi_j\in \Mod(\mathbb{T}^2)$ to denote the monodromy along $\partial U$ at $p$.
Let $\psi\in \Mod(\mathbb{T}^2)$ be an arbitrary mapping class class such that
$\psi \phi_j=\phi_j \psi$.
We suppose that there exists a (self-)homeomorphism
$\Psi_M:f^{-1}(U)\rightarrow f^{-1}(U)$
such that $f\circ \Psi_M = f$
and $\Psi_M\mid_{f^{-1}(p)}$ represents $\psi$.
In this case, we say that the singular fibre $f^{-1}(p)$ is \emph{locally symmetric}.
\end{definition}

In particular, all singular fibres of a torus Lefschetz fibration are locally symmetric.

\begin{definition}
Suppose that $f_1:M_1\rightarrow S^2$ and $f_2:M_2\rightarrow S^2$ are torus fibrations with branch sets $\mathcal{B}_1=\{p_j^{(1)}\}$ and $\mathcal{B}_2=\{p_j^{(2)}\}$.
We say that the singular fibres $f_1^{-1}(p_j^{(1)})$ and $f_2^{-1}(p_j^{(2)})$ are \emph{locally fibre-preserving homeomorphic}
if, for any sufficiently small neighbourhood $U_j^{(1)}$ of $p_j^{(1)}$, there exist homeomorphisms
$\Psi_{S,j}:U_j^{(1)}\rightarrow S^2$ and
$\Psi_{M,j}:f_1^{-1}(U_j^{(1)})\rightarrow f_2^{-1}(\Psi_{S,j}(U_j^{(1)}))$
such that $f_2\circ \Psi_{M,j}=\Psi_{S,j}\circ f_1$.
We further say that $(M_1,f_1)$ and $(M_2,f_2)$ are \emph{fibre-preserving homeomorphic}
if there exist homeomorphisms $\Psi_S:S^2\rightarrow S^2$ and
$\Psi_M:M_1\rightarrow M_2$
such that $f_2\circ \Psi_M=\Psi_S\circ f_1$.
\end{definition}

\begin{definition}
A locally fibre-preserving homeomorphism
$(\Psi_{S,j},\Psi_{M,j})$ as above
may be \emph{compatible} with given global monodromies $(\phi_1^{(1)},\ldots,\phi_n^{(1)})$ of $f_1$ and $(\phi_1^{(2)},\ldots,\phi_n^{(2)})$ of $f_2$, in the following sense.
Recall that the global monodromy is determined by a base point $p^{(i)}$ and a collection of loops $\gamma_1^{(i)},\ldots,\gamma_n^{(i)}$ based at $p^{(i)}$ such that $\gamma_j^{(i)}$ is exactly the boundary of a neighbourhood of $p_j^{(i)}$, say $\gamma_j^{(i)}=\partial D_j^{(i)}$, for $i=1,2$.
Without loss of generality, assume that $U_j^{(1)}\subset D_j^{(1)}$
and $\Psi_{S,j}(U_j^{(1)})\subset D_j^{(2)}$.
Let $\beta_j^{(i)}$ be an arbitrary path in $D_j^{(i)}$
joining $p^{(i)}$ to some point on $\partial D_j^{(i)}$, for $i=1,2$.
The locally fibre-preserving homeomorphism $(\Psi_{M,j},\Psi_{S,j})$ is pushed forward
to a homeomorphism $\psi\in \Homeo(\mathbb{T}^2)$ between the generic fibres at base points.
We say that $(\Psi_{M,j},\Psi_{S,j})$ is \emph{compatible} with the global monodromies if $[\psi]\phi_j^{(1)}=\phi_j^{(2)}[\psi]$.
\end{definition}

The compatibility does not depend on the choice of $\beta_j^{(i)}$, for $i=1,2$.
Indeed, for a different choice of $(\beta_j^{(1)}, \beta_j^{(2)})$,
then $[\psi]$ is replaced by
$(\phi_j^{(2)})^{k_2} [\psi] (\phi_j^{(1)})^{k_1}$
for some $k_1,k_2\in \mathbb{Z}$.
It is easy to check that
\begin{equation*}
\Big( (\phi_j^{(2)})^{k_2} [\psi] (\phi_j^{(1)})^{k_1} \Big)
\phi_j^{(1)} =
\phi_j^{(2)} \Big( (\phi_j^{(2)})^{k_2} [\psi] (\phi_j^{(1)})^{k_1} \Big).
\end{equation*}
Besides, the compatibility does not depend on the choice of global monodromies.
Indeed, a different pair of global monodromies replaces $\phi_j^{(1)}$ and $\phi_j^{(2)}$ with
$Q_1^{-1}\phi_j^{(1)}Q_1$
and
$Q_2^{-1}\phi_j^{(2)}Q_2$, respectively.
The set of all possibilities for $[\psi]$ is
\begin{equation*}
\left\{
\Big(Q_2^{-1} \phi_j^{(2)} Q_2\Big)^{k_2}
Q_2^{-1}[\psi] Q_1
\Big(Q_1^{-1} \phi_j^{(1)} Q_1\Big)^{k_1}
\middle\vert
k_1,k_2\in \mathbb{Z}
\right\}.
\end{equation*}
It is easy to check that
\begin{align*}
\Big(
\big(Q_2^{-1} \phi_j^{(2)} Q_2\big)^{k_2}
Q_2^{-1}[\psi] Q_1
\big(Q_1^{-1} \phi_j^{(1)} Q_1\big)^{k_1}
\Big)
&
Q_1^{-1} \phi_j^{(1)} Q_1
\\
=&
Q_2^{-1} \phi_j^{(2)} Q_2
\Big(
\big(Q_2^{-1} \phi_j^{(2)} Q_2\big)^{k_2}
Q_2^{-1}[\psi] Q_1
\big(Q_1^{-1} \phi_j^{(1)} Q_1\big)^{k_1}
\Big).
\end{align*}

\begin{theorem}
\label{theorem::fibre-preserving equivalence}
Let $f_1:M_1\rightarrow S^2$ and $f_2:M_2\rightarrow S^2$ be torus fibrations with branch sets $\mathcal{B}_1=\{p_j^{(1)}\}$ and $\mathcal{B}_2=\{p_j^{(2)}\}$, $|\mathcal{B}_1|=n=|\mathcal{B}_2|$,
with global monodromies
$(\phi_1^{(1)},\ldots,\phi_n^{(1)})$ and $(\phi_1^{(2)},\ldots,\phi_n^{(2)})$
such that each singular fibre is locally symmetric.
Suppose that, for each $j$, there exists a locally fibre-preserving homeomorphism between $f_1^{-1}(p_j^{(1)})$ and $f_2^{-1}(p_j^{(2)})$ compatible with the given global monodromies.
Let
$\widetilde{f_1}=f_1\oplus f_{\mathcal{O}(f_1)}^L:\widetilde{M}_1\rightarrow S^2$ and $\widetilde{f_2}=f_2\oplus f_{\mathcal{O}(f_2)}^L:\widetilde{M}_2\rightarrow S^2$ be direct sums.
Then $(\widetilde{M}_1,\widetilde{f}_1)$ and $(\widetilde{M}_2,\widetilde{f}_2)$ are fibre-preserving homeomorphic.
\end{theorem}

\begin{remark}
The one-to-one correspondence between singular fibres via locally fibre-preserving homeomorphisms in the hypothesis of Theorem \ref{theorem::fibre-preserving equivalence} implies that $\mathcal{O}(f_1)=\mathcal{O}(f_2)$.
\end{remark}

The following definition first appeared in Part \URNum{2}, Definition 4 in \cite{moishezon1977complex}.

\begin{definition}
Suppose that $f:M\rightarrow S^2$ is a torus fibration with branch set $\mathcal{B}=\{p_j\}$.
Let $\alpha:S^1\rightarrow \Homeo_0(\mathbb{T}^2)$ be a closed curve in the group of homeomorphisms of $\mathbb{T}^2$ isotopic to the identity.
Let $D\subset S^2\setminus \mathcal{B}$ be a disc.
Identify $\partial D$ with $S^1$
and $f^{-1}(D)$ with $D\times \mathbb{T}^2$,
then $\alpha$ defines a canonical homeomorphism
\begin{equation*}
\widetilde{\alpha}:
f^{-1}(\partial D) \rightarrow \partial(D\times \mathbb{T}^2).
\end{equation*}
Denote
$M_{D,\alpha} = \overline{M\setminus f^{-1}(D)} \cup_{\widetilde{\alpha}} (D\times \mathbb{T}^2)$
and let $f_{D,\alpha}:M_{D,\alpha}\rightarrow S^2$ be the map
which is equal to $f$ on $\overline{M\setminus f^{-1}(D)}$ and equal to the projection $D\times \mathbb{T}^2 \rightarrow D$ on $D\times \mathbb{T}^2$.
Thus the map $f_{D,\alpha}$ is a torus fibration, called \emph{the $\alpha$-twisting of $M$ at $D$}.
\end{definition}

\begin{lemma}
\label{lemma::fibre-pres-homeo-a-twisting}
Let $f:M\rightarrow S^2$ be a torus fibration and $f_{D,\alpha}$ be an $\alpha$-twisting of $f$ for some $\alpha:S^1\rightarrow \Homeo_0(\mathbb{T}^2)$ and disc $D\subset S^2$.
Suppose that $f$ has surjective monodromy homomorphisms.
Then $f$ and $f_{D,\alpha}$ are fibre-preserving homeomorphic.
\end{lemma}

\begin{proof}
See Proposition 2.1 in \cite{funar2022singular}.
\end{proof}

\begin{proof}[Proof of Theorem \ref{theorem::fibre-preserving equivalence}]
By Theorem \ref{theorem::global monodromy fibre sum Hurwitz equivalent},
we suppose that $(\widetilde{M}_1, \widetilde{f}_1)$ and $(\widetilde{M}_2, \widetilde{f}_2)$ have the same branch set $\mathcal{B}=\{p_1,\ldots,p_{l+m}\}$.
Taking the base point $p\in S^2\setminus\mathcal{B}$ and loops $\gamma_1,\ldots,\gamma_{l+m}$ based at $p$, we further suppose that the global monodromies of $(\widetilde{M}_1, \widetilde{f}_1)$ and $(\widetilde{M}_2, \widetilde{f}_2)$
determined by $p$, $\gamma_1,\ldots,\gamma_{l+m}$
coincide, say
\begin{equation*}
(g_1,\ldots, g_l)\bullet (\phi_1,\ldots,\phi_m).
\end{equation*}

Since the compatibility of a locally fibre-preserving homeomorphism does not depend on the global monodromy, we may assume that there exists the permutation $\sigma\in S_m$ such that
\begin{itemize}[-]
\item $\widetilde{f}_1^{-1}(p_j)$ and $\widetilde{f}_2^{-1}(p_j)$ are locally fibre-preserving homeomorphic compatible with the global monodromies, for $j=1,\ldots,l$;

\item $\widetilde{f}_1^{-1}(p_j)$ and $\widetilde{f}_2^{-1}(p_{l+\sigma(j-l)})$ are locally fibre-preserving homeomorphic compatible with the global monodromies, for $j=l+1,\ldots,l+m$.
\end{itemize}

Let $G=\mathbb{Z}\epsilon_1 + \ldots + \mathbb{Z}\epsilon_{l+m}$ be the free group of rank $l+m$.
Consider the tuples
\begin{align*}
&\Big(
(\epsilon_1,g_1),\ldots,(\epsilon_l,g_l),(\epsilon_{l+1},\phi_1),\ldots,(\epsilon_{l+m},\phi_m)
\Big),\\
&\Big(
(\epsilon_1,g_1),\ldots,(\epsilon_l,g_l),(\epsilon_{l+\sigma(1)},\phi_1),\ldots,(\epsilon_{l+\sigma(m)},\phi_m)
\Big)
\end{align*}
in $G\times \SL(2,\mathbb{Z})$.
The swappability of the global monodromy implies that one tuple can be transformed into the other by elementary transformations.
Thus, using a fibre-preserving homeomorphism, we may suppose that $\sigma$ is the identical permutation.

Consider the locally fibre-preserving homeomorphisms $\Psi_{S,j}:U_j^{(1)}\rightarrow U_j^{(2)}$, $\Psi_{M,j}:\widetilde{f}_1^{-1}(U_j^{(1)})\rightarrow \widetilde{f}_2^{-1}(U_j^{(2)})$
with sufficiently small neighbourhoods $U_j^{(i)}$ of $p_j$, for $j=1,\ldots,l+m$ and $i=1,2$.
They extend to locally fibre-preserving homeomorphisms
$\Psi'_{S,j}:U_j^{(1)}\cup \beta_j^{(1)}\rightarrow U_j^{(2)}\cup\beta_j^{(2)}$,
$\Psi'_{M,j}:\widetilde{f}_1^{-1}(U_j^{(1)}\cup \beta_j^{(1)})\rightarrow \widetilde{f}_2^{-1}(U_j^{(2)}\cup\beta_j^{(2)})$
where, for $i=1,2$, $\beta_j^{(i)}$ is a path joining $p$ to some point $d_j^{(i)}\in \partial U_j^{(i)}$ such that $\beta_1^{(i)},\ldots,\beta_{l+m}^{(i)}$, $U_1^{(i)},\ldots,U_{l+m}^{(i)}$ are disjoint away from $p$ and $d_1^{(i)},\ldots,d_{l+m}^{(i)}$.

The mapping class represented by $\varphi_j = \Psi'_{M,j}\mid_{\widetilde{f}_1^{-1}(p)}$ satisfies
$[\varphi_j]g_j=g_j[\varphi_j]$ if $j=1,\ldots,l$ or
$[\varphi_j]\phi_{j-l}=\phi_{j-l}[\varphi_j]$ otherwise.
All singular fibres are locally symmetric.
Set $\Gamma_i=\bigcup_j U_j^{(i)} \cap \beta_j^{(i)}$ for $i=1,2$.
Therefore we obtain a fibre-preserving homeomorphism
$\Psi_S:\Gamma_1\rightarrow \Gamma_2$,
$\Psi_M:\widetilde{f}_1^{-1}(\Gamma_1)\rightarrow \widetilde{f}_2^{-1}(\Gamma_2)$.
One may further assume that $\Gamma_1=\Gamma=\Gamma_2$ without loss of generality.

It remains to prove that $\widetilde{f}_1\mid_{\Gamma}$
and $\widetilde{f}_2\mid_{\Gamma}$
extend to the unique torus fibration over the complementary disc of $\Gamma$ within $S^2$, up to fibre-preserving homeomorphism,
but this follows from Lemma \ref{lemma::fibre-pres-homeo-a-twisting}.
\end{proof}

%% file: Section3/singular_fibr.tex
\subsection{Singular fibrations and singularities}
\label{subsection::singular fibration}

This subsection introduces \emph{singular fibrations} and illustrates Corollary \ref{corollary::main_result_singular_fibrations}.

Let $f:M^4\rightarrow S^2$ be a smooth map between a connected closed oriented $4$-manifold $M^4$ and the $2$-sphere with finitely many critical points, with generic fibre $F^2$.
Church and Timourian proved that each singularity $p$ of $f$ is \emph{cone-like}, i.e. the singularity $p$ admits a cone neighbourhood in the singular fibre $V=f^{-1}(f(p))$; see \cite[Lemma 2.1 and (Lemma) 2.4]{church1974differentiable} and also see \cite[p.835-836]{funar2011global}.

Isolated singularities are separated. In fact, there exist arbitrarily small adapted neighbourhoods of cone-like singularities, as introduced by King in \cite[p.396]{king1978topological}.
An \emph{adapted neighbourhood} around a singularity $p\in M^4$ is a compact neighbourhood $Z^4\subset M^4$ satisfying the following:

\begin{enumerate}[1)]
    \item The restriction $f|_{Z^4}:Z^4\rightarrow D^2$ is a proper map onto a disk $D^2\subset S^2$;
    
    \item The fibre $f^{-1}(x)$ is transversal to $\partial Z^4$ for each $x\in int(D^2)$ and $E=f^{-1}(S^1)\cap Z^4\subset \partial Z^4$;
    
    \item Set $V=f^{-1}(f(p))$ and $K=V\cap \partial Z^4$. Then $N(K)=f^{-1}(D_0^2) \cap \partial Z^4$ is a tubular neighbourhood of $K$ within $\partial Z^4$ endowed with a trivialization $\theta: N(K)\rightarrow K \times D_0^2$ induced by $f$, where $D_0^2\subset D^2$ is a sufficiently small disk containing $f(p)$.
    
    \item The composition $f_K = r\circ f: \partial Z^4\setminus K \rightarrow D^2 \rightarrow S^1$ is a locally trivial fibration over $S^1$, where $r$ is the radical projection;
    
    \item The data $(\partial Z^4, K, f_K, \theta)$ is an open book decomposition.
\end{enumerate}
It is equivalent to the date $(f_Z,\Phi)$ satisfying the following:

\begin{enumerate}[1)]
\item The map $f_Z:Z^4\rightarrow D^2$ is proper and induced by $f$. Set $V=f_Z^{-1}(f_Z(p))$, $K=V\cap \partial Z^4$ and $E^3=f_Z^{-1}(S^1) \subset \partial Z^4$. Then the restriction $f_Z:E^3\rightarrow S^1$ is a fibration with fibre $F_p^2$.

\item The flow $\Phi$ on $Z^4$ is continuous along directions parallel to $D^2$ such that
\begin{enumerate}[i),topsep=0pt,itemsep=-1ex,partopsep=1ex,parsep=1ex]
    \item $f(\Phi(z,d)) = f(z)+d$ for $z\in Z^4$ and $d\in D^2$ when both sides are within $Z^4$;
    \item the mapping $(x,t)\mapsto \Phi(x, -t f_Z(x))$ is a homeomorphism from $E^3\times [0,1)$ to $Z^4\setminus V$;
    \item there exists a vanishing compact subset $\mathcal{A}\subset E^3$ such that $x\mapsto \Phi(x, -f_Z(x))$ induces a homeomorphism from $E^3\setminus \mathcal{A}$ to $V\setminus p$ and sends $\mathcal{A}$ to $p$. 
\end{enumerate}
\end{enumerate}

King proved that, for the fibration of a manifold $M^m$ in dimension $m\neq 4,5$, one can always find adapted neighbourhoods for singularities diffeomorphic to the $m$-disk.
In dimension $4$, however, adapted neighbourhoods can only be supposed to be contractible.
We call a singularity \emph{regular} if it admits an arbitrarily small adapted neighbourhood which is diffeomorphic to the $4$-disk.

\begin{definition}
A smooth map $f:M^4\rightarrow S^2$ between a connected closed oriented $4$-manifold and the $2$-sphere is a \emph{singular fibration} if it has only finitely many critical points, all of them being regular.
\end{definition}

The binding $K\subset \partial Z^4$ of an open book decomposition is a fibered link.
Each fibre of $f_K$ is a surface that has the boundary $K$ and is homotopic to the local Milnor fibre $F_p^2$.
It is proved in \cite[Theorem 1]{king1978topological} that the local mapping torus $E^3$ and the vanishing compact subset $\mathcal{A}\subset E^3$ up to isotopy form a complete invariant of the adapted neighbourhood up to fibre-preserving homeomorphism.
In particular, if a singular fibre contains only one singularity and the fibre monodromy is given, then the singular fibre is determined by the isotopy class of local Milnor fibre, up to fibre-preserving homeomorphism.

In general, there could be many singularities in a singular fibre, say $p_1,\ldots,p_n$. The horizontal homeomorphisms given by disjoint adapted neighbourhoods reveal that the local Milnor fibres $F_{p_1}^2,\ldots,F_{p_n}^2$ are disjoint compact subsurfaces embedded in the generic fibre $F^2$ of the fibration.
The fibre monodromy around the singular fibre is a mapping class of the generic fibre $F^2$, denoted by $\phi_{f^{-1}(f(p_i))}$. The inclusions $\iota_i:F_{p_i}^2\hookrightarrow F$ induce the homomorphisms $\Mod(F_{p_i}^2)\rightarrow \Mod(F^2)$ which send the local monodromies $\phi_{F_{p_i}^2}$ of the mapping tori $E^3\rightarrow S^1$ to mapping classes of the generic fibre. Therefore
\begin{equation}
    \label{equation::local monodromies and local-global monodromy}
    \phi_{f^{-1}(f(p_i))} = \iota_{1,*}(\phi_{F_{p_1}^2}) \circ \ldots \circ \iota_{n,*}(\phi_{F_{p_n}^2}),
\end{equation}
which does not depend on the order.
Furthermore, the following should be well-known.

\begin{lemma}
In a singular fibration $f:M^4\rightarrow S^2$, each local Milnor fibre of an adapted neighbourhood diffeomorphic to the $4$-disk is connected with a non-empty boundary.
\end{lemma}
\begin{proof}
Suppose that $p$ is a singularity of the singular fibration $f:M^4\rightarrow S^2$.
If we assume that the binding link $K$ of a singularity is vacuous,
the adapted neighbourhood implies a locally trivial fibre bundle $S^3\rightarrow S^1$, which is a contradiction.
Since the completion of the local Milnor fibre has a non-empty boundary, the reduced cohomology group $\tilde{H}^2(\overline{F_p^2})$ is trivial.
We use Alexander duality and obtain that $\tilde{H}_0(S^3\setminus\overline{F_p^2})$ is trivial.
Hence $F_p^2$ is connected.
\end{proof}

\begin{definition}
A continuous map $g_1:X_1\rightarrow Y_1$ is \emph{locally topologically equivalent} at $x_1\in X_1$ to a continuous map $g_2:X_2\rightarrow Y_2$ at $x_2\in X_2$ if there exist sufficiently small open neighbourhoods $U_1$ of $x_1$, $U_2$ of $x_2$, $V_1$ of $g_1(x_1)$, $V_2$ of $g_2(x_2)$ and homeomorphisms $\alpha:U_1\rightarrow U_2$, $\beta:V_1\rightarrow V_2$ such that $\beta\circ g_1\mid_{U_1} = g_2\circ \alpha\mid_{U_1}$.
\end{definition}

A point at which $f$ fails to be locally topologically equivalent to the projection $\mathbb{R}^4 \rightarrow \mathbb{R}^2$ is called a \emph{branch point},
which is necessarily a singularity.
Church and Lamotke have shown that a local Milnor fibre is diffeomorphic to the $2$-disk if and only if the associated singularity is not a branch point; see \cite[Proposition p.151]{church1975non}.
We conclude that, up to fibre-preserving homeomorphism, one may assume that a torus singular fibration has no local Milnor fibre of genus $0$ with only $1$ boundary component.

\subsubsection{Local Milnor fibre of genus zero}

When the local Milnor fibre $F_p^2$ is a genus zero surface with $r\ge 2$ boundary components,
then $E^3\cong \partial Z^4\setminus K$ is the mapping torus of some mapping class $\phi_{F_p^2}$ that is identical on boundary, denoted by $\mathcal{M}_{\phi_{F_p^2}}$.
The group of mapping classes identical on the boundary, denoted by $\Mod^*(F_p^2)$, is generated by Dehn twists along the following loops; see \cite{wajnryb1999elementary}.
\begin{itemize}[-]
    \item Loops $\delta_{i,j}$, $2\le i<j\le r$, that each separates two boundary components from the others.
    \item Peripheral loops $\alpha_2,\ldots,\alpha_r$, that are parallel to the latter $r-1$ boundary components.
\end{itemize}
The peripheral loops are mutually disjoint and they keep away from the loops $\delta_{i,j}$. Therefore, $\phi_{F_p^2}$ is the composition of the product of (positive and negative) Dehn twists along peripheral loops and a mapping class $\varphi_{F_p^2}$ generated by the Dehn twists along the rest loops, denoted by
\begin{equation*}
    \phi_{F_p^2} = (\prod\limits_{i=2}^r T_{\alpha_i}^{u_i}) \varphi_{F_p^2}
\end{equation*}
with $u_2,\ldots,u_r\in\mathbb{Z}$.
The following shows a necessary property for local Milnor fibres of genus zero in a fibration $f:M^4\rightarrow S^2$.

\begin{lemma}
\label{lemma::genus_zero_surface_mapping_class_peripheral_loops}
Let $f:M^4\rightarrow S^2$ be a smooth map between a connected closed oriented $4$-manifold $M^4$ and the $2$-sphere. Let $p\in M^4$ be an isolated singularity. Given a contractible adapted neighbourhood of $p$, if the local Milnor fibre $F_p^2$ is a genus zero surface with $r\ge 2$ boundary components and the local monodromy is given by $\phi_{F_p^2}=\prod\limits_{i=2}^r T_{\alpha_i}^{u_i}$ with $u_2,\ldots,u_r\in\mathbb{Z}$, then $u_i=\pm 1$, $i=2,\ldots,r$.
\end{lemma}
\begin{proof}
The first homology group of $F_p^2$ is isomorphic to $\mathbb{Z}^{r-1}$ and generated by the cycles around boundary components, but excluding the first component.
Therefore, $\phi_{F_p^2,*}=id_{H_1(F_p^2,\mathbb{Z})}$ and the homology group $H_1(\mathcal{M}_{\phi_{F_p^2}},\mathbb{Z}) = H_1(F_p^2,\mathbb{Z}) \rtimes_{\phi_{F_p^2},*} \langle [\gamma] \rangle$ is isomorphic to $\mathbb{Z}^r$, where $\gamma$ is the closed curve in the mapping torus induced by a fixed point on the first boundary component of $F_p^2$.

We write $H_1(\mathcal{M}_{\phi_{F_p^2}},\mathbb{Z})=\langle a_2, \ldots, a_r, t \rangle$.
The boundary $\partial Z^4$ of the adapted neighbourhood is the union of the mapping torus $\mathcal{M}_{\phi_{F_p^2}}$ and $r$ more solid tori, which is a homology $3$-sphere.
The inclusion mapping the connected components of the intersection to the mapping torus derives from (positive or negative) powers of the Dehn twist along peripheral loops, which are denoted by $T_{\alpha_1}^{u_1},\ldots,T_{\alpha_r}^{u_r}$ respectively. By Mayer-Vietoris we have
\begin{equation*}
    H_2(\partial Z^4,\mathbb{Z}) \longrightarrow H_1(\mathbb{T}^2,\mathbb{Z})^r \xrightarrow{~\tau~} H_1(\mathcal{M}_{\phi_{F_p^2}},\mathbb{Z}) \oplus H_1(S^1\times D^1)^r \longrightarrow H_1(\partial Z^4,\mathbb{Z})
\end{equation*}
where $\tau$ is an isomorphism. After the choice of the natural basis, the corresponding $(2r)\times (2r)$-matrix is given by
\begin{equation*}
    A = 
    \begin{bmatrix}
    -1 & 0 & 1 & u_2 \\
    -1 & 0 &   &     & 1 & u_3 \\
    -1 & 0 &   &     &   &     & 1 & u_4 \\
    \cdot &    &   &     &   &     &   &     & \cdot \\
    \cdot &    &   &     &   &     &   &     & & \cdot \\
    \cdot &    &   &     &   &     &   &     & & & \cdot \\
    -1 & 0 &   &     &   &     &   &     & \cdot & \cdot & \cdot & 1 & u_r\\
    0 & 1 & 0 & 1 & 0 & 1 & 0 & 1 & \cdot & \cdot & \cdot & 0 & 1 \\
    1 & 0 \\
      &   & 1 & 0 \\
      &   &   &   & 1 & 0 \\
      &   &   &   &   &   & 1 & 0 \\
      &    &   &     &   &     &   &     & \cdot \\
      &    &   &     &   &     &   &     & & \cdot \\
      &    &   &     &   &     &   &     & & & \cdot \\
      & & & & & & & & & & & 1 & 0
    \end{bmatrix}
\end{equation*}
satisfying $det(A)=\pm 1$. It follows that $u_2\cdots u_r=\pm 1$.
\end{proof}

In particular, we have the following consequence.
\begin{corollary}
\label{corollary::singular fibration milnor genus 0 and two bds}
Let $f:M^4\rightarrow S^2$ be a smooth map between a connected closed oriented $4$-manifold and the $2$-sphere. Given a contractible adapted neighbourhood of a singularity, if the local Milnor fibre is a genus zero surface with exactly two boundary components, the local monodromy is either a positive or a negative Dehn twist.
\end{corollary}
\begin{proof}
In this case, the local Milnor fibre $F_p^2$ is an annulus whose mapping class group is generated by the Dehn twist along the unique peripheral loop. By Lemma \ref{lemma::genus_zero_surface_mapping_class_peripheral_loops} we have $u_2=\pm 1$. Hence $\phi_{F_p^2}$ is either the positive or the negative Dehn twist.
\end{proof}

\subsubsection{Local Milnor fibre of genus one}

We consider the case when the local Milnor fibre for a contractible adapted neighbourhood of a cone-like singularity in $f:M^4\rightarrow S^2$ is a torus with $r\ge 1$ disks removed, say $F_p^2=\mathbb{T}^2\setminus (D_1\sqcup \ldots \sqcup D_r)$.
Again, let $\phi_{F_p^2}\in \Mod(F_p^2)$ be the local monodromy.
By the Mayer–Vietoris sequence on $\partial Z^4$ we have
\begin{equation*}
    H_2(\partial Z^4,\mathbb{Z})\longrightarrow
    H_1(\mathbb{T}^2,\mathbb{Z})^r \longrightarrow
    H_1(\mathcal{M}_{\phi_{F_p^2}},\mathbb{Z}) \oplus H_1(S^1\times D^1,\mathbb{Z})^r
    \longrightarrow H_1(\partial Z^4,\mathbb{Z}).
\end{equation*}
Since the boundary $\partial Z^4$ is a homology $3$-sphere,$H_1(\mathcal{M}_{\phi_{F_p^2}},\mathbb{Z})$ is isomorphic to $\mathbb{Z}^r$.

Now we compute the homology group $H_1(\mathcal{M}_{\phi_{F_p^2}},\mathbb{Z})$ of the mapping torus. Write $\mathcal{M}_{\phi_{F_p^2}}$ as the union of $A=F_p^2\times I_1$ and $B=F_p^2\times I_2$ and take the inclusion maps $i:A\cap B\hookrightarrow A$, $j:A\cap B\hookrightarrow B$, $k:A\hookrightarrow \mathcal{M}_{\phi_{F_p^2}}$ and $l:B\hookrightarrow \mathcal{M}_{\phi_{F_p^2}}$. By Mayer–Vietoris we have
\begin{align*}
    \longrightarrow &H_1(A\cap B,\mathbb{Z}) \xrightarrow{(i_*,j_*)}
    H_1(A,\mathbb{Z})\oplus H_1(B,\mathbb{Z}) \xrightarrow{k_*-l_*}
    H_1(\mathcal{M}_{\phi_{F_p^2}},\mathbb{Z}) \xrightarrow{\partial_*}\\
    &H_0(A\cap B,\mathbb{Z}) \xrightarrow{(i_*,j_*)}
    H_0(A,\mathbb{Z})\oplus H_0(B,\mathbb{Z}) \xrightarrow{k_*-l_*}
    H_0(\mathcal{M}_{\phi_{F_p^2}},\mathbb{Z}) \longrightarrow 0.
\end{align*}
Notice that $im\partial_*$ is isomorphic to $ker(H_0(A\cap B,\mathbb{Z})\xrightarrow{(i_*,j_*)}H_0(A,\mathbb{Z})\oplus H_0(B,\mathbb{Z}))\simeq \mathbb{Z}$.
To ensure that $H_1(\mathcal{M}_{\phi_{F_p^2}},\mathbb{Z})$ is isomorphic to $\mathbb{Z}^r$, we require that $ker\partial_*\simeq \mathbb{Z}^{r-1}$ and therefore 
\begin{equation*}
    im(H_1(A\cap B,\mathbb{Z}) \xrightarrow{(i_*,j_*)}H_1(A,\mathbb{Z})\oplus H_1(B,\mathbb{Z}))\simeq \mathbb{Z}^{3+r}.
\end{equation*}

\begin{lemma}
\label{lemma::singular fibration milnor genus 1 and 1 bd}
Let $f:M^4\rightarrow S^2$ be a smooth map between a connected closed oriented $4$-manifold $M^4$ and the $2$-sphere. Let $p$ be a singularity of $f$ with a contractible adapted neighbourhood.
Suppose that the local Milnor fibre $F_p^2$ is a torus with a disk removed and consider the inclusion $\iota:F_p^2\hookrightarrow \mathbb{T}^2$. Then the binding link $K\subset \partial Z^4\cong S^3$ is either the trefoil knot or the figure-eight knot. Furthermore, the local monodromy $\phi_{F_p^2}$ induces a mapping class of the torus $\iota_{*}(\phi_{F_p^2}) \in \Mod(\mathbb{T}^2)\simeq \SL(2,\mathbb{Z})$ which is conjugate to one of
\begin{equation*}
    \begin{bmatrix}0 & -1\\1 & 1\end{bmatrix},
    \begin{bmatrix}1 & 1\\-1 & 0\end{bmatrix},
    \begin{bmatrix}2 & 1\\1 & 1\end{bmatrix}.
\end{equation*}
\end{lemma}
\begin{proof}
We only prove the assertion of the local monodromy and a complete proof has been introduced by Burde and Zieschang (see Proposition 5.14 in \cite{burde2008knots}).

The mapping class group of $F_p^2$ is generated by the Dehn twists along two intersecting loops $\alpha,\beta$ and the Dehn twist along the peripheral loop $\delta$.
With a careful arrangement, the peripheral loop is away from the others and therefore the local monodromy is the composition $\phi_{F_p^2}=T_{\delta}^u \circ \varphi_{F_p^2}$ with $u\in\mathbb{Z}$ and $\varphi_{F_p^2}$ generated by the Dehn twists along $\alpha,\beta$.
Thus, along the inclusion $\iota:F_p^2\hookrightarrow \mathbb{T}^2$, the pushforward $\iota_*(\phi_{F_p^2})$ is equal to the pushforward $\iota_*(\varphi_{F_p^2})$.
Fix the isomorphism between $\Mod(\mathbb{T}^2)$ and $\SL(2,\mathbb{Z})$ such that the induced homomorphism $\Mod(F_p^2)\rightarrow \SL(2,\mathbb{Z})$ sends $T_{\alpha}$ (resp. $T_{\beta}$) to
\begin{equation*}
    \begin{bmatrix}
    1 & 0\\ 1 & 1
    \end{bmatrix}
    \text{~(resp. }
    \begin{bmatrix}
    1 & -1\\ 0 & 1
    \end{bmatrix}
    \text{).}
\end{equation*}
Suppose that $\iota_*(\phi_{F_p^2})\in \Mod(\mathbb{T}^2)$ is expressed by $A=\begin{bmatrix}
a & b\\
c & d
\end{bmatrix}\in \SL(2,\mathbb{Z})$

We take the basis of the homology group $H_1(F_p^2)$ consisting of the cycles which are parallel with $\alpha$ and $\beta$, which further determines the bases of $H_1(A,\mathbb{Z})$, $H_1(B,\mathbb{Z})$ and $H_1(A\cap B,\mathbb{Z})$.
The pushforward $\phi_{F_p^2,*}:H_1(F_p^2,\mathbb{Z})\rightarrow H_1(F_p^2,\mathbb{Z})$ is again expressed by $A$.
The homomorphism $H_1(A\cap B,\mathbb{Z}) \xrightarrow{(i_*,j_*)}H_1(A,\mathbb{Z})\oplus H_1(B,\mathbb{Z})$ is an isomorphism whose corresponding $4\times 4$-matrix is given by
\begin{equation*}
    \begin{bmatrix}
    I & A\\
    I & I
    \end{bmatrix}.
\end{equation*}
satisfying $det(I-A)=\pm 1$. Hence $a+d=1$ or $3$.
\end{proof}

Conversely, we do have a connected closed oriented $4$-manifold $M^4$ with a singular fibration $f:M^4\rightarrow S^2$ that has the singularities as desired.
Both the trefoil knot and the figure-eight knot are defined by the links of polynomial maps $\mathbb{R}^4\rightarrow \mathbb{R}^2$ with an isolated critical point at $0$.
An explicit realisation of the trefoil knot was first given by Brauner (see \cite{brauner1928verhalten}) who constructs a complex polynomial
\begin{equation*}
    (f_{\text{Brauner}}:\mathbb{C}^2\rightarrow \mathbb{C}): (u,v)\mapsto u^2-v^3.
\end{equation*}
Perron found the first realisation of the figure-eight knot in \cite{perron1982noeud}.

We end with the proof of Corollary \ref{corollary::main_result_singular_fibrations}.

\begin{proof}[Proof of Corollary \ref{corollary::main_result_singular_fibrations}]
Without loss of generality, we assume that there does not exist any local Milnor fibre of genus $0$ with only $1$ boundary components.
By Corollary \ref{corollary::singular fibration milnor genus 0 and two bds} and Lemma \ref{lemma::singular fibration milnor genus 1 and 1 bd}, the type of singularities $\mathcal{O}(f_1)=\mathcal{O}(f_2)$ consists of simple conjugacy classes of $\SL(2,\mathbb{Z})$.

All singular fibres are locally symmetric.
If the local Milnor fibre is an annulus, then a mapping class that commutes with the fibre monodromy preserves this annulus up to isotopy.
If the local Milnor fibre is a torus with a disc removed, then no mapping class changes the local Milnor fibre up to isotopy.

Any pair of singular fibres with conjugate fibre monodromies
has local Milnor fibres compatible with their fibre monodromies,
%i.e. using the difference between fibre monodromies,
so they have the same local Milnor fibre up to isotopy.
Therefore, there exists a local fibre-preserving homeomorphism compatible with their fibre monodromies.

The corollary follows from Theorem \ref{thmx::C} and Theorem \ref{theorem::fibre-preserving equivalence}.
\end{proof}

%% file: Section4/short.tex
\section{Theorem of R. Livn\'e, complement and extension}
\label{section::Livne}

In this section, $G$ is the modular group $\PSL(2,\mathbb{Z})\simeq \mathbb{Z}/2\mathbb{Z} \ast \mathbb{Z}/3\mathbb{Z}$, which we represent as $\langle a,b\mid a^3=b^2=1\rangle$. Each element in $G$ has the unique \emph{reduced form} as a word in $\{a,a^2,b\}$ where $b$'s and powers of $a$ appear alternatively. The \emph{length} of an element $g\in G$ is defined as the length of its reduced form, denoted by $l(g)$.

Recall that elements in
\begin{equation*}
    \mathcal{S}=\{a,a^2,b,s_0=a^2b,s_1=aba,s_2=ba^2,t_0=ba,t_1=a^2ba^2,t_2=ab\}\subset G
\end{equation*}
are ``\emph{short}''; the rest conjugates of short elements in $G$ are called ``\emph{long}'', which are expressed by $Q^{-1}a^{\epsilon}Q$, $Q^{-1}bQ$ or $Q^{-1}a^{\epsilon}ba^{\epsilon}Q$ with $\epsilon=1,2$ and $l(Q)\ge 1$. The following diagram shows all conjugates of short elements and their conjugates with $a,a^2$ and $b$.

\begin{center}
\begin{tikzpicture}[node distance={16mm}, thick]
\node (a) {$a$};
\node (aa) [right of=a] {$a^2$};
\node (b) [right of=aa] {$b$};
\node (s0) [right of=b] {$a^2b$};
\node (s1) [right of=s0] {$aba$};
\node (s2) [right of=s1] {$ba^2$};
\node (t0) [right of=s2] {$ba$};
\node (t1) [right of=t0] {$a^2ba^2$};
\node (t2) [right of=t1] {$ab$};
\draw[->] (a) to [out=135,in=45,looseness=6.5] node[above] {$a$} (a);
\node (3) [below of=a, node distance={15mm}] {$Q^{-1}aQ$};
\draw[->] (a) -- node[right] {$b$} (3);
\draw[->] (aa) to [out=135,in=45,looseness=4.5] node[above] {$a$} (aa);
\node (4) [below of=aa, node distance={15mm}] {$Q^{-1}a^2Q$};
\draw[->] (aa) -- node[right] {$b$} (4);
\draw[->] (b) to [out=135,in=45,looseness=6.5] node[above] {$b$} (b);
\node (5) [below of=b, node distance={15mm}] {$Q^{-1}bQ$};
\draw[->] (b) -- node[right] {$a$} (5);
\draw[->] (s0) -- node[above] {$a$} (s1);
\draw[->] (s1) -- node[above] {$a$} (s2);
\draw[->] (s2) to [out=90,in=90,looseness=0.5] node[above] {$a$} (s0);
\draw[<->] (s0) to [out=270,in=270,looseness=0.5] node[above] {$b$} (s2);
\node (1) [below of=s1, node distance={25mm}] {$Q^{-1}abaQ$};
\draw[->] (s1) to [out=315,in=45,looseness=0.5] node[right] {$b$} (1);
\draw[->] (t0) -- node[above] {$a$} (t1);
\draw[->] (t1) -- node[above] {$a$} (t2);
\draw[->] (t2) to [out=90,in=90,looseness=0.5] node[above] {$a$} (t0);
\draw[<->] (t0) to [out=270,in=270,looseness=0.5] node[above] {$b$} (t2);
\node (2) [below of=t1, node distance={25mm}] {$Q^{-1}a^2ba^2Q$};
\draw[->] (t1) to [out=315,in=45,looseness=0.5] node[right] {$b$} (2);
\end{tikzpicture} 
\end{center}

Note that there are nice circuits along $s_0,s_1,s_2$ and $t_0,t_1,t_2$. In fact, we will see a lot of symmetric properties on them. For convenience, the subscripts are regarded as elements in $\mathbb Z/3\mathbb Z$ and represented by $0,1,2$ without further explanations.

Recall that elements in
\[
    \mathcal{S}_2=\mathcal{S}\cup 
    \{
    bab, ba^2b, a^2ba, aba^2,
    a^2bab, ababa, baba^2, ba^2ba, a^2ba^2ba^2, aba^2b
    \}
\]
are ``almost short''; the rest conjugates of almost short elements in $G$ are called ``almost long'', which are expressed by
\[
    Q^{-1}ba^{\epsilon}bQ,
    Q^{-1}a^{\epsilon}ba^{\epsilon}Q,
    Q^{-1}a^{\epsilon}ba^{-\epsilon}Q \text{ or }
    Q^{-1}a^{\epsilon}ba^{\epsilon}ba^{\epsilon}Q
\]
with $\epsilon=1,2$ and $l(Q)\ge 1$. The almost short elements correspond to six conjugacy classes of $G$, five of which have been illustrated and the following is the last one.

\begin{center}
\begin{tikzpicture}[node distance={16mm}, thick]
\node (aabab)  {$a^2bab$};
\node (ababa) [right of=aabab] {$ababa$};
\node (babaa) [right of=ababa] {$baba^2$};
\node (baaba)  [below of=aabab] {$ba^2ba$};
\node (aabaabaa) [right of=baaba] {$a^2ba^2ba^2$};
\node (abaab) [right of=aabaabaa] {$aba^2b$};
\node (1) [left of=aabab, node distance={25mm}] {$Q^{-1}ababaQ$};
\node (2) [right of=abaab, node distance={25mm}] {$Q^{-1}a^2ba^2ba^2Q$};

\draw[->] (aabab) -- node[above] {$a$} (ababa);
\draw[->] (ababa) -- node[above] {$a$} (babaa);
\draw[->] (babaa) to [out=270,in=270,looseness=0.5] node[above] {$a$} (aabab);
\draw[->] (baaba) -- node[below] {$a$} (aabaabaa);
\draw[->] (aabaabaa) -- node[below] {$a$} (abaab);
\draw[->] (abaab) to [out=90,in=90,looseness=0.5] node[below] {$a$} (baaba);

\draw[<->] (aabab) to [out=225,in=135,looseness=0.5] node[right] {$b$} (baaba);
\draw[<->] (babaa) to [out=315,in=45,looseness=0.5] node[right] {$b$} (abaab);
\draw[->] (ababa) to [out=90,in=90,looseness=0.5] node[below] {$b$} (1);
\draw[->] (aabaabaa) to [out=270,in=270,looseness=0.5] node[above] {$b$} (2);
\end{tikzpicture} 
\end{center}

Recall that elementary transformations $R_i$, $1\le i\le n-1$ on $n$-tuples in $G$ send $(g_1,\ldots g_n)$ to $(g_1,\ldots,g_{i-1},g_{i+1},g_{i+1}^{-1}g_ig_{i+1},g_{i+2},\ldots,g_n)$ respectively. The inverse of $R_i$ is given by $R_i^{-1}$ sending $(g_1,\ldots,g_n)$ to $(g_1,\ldots,g_{i-1},g_ig_{i+1}g_i^{-1},g_i,g_{i+2},\ldots,g_n)$.
Both $R_i$ and $R_i^{-1}$ are called elementary transformations.
Especially, we will neither apply $R_i$ if $g_i=1$ nor apply $R_i^{-1}$ if $g_{i+1}=1$, but use $R_i^{-1}$ and $R_i$ instead respectively to avoid troubles.

Elementary transformations introduce many elegant substitutions for pairs of short elements. Here we list some substitutions in the following graphs for readers unfamiliar with them.

\begin{tikzpicture}[node distance={20mm}, thick]
\node (s0s2) {$(s_0,s_2)$};
\node (s1s0) [below right of=s0s2] {$(s_1,s_0)$};
\node (s2s1) [above right of=s0s2] {$(s_2,s_1)$};
\draw[-] (s0s2) -- (s1s0);
\draw[-] (s0s2) -- (s2s1);
\draw[-] (s1s0) -- (s2s1);

\node (s0a) [right of=s2s1] {$(s_i,a)$};
\node (as1) [right of=s0a] {$(a,s_{i+1})$};
\draw[-] (s0a) -- (as1);
\node (s0A) [below of=s0a, node distance={5mm}] {$(s_i,a^2)$};
\node (As2) [right of=s0A] {$(a^2,s_{i-1})$};
\draw[-] (s0A) -- (As2);
\node (s0b) [below of=s0A, node distance={10mm}] {$(s_0,b)$};
\node (bs2) [right of=s0b] {$(b,s_2)$};
\node (s2b) [below of=s0b, node distance={5mm}] {$(s_2,b)$};
\node (bs0) [right of=s2b] {$(b,s_0)$};
\draw[-] (s0b) -- (bs2);
\draw[-] (s2b) -- (bs0);

\node (t0a) [right of=as1] {$(t_i,a)$};
\node (at1) [right of=t0a] {$(a,t_{i+1})$};
\draw[-] (t0a) -- (at1);
\node (t0A) [below of=t0a, node distance={5mm}] {$(t_i,a^2)$};
\node (At2) [right of=t0A] {$(a^2,t_{i-1})$};
\draw[-] (t0A) -- (At2);
\node (t0b) [below of=t0A, node distance={10mm}] {$(t_0,b)$};
\node (bt2) [right of=t0b] {$(b,t_2)$};
\node (t2b) [below of=t0b, node distance={5mm}] {$(t_2,b)$};
\node (bt0) [right of=t2b] {$(b,t_0)$};
\draw[-] (t0b) -- (bt2);
\draw[-] (t2b) -- (bt0);

\node (t2t0) [right of=at1] {$(t_2,t_0)$};
\node (t0t1) [below right of=t2t0] {$(t_0,t_1)$};
\node (t1t2) [below left of=t0t1] {$(t_1,t_2)$};
\draw[-] (t0t1) -- (t1t2);
\draw[-] (t0t1) -- (t2t0);
\draw[-] (t1t2) -- (t2t0);

\node (s0t1) [below of=bs0, node distance={10mm}] {$(s_i,t_{i+1})$};
\node (t1s2) [right of=s0t1] {$(t_{i+1},s_{i-1})$};
\node (s0t2) [below of=s0t1, node distance={5mm}] {$(s_i,t_{i-1})$};
\node (t1s0) [right of=s0t2] {$(t_{i+1},s_i)$};
\draw[-] (s0t1) -- (t1s2);
\draw[-] (s0t2) -- (t1s0);
\end{tikzpicture} 

\begin{definition}
An $n$-tuple $(g_1,\ldots,g_n)$ in $G$ is said to be \emph{inverse-free} if, applying any finite sequence of elementary transformations, the resulting $n$-tuple satisfies the following requirements:
\begin{enumerate}[-,topsep=0pt,itemsep=-1ex,partopsep=1ex,parsep=1ex]
\item it contains no adjacent elements which are mutually inverse;
\item it contains no sub-triple of the form $(h,h,h)$ with $h^3=1$.
\end{enumerate}
\end{definition}

For instance, $(s_1,t_1)$, $(a,a^2)$, $(b,b)$, $(a,a,a)$ and their concatenations are not inverse-free.

\begin{theorem}[Livn\'e]
\label{theorem::Livne}
Let $g_1,\ldots, g_n$ be conjugates of $s_1$ such that $g_1\cdots g_n=1$. Then, the $n$-tuple $(g_1,\ldots,g_n)$ is Hurwitz equivalent to an $n$-tuple $(h_1,\ldots,h_n)$ with each $h_i$ short (i.e. the component $h_i$ is equal to one of $s_0$, $s_1$ and $s_2$).
\end{theorem}
Moishezon showed a proof of Theorem \ref{theorem::Livne} and introduced the following complement in \cite{moishezon1977complex}.
\begin{theorem}[Moishezon]
\label{theorem::Moishezon}
Let $h_1,\ldots, h_n$ be such that each of $h_i$, $i=1,\ldots,n$, is equal to one of $s_0$, $s_1$ and $s_2$ satisfying $h_1\cdots h_n=1$. Then, $n\equiv 0\pmod{6}$ and the $n$-tuple $(h_1,\ldots,h_n)$ is Hurwitz equivalent to $(s_0,s_2)^{n/2}$.
\end{theorem}

In this section, We first extend the above theorems for $(g_1,\ldots,g_n)$ with each $g_i$ conjugate to some short element, then we show a similar result for $(g_1,\ldots,g_n)$ when each $g_i$ is conjugate to some almost short element.

\subsection{Tuples of short elements}

Recall the set of short elements is $\mathcal{S}=\{a,a^2,b,s_0,s_1,s_2,t_0,t_1,t_2\}$.
We first show that an inverse-free tuple of short elements cannot contain both $s_i$ and $t_j$ for any $(i,j)\in (\mathbb Z/3 \mathbb Z)^2$.

\begin{proposition}
\label{proposition::ST}
Let $g_1,\ldots,g_n$ be short satisfying at most one of them is equal to one of $a$, $a^2$, $b$ and $g_1\cdots g_n=1$. Suppose that $(g_1,\ldots,g_n)$ is inverse-free. Then, either each of $g_i$, $i=1,\ldots,n$ is equal to one of $a,a^2,b,s_0,s_1,s_2$ or each of $g_i$, $i=1,\ldots,n$ is equal to one of $a,a^2,b,t_0,t_1,t_2$.
\end{proposition}
\begin{proof}
Assume that at least one of $g_1,\ldots,g_n$ is conjugate to $s_0$ and at least one of $g_1,\ldots,g_n$ is conjugate to $t_0$.
The substitution of $(s_k,t_{k+1})$, $(t_{k+1},s_{k-1})$ and the substitution of $(s_k,t_{k-1})$, $(t_{k+1},s_k)$ imply that $(g_1,\ldots,g_n)$ can be transformed by elementary transformations into $(h_1,\ldots,h_n)$ with $p,q\ge 1$, $p+q\in\{n-1,n\}$ such that $h_1\in\{a,a^2,b,s_0,s_1,s_2\}$, $h_i\in\{s_0,s_1,s_2\}$ for $i=2,\ldots,n-q$ and $h_i\in\{t_0,t_1,t_2\}$ for $i=n-q+1,\ldots,n$.
Let $\mathcal{A}$ be the set of elements in $\{h_1,\ldots,h_n\}$, $\mathcal{A}_s=\mathcal{A}\cap\{s_0,s_1,s_2\}$ and $\mathcal{A}_t=\mathcal{A}\cap\{t_0,t_1,t_2\}$. The inverse-freeness requires that $s_j$ and $t_j$ cannot appear together in $\mathcal{A}$.

\allowdisplaybreaks
Assume that $|\mathcal{A}_s|=1=|\mathcal{A}_t|$. Then the product $h_1\cdots h_n$ is expressed by $\tilde{h}s_j^ut_k^v$ with $\tilde{h}\in\{1,a,a^2,b\}$, $u,v\ge 1$ and $j\neq k$. To ensure that $(g_1,\ldots,g_n)$ is inverse-free, the product must be one of the following forms with $u,v\ge 1$.
\begin{align*}
\tilde{h}=1 ~~\Rightarrow~~
  & s_0^u t_1^v=(a^2b)^u(a^2ba^2)^v,\\
  & s_0^u t_2^v=(a^2b)^u(ab)^v,\\
  & s_1^u t_0^v=(aba)^u(ba)^v,\\
  & s_1^u t_2^v=(aba)^u(ab)^v,\\
  & s_2^u t_0^v=(ba^2)^u(ba)^v,\\
  & s_2^u t_1^v=(ba^2)^u(a^2ba^2)^v.\\
\tilde{h}=a ~~\Rightarrow~~
  & a s_0^u t_1^v=a(a^2b)^u(a^2ba^2)^v=b(a^2b)^{u-1}(a^2ba^2)^v,\\
  & a s_1^u t_2^v=a(aba)^u(ab)^v,\\
  & a s_2^u t_0^v=a(baa)^u(ba)^v.\\
\tilde{h}=a^2 ~~\Rightarrow~~
  & a^2 s_0^u t_2^v=a^2(a^2b)^u(ab)^v,\\
  & a^2 s_1^u t_0^v=a^2(aba)^u(ba)^v=ba(aba)^{u-1}(ba)^v,\\
  & a^2 s_2^u t_1^v=a^2(ba^2)^u(a^2ba^2)^v.\\
\tilde{h}=b ~~\Rightarrow~~
  & b s_0^u t_1^v=b(a^2b)^u(a^2ba^2b)^v,\\
  & b s_1^u t_0^v=b(aba)^u(ba)^v,\\
  & b s_1^u t_2^v=b(aba)^u(ab)^v,\\
  & b s_2^u t_1^v=b(ba^2)^u(a^2ba^2)^v=a^2(ba^2)^{u-1}(a^2ba^2)^v.    
\end{align*}
However, each of them cannot express $1$, which contradicts the fact that $h_1\cdots h_n=g_1\cdots g_n=1$.
\interdisplaylinepenalty=10000

Assume that $|\mathcal{A}_s|=1$ and $|\mathcal{A}_t|=2$. Pairs of the form $(t_j,t_{j+1})$ never appear since the substitutions of $(t_0,t_1)$, $(t_1,t_2)$ and $(t_2,t_0)$ imply a contradiction with the inverse-freeness. If $\tilde{h}\neq 1$, then the tuple $(h_1,\ldots,h_n)$ is expressed by $(\tilde{h}) \bullet (s_j)^u \bullet (t_{j-1})^v \bullet (t_{j+1})^w$ with $\tilde{h}\in\{a,a^2,b\}$, $u,v,w\ge 1$ and some $j$. The substitution of $(a,s_k)$, $(s_{k-1},a)$ and the substitution of $(a^2,s_k)$, $(s_{k+1},a^2)$ reveal that $\tilde{h}=b$. However, the substitution of $(b,s_0)$, $(s_2,b)$, the substitution of $(b,s_2)$, $(s_0,b)$ and the following substitutions
\begin{eqnarray*}
(b) \bullet (s_1)^u \bullet (t_0)^v \bullet (t_2)^w \longrightarrow 
(b) \bullet (t_0)^w \bullet (s_1)^u \bullet (t_0)^v \longrightarrow
(b) \bullet (t_2)^v \bullet (t_0)^w \bullet (s_1)^u
\end{eqnarray*}
further conclude that $\tilde{h}=1$. Thus, $g_1\cdots g_n=h_1\cdots h_n$ is expressed by $s_0^u t_2^v t_1^w$, $s_1^u t_0^v t_2^w$ or $s_2^u t_1^v t_0^w$ each of which cannot express $1$, which is a contradiction.

We have a similar argument for the case where $|\mathcal{A}_s|=2$ and $|\mathcal{A}_t|=1$. Hence either none of $g_1,\ldots,g_n$ is conjugate to $s_0$ or none of $g_1,\ldots,g_n$ is conjugate to $t_0$. We finish the prove of the proposition.
\end{proof}

As an immediate consequence, we have Lemma \ref{lemma::ST}.
\begin{lemma}
\label{lemma::ST}
Let $g_1,\ldots, g_n$ be equal to $s_0$, $s_1$, $s_2$, $t_0$, $t_1$ or $t_2$ satisfying $g_1\cdots g_n=1$. Suppose that $(g_1,\ldots,g_n)$ is inverse-free. Then, $n\equiv 0\pmod{6}$ and the $n$-tuple $(g_1,\ldots,g_n)$ is Hurwitz equivalent to either $(s_0,s_2)^{n/2}$ or $(t_0,t_2)^{n/2}$.
\end{lemma}
\begin{proof}
By Proposition \ref{proposition::ST}, either each $g_i$ is equal to one of $s_0$, $s_1$, $s_2$, or each $g_i$ is equal to one of $t_0$, $t_1$, $t_2$. By Theorem \ref{theorem::Moishezon}, the $n$-tuple can be transformed by elementary transformations into either $(s_0,s_2,s_0,s_2,s_0,s_2)^{n/6}$ or $(t_0,t_2,t_0,t_2,t_0,t_2)^{n/6}$.
\end{proof}

Note that $s_0s_2s_0s_2s_0s_2=t_0t_2t_0t_2t_0t_2=1$ and in fact sextuples with alternative $s_i$'s and $s_j$'s (resp. $t_i$'s and $t_j$'s) can be transformed into each other by elementary transformations. In general, we will show the reduced form of the product for a tuple with alternative powers of $s_0$ and $s_2$ (resp. powers of $t_0$ and $t_2$). We will only prove Proposition \ref{proposition::reduced_form_s0_s2} but omit the proof of Proposition \ref{proposition::reduced_form_t0_t2}, which is quite similar. The idea comes from Moishezon (see \cite[p.181-187]{moishezon1977complex}) but with a slight modification and a more subtle analysis.

\begin{proposition}
\label{proposition::reduced_form_s0_s2}
Let $(g_1,\ldots,g_n)$ be a tuple of $s_0$, $s_2$ with $n\ge 1$ and take $\mu,\nu\ge 1$.
Let $\mathcal{T}$ be the set of tuples of $s_0$, $s_2$ obtained from $(g_1,\ldots,g_n)$ by elementary transformations.
Suppose that each tuple in $\mathcal{T}$ satisfies the following requirements:
\begin{enumerate}[i),topsep=0pt,itemsep=-1ex,partopsep=1ex,parsep=1ex]
    \item it starts with at least $\mu$ $s_2$;
    \item it ends with at least $\nu$ $s_0$;
    \item it contains no consecutive sub-tuples of the form $(s_0,s_2)^3$.
\end{enumerate}
Then, the reduced form of $g_1\cdots g_n$ is given by $(ba^2)^{\mu-1}bRb(a^2b)^{\nu-1}$ with some $R\in G$.
\end{proposition}

\begin{proposition}
\label{proposition::reduced_form_t0_t2}
Let $(g_1,\ldots,g_n)$ be a tuple of $t_0$, $t_2$ with $n\ge 1$ and take $\mu,\nu\ge 1$.
Let $\mathcal{T}$ be the set of tuples of $t_0$, $t_2$ obtained from $(g_1,\ldots,g_n)$ by elementary transformations.
Suppose that each tuple in $\mathcal{T}$ satisfies the following requirements:
\begin{enumerate}[i),topsep=0pt,itemsep=-1ex,partopsep=1ex,parsep=1ex]
    \item it starts with at least $\mu$ $t_0$;
    \item it ends with at least $\nu$ $t_2$;
    \item it contains no consecutive sub-tuples of the form $(t_0,t_2)^3$.
\end{enumerate}
Then, the reduced form of $g_1\cdots g_n$ is given by $(ba)^{\mu-1}bRb(ab)^{\nu-1}$ with some $R\in G$.
\end{proposition}

\begin{proof}[Proof of Proposition \ref{proposition::reduced_form_s0_s2}]
Using elementary transformations on $(g_1,\ldots,g_n)$ we can get different resulting tuples in $\{s_0,s_2\}$, which form the set $\mathcal{T}$. Suppose that $(h_1,\ldots,h_n)$ is the maximal among them according to the lexicographical order given by $s_0<s_2$. We write $(h_1,\ldots,h_n)$ in the following form
\begin{equation*}
    (h_1,\ldots,h_n) = \prod_{i=1}^N (s_2)^{u_i} \bullet (s_0)^{v_i}
\end{equation*}
with $\sum_{i=1}^N (u_i+v_i)=n$, where $u_1\ge\mu$, $v_N\ge \nu$ and $u_i>0, v_i>0$ for all $i=1,\ldots,N$.

\textbf{Claim 1:} $u_i\ge 2$, $i=2,3,\ldots,N$.

\textbf{Claim 2:} $v_1\ge 2$.

\textbf{Claim 3:} For $i\in\{1,2,\ldots,N-2\}$, if $v_i=1$ then $v_{i+1}>1$.

\textbf{Claim 4:} For $i\in\{2,3,\ldots,N-1\}$, if $v_i=1$ then $u_i\ge 3$ and $u_{i+1}\ge 3$.

Claim 1 relies on the maximality of $(h_1,\ldots,h_n)$. Claim 2 uses the first hypothesis. The third hypothesis guarantees both Claim 3 and 4. Now, set $Y_i=(s_0)^{v_i}\bullet (s_2)^{u_{i+1}}$ for $i=1,\ldots, N-1$, say it to be of the \emph{second type} if $v_i\ge 2$, the \emph{first type} if $v_i=1$. Claim 2 shows that $Y_1$ is of the second type and Claim 3 reveals that there is no adjacent pair in the first type.
Hence, we are able to find sub-tuples $Z_1,\ldots, Z_M$ of $(h_1,\ldots,h_n)$ such that each $Z_j$, $j=1,\ldots,M$, is either
\begin{itemize}[-]
    \item equal to some $Y_i$ of the second type with $i\in\{1,\ldots,N-1\}$, or
    \item the concatenation $Y_i\bullet Y_{i+1}$ with $i\in\{1,\ldots,N-2\}$ where $Y_{i+1}$ is of the first type
\end{itemize}
and we can write $(h_1,\ldots,h_n)$ in the form
\begin{equation*}
    (h_1,\ldots,h_n) = (s_2)^{u_1}\bullet \prod_{j=1}^M Z_j \bullet s_0^{v_N}.
\end{equation*}

For $j=1,\ldots,M$, if $Z_j$ is equal to some $Y_i$ of the second type, then the product of components of $Z_j$ has the reduced form $a^2 R_j a^2$ with some $R_j\in G$. Indeed, $Z_j=(s_0)^{v_i}\bullet (s_2)^{u_{i+1}}$ with $v_i\ge 2$ and $u_{i+1}\ge 2$, the product of whose components is equal to $(a^2b)^{v_i-1}a(ba^2)^{u_{i+1}-1}$. 
If $Z_j=Y_i\bullet Y_{i+1}$, $i\in\{1,\ldots,N-2\}$, then the product of its components is given by
\begin{equation*}
    s_0^{v_i} s_2^{u_{i+1}} s_0^{v_{i+1}} s_2^{u_{i+2}} = (a^2b)^{v_i-1} a (ba^2)^{u_{i+1}-2} a^2 (ba^2)^{u_{i+2}-2}
\end{equation*}
with $v_{i+1}=1$, $u_{i+1}\ge 3$, $u_{i+2}\ge 3$ and $v_i\ge 2$, which also has the reduced form $a^2 R_j a^2$ with some $R_j\in G$.
Hence, $g_1\cdots g_n=h_1\cdots h_n$ has the form
\begin{equation*}
    (ba^2)^{u_1} \prod_{j=1}^M (a^2 R_j a^2) (a^2b)^{v_N}
\end{equation*}
with $R_j\in G$, $j=1,\ldots,M$ where each of $a^2 R_j a^2$ is reduced.
\end{proof}

%% file: Section4/conjugate.tex
\subsection{Conjugates of short elements and tuples}
\label{subsection::Conjugates of short elements and their tuples}

Suppose that $(g_1,\ldots, g_n)$ is an $n$-tuple with each $g_i$ conjugate to some short element (i.e. the component $g_i$ is conjugate to $a$, $a^2$, $b$ or $s_1$, $t_1$). 
In this subsection we show that, in the vast majority of cases, by successive application of elementary transformations the $n$-tuple can be transformed into an $n$-tuple of short elements.

\begin{lemma}
\label{lemma::h=g1g2}
Let $g_1$, $g_2$, $h$, $Q' \in G$ be such that $h=g_1 g_2$. Then both
\[
(Q'^{-1}h^{-1}g_1hQ',Q'^{-1}h^{-1}g_2hQ')
\text{ and }
(Q'^{-1}hg_1h^{-1}Q',Q'^{-1}hg_2h^{-1}Q')
\]
are Hurwitz equivalent to $(Q'^{-1}g_1Q',Q'^{-1}g_2Q')$.
\end{lemma}
\begin{proof}
Using $R_1^{-2}$, we transform
\begin{equation*}
(Q'^{-1}h^{-1}g_1hQ',Q'^{-1}h^{-1}g_2hQ')    
\end{equation*}
into $(Q'^{-1}g_1Q',Q'^{-1}g_1^{-1}hQ')$.
The result is equal to $(Q'^{-1}g_1Q',Q'^{-1}g_2Q')$ as $g_1^{-1}h=g_2$.
Similarly
\begin{equation*}
(Q'^{-1}hg_1h^{-1}Q',Q'^{-1}hg_2h^{-1}Q')    
\end{equation*}
can be transformed into $(Q'^{-1}g_1Q',Q'^{-1}g_2Q')$ by applying $R_1^2$.
\end{proof}

\begin{lemma}
\label{lemma::aba,a^2->ab}
Let $\epsilon=\pm 1$ and suppose that $(\tau_1,\tau_2)$ is equal to one of
\begin{equation*}
    \{ (a^{\epsilon}ba^{\epsilon},a^{-\epsilon}),
       (a^{-\epsilon},a^{\epsilon}ba^{\epsilon}),
       (ba^{-\epsilon},a^{-\epsilon}),
       (a^{-\epsilon},a^{-\epsilon}b),
       (a^{\epsilon},ba^{\epsilon}),
       (a^{\epsilon}b,a^{\epsilon})
       \}.
\end{equation*}
Let $(g_1,g_2)=(Q^{-1}\tau_1 Q,Q^{-1}\tau_2 Q)$ be a pair in $G$ with $Q\in G$ and suppose that $Q^{-1}\tau_1 \tau_2 Q$ is short. Then $(g_1,g_2)$ is Hurwitz equivalent to a pair of short elements.
\end{lemma}
\begin{proof}
When $\tau_1 \tau_2=a^{\epsilon}b$, since $Q^{-1}a^{\epsilon}bQ$ is short, $Q$ is either $(a^{\epsilon}b)^ka^{\zeta}$ or $(ba^{-\epsilon})^la^{\zeta}$ with $k,l\ge 0$ and $\zeta=0,1,2$. If $Q=a^{\zeta}$, then both $Q^{-1}\tau_1Q$ and $Q^{-1}\tau_2Q$ are short. The result follows from Lemma \ref{lemma::h=g1g2}. When $\tau_1\tau_2=ba^{\epsilon}$ or $\tau_1\tau_2=a^{\epsilon}ba^{\epsilon}$, the proof is similar.
\end{proof}

%%%%%

We introduce the following operations and their restorations on an $n$-tuple $(g_1,\ldots,g_n)$ of elements in $G$ conjugate to some short elements.

\begin{itemize}[-]
    \item \underline{Operation $1$:} For $i\in\{1,\ldots,n-1\}$, suppose that the reduced forms of $g_i$ and $g_{i+1}$ are expressed by $Q_i^{-1}\tau_iQ_i$ and $Q_{i+1}^{-1}\tau_{i+1}Q_{i+1}$ with $\tau_i,\tau_{i+1}\in\mathcal{S}$, $Q_i,Q_{i+1}\in G$ such that $Q_i=Q_{i+1}$, $(\tau_i,\tau_{i+1})$ is listed in Table \ref{table::short elements} and either $Q_i=1$ or $\tau_i\tau_{i+1}=1$ or both $\tau_i$, $\tau_{i+1}$ are powers of $a$. Then, the operation is a contraction as in Subsection \ref{subsection::contractions-restorations} that replaces $(g_i,g_{i+1})$ with $g_i g_{i+1}$.
    
    \item \underline{Operation $2$:} For $i\in\{1,\ldots,n\}$, suppose that $g_i=1$. The operation moves the identical component to the rightmost position via elementary transformations, removes it and reduces $(g_1,\ldots,g_n)$ to an $(n-1)$-tuple. 
\end{itemize}

\begin{table}[htb]
\centering
%\begin{adjustbox}{width=\textwidth}
\begin{tabular}{ c|ccccccccc| } 
 \diagbox{$g_i$}{$g_i\cdot g_{i+1}$}{$g_{i+1}$}
          &$a$   &$a^2$    &$b$  &$a^2b$&$aba$&$ba^2$   &$ba$ &$a^2ba^2$&$ab$  \\ 
 \hline
 $a$      &$a^2$ &$1$      &     &$b$   &     &         &$aba$&$ba^2$   &$a^2b$\\ 
 $a^2$    &$1$   &$a$      &     &$ab$  &$ba$ &$a^2ba^2$&     &         &$b$   \\ 
 $b$      &      &         &$1$  &      &     &$a^2$    &$a$  &         &      \\ 
 $a^2b$   &      &$a^2ba^2$&$a^2$&      &     &$a$      &$1$  &         &      \\ 
 $aba$    &      &$ab$     &     &$a$   &     &         &     &$1$      &      \\ 
 $ba^2$   &$b$   &$ba$     &     &      &$a$  &         &     &         &$1$   \\ 
 $ba$     &$ba^2$&$b$      &     &$1$   &     &         &     &$a^2$    &      \\ 
 $a^2ba^2$&$a^2b$&         &     &      &$1$  &         &     &         &$a^2$ \\ 
 $ab$     &$aba$ &         &$a$  &      &     &$1$      &$a^2$&         &      \\ 
 \hline
\end{tabular}
%\end{adjustbox}
\caption{Some pairs $(g_i,g_{i+1})$ of short elements and the products $g_ig_{i+1}$.}
\label{table::short elements}
\end{table}

Operation $1$ is a contraction, whose restoration is introduced in Subsection \ref{subsection::elementary transformations}.
The restoration of Operation $2$ will simply add an identical element on the right side of the tuple.
The following proposition shows that, if we use the technique introduced in Subsection \ref{subsection::elementary transformations} carefully, the resulting tuple is under control.

\begin{proposition}
\label{proposition::reduce_and_reloading}
Let $(g_1,\ldots,g_n)$ be an inverse-free $n$-tuple of elements in $G$ conjugate to some short elements such that $g_1\cdots g_n=1$.
Suppose that we first apply the following operations successively on $(g_1,\ldots,g_n)$:
\begin{enumerate}[i),topsep=0pt,itemsep=-1ex,partopsep=1ex,parsep=1ex]
    \item the elementary transformation $R_i$, but avoiding that $g_i$ is short and $g_{i-1}^{-1}g_ig_{i+1}$ is long;
    \item the elementary transformation $R_i^{-1}$, but avoiding that $g_{i+1}$ is short and $g_ig_{i+1}g_i^{-1}$ is long;
    \item Operation $1$;
    \item Operation $2$;
\end{enumerate}
then apply restorations of Operation $1$ and $2$ in the reverse order.
If all components in the resulting tuple before restorations are short,
then the initial tuple is Hurwitz equivalent to the resulting tuple after restorations and further Hurwitz equivalent to a tuple of short elements.
\end{proposition}

\begin{proof}
Lemma \ref{lemma::technique} shows that the initial tuple is Hurwitz equivalent to the resulting tuple after all operations and restorations. We suppose that each component is short in the tuple before restorations.

Operation $1$ may combine $Q^{-1}\tau_iQ$ and $Q^{-1}\tau_jQ$ into $Q^{-1}\tau_i\tau_j Q$ with $Q\in G$ and $\tau_i,\tau_j\in\mathcal{S}$. By elementary transformations, the product is sent to a conjugate of the form $P^{-1}Q^{-1}\tau_i\tau_j QP$ with some $P\in G$. To restore the operation, it is further rewritten as a pair
\begin{equation*}
    (P^{-1}Q^{-1}\tau_i QP,P^{-1}Q^{-1}\tau_j QP).
\end{equation*}
Suppose that $P^{-1}Q^{-1}\tau_i\tau_j QP$ is short and $(\tau_i,\tau_j)$ is listed in Table \ref{table::short elements}.

When $\tau_i\tau_j\in\{a,a^2\}$, the element $QP$ must be a power of $a$.
It is not true that both $P^{-1}Q^{-1}\tau_i Q P$ and $P^{-1}Q^{-1}\tau_j Q P$ are short in general, as a conjugate of $b$ with a power of $a$ may be long.
We list all exceptional possibilities of $(P^{-1}Q^{-1}\tau_i Q P , P^{-1}Q^{-1}\tau_j Q P)$ as below.
\begin{equation*}
    (ba^2,aba^2), (aba,a^2ba),
    (ba,a^2ba), (a^2ba^2,aba^2),
    (a^2ba,a^2b), (aba^2,aba),
    (aba^2,ab), (a^2ba,a^2ba^2).
\end{equation*}
However, each of them can be transformed into a pair of short elements by at most two elementary transformations.

When $\tau_i\tau_j=b$, the element $QP$ must be a power of $b$. A conjugate of $a$ or $a^2$ with a power of $b$ may be long. The exceptional possibilities of $(P^{-1}Q^{-1}\tau_i Q P,P^{-1}Q^{-1}\tau_j Q P)$ that one of the components is long are listed as below.
\begin{equation*}
    (bab,ba^2), (ba^2b,ba), (a^2b,bab), (ab,ba^2b).
\end{equation*}
Again, each of them can be transformed into a pair of short elements by an elementary transformation.

When $\tau_i\tau_j=1$, we get $P=1$.
Assume that one of $Q^{-1}\tau_iQ$ and $Q^{-1}\tau_jQ$ is long,
the inverse-freeness of $(g_1,\ldots,g_n)$ implies that $\tau_i$ and $\tau_j$ are powers of $a$ and $Q$ is not a power of $a$. Due to the hypothesis that elementary transformations never make short elements long, after all restorations, $Q^{-1}\tau_i\tau_j Q$ (as an additional identical element) will become a sub-tuple $(h_1,\cdots,h_m)$ with $m\ge 2$ such that $h_1,\ldots,h_m$ are conjugate to the powers of $a$ simultaneously.
It contradicts the inverse-freeness.
Thus, both $Q^{-1}\tau_iQ$ and $Q^{-1}\tau_jQ$ are short.

The remaining cases shown in Table \ref{table::short elements} are covered by Lemma \ref{lemma::aba,a^2->ab}. Hence, by successive application of elementary transformations, each of the components of the resulting tuple is short.
\end{proof}

\begin{definition}
The $\mathcal{S}$-\emph{complexity} of an element $g\in G$ conjugate to some element in $\mathcal{S}$ is defined as $f(g)$ such that
\begin{equation*}
    f(g) =
    \begin{cases}
        l(Q) & \text{if $g=Q^{-1}wQ$ is long with $w\in\{a^{\epsilon},b,a^{\epsilon}ba^{\epsilon}|\epsilon=1,2\}$ and $Q\in G$;}\\
        0 & \text{if $g$ is short.}
    \end{cases}
\end{equation*}
\end{definition}

\begin{definition}
Let $(g_1,\ldots,g_n)$ be an $n$-tuple in $G$ such that each of $g_i$, $i=1,\ldots,n$, is conjugate to some element in $\mathcal{S}$.
A sequence of elementary transformations 
$(R_{i_1}^{\epsilon_1},\ldots,R_{i_m}^{\epsilon_m})$, $\epsilon_1,\ldots,\epsilon_m\in\{1,-1\}$, is said to make the sum of $\mathcal{S}$-complexities of $(g_1,\ldots,g_n)$ \emph{strictly-smaller} if, for each $m'<m$, the composition
$R_{i_{m'}}^{\epsilon_{m'}}\circ\cdots\circ R_{i_1}^{\epsilon_1}$
transforms $(g_1,\ldots,g_n)$ into a tuple with the same sum of $\mathcal{S}$-complexities but
$R_{i_m}^{\epsilon_m}\circ\cdots\circ R_{i_1}^{\epsilon_1}$
transforms $(g_1,\ldots,g_n)$ into a tuple with a smaller sum of $\mathcal{S}$-complexities.
\end{definition}

A sequence of elementary transformations that makes the sum of $\mathcal{S}$-complexities strictly-smaller never makes short elements long, as described in Proposition \ref{proposition::reduce_and_reloading} \RNum{1}) and \RNum{2}).

Let $(g_1,\ldots,g_n)$ be an $n$-tuple in $G$.
For $i=1,\ldots,n-1$, suppose that the reduced forms of $g_i$ and $g_{i+1}$ are expressed by $t_{k_i}^{(i)}\ldots t_1^{(i)}$ and $\tilde{t}_1^{(i)}\ldots \tilde{t}_{l_i}^{(i)}$ with $k_i=l(g_i)$, $l_i=l(g_{i+1})$, $t_j^{(i)}\in\{a,a^2,b\}$, $j=1,\ldots,k_i$ and $\tilde{t}_j^{(i)}\in\{a,a^2,b\}$, $j=1,\ldots,l_i$.
The reduced form of $g_ig_{i+1}$ is then either
\[
t_{k_i}^{(i)}\ldots t_{m_i+1}^{(i)}r_i\tilde{t}_{m_i+1}^{(i)}\ldots \tilde{t}_{l_i}^{(i)}
\text{ or }
t_{k_i}^{(i)}\ldots t_{m_i+1}^{(i)}r_i
\text{ or }
r_i\tilde{t}_{m_i+1}^{(i)}\ldots \tilde{t}_{l_i}^{(i)}
\]
where $r_i\in G$, $l(r_i)\le 1$ and $0\le m_i\le k_i,l_i$.

\begin{lemma}
\label{lemma::m_i-1_and_m_i}
Let $(g_1,\ldots,g_n)$ be an $n$-tuple in $G$ such that each of $g_i$, $i=1,\ldots,n$, is conjugate to some element in $\mathcal{S}$ and $g_1\cdots g_n=1$. Let $m_i$ be the same as above and set $m_0=m_n=0$ for convenience. Suppose that
\begin{enumerate}[(1),topsep=0pt,itemsep=-1ex,partopsep=1ex,parsep=1ex]
    \item there is no pair of adjacent components $g_i$, $g_{i+1}$ of the reduced forms $Q^{-1}\tau_iQ$, $Q^{-1}\tau_{i+1}Q$ with $Q\in G$ and $(\tau_i,\tau_{i+1})$ in Table \ref{table::short elements} such that either $Q_i=1$ or $\tau_i\tau_{i+1}=1$ or both $\tau_i$, $\tau_{i+1}$ are powers of $a$.
    \item there is no sequence of elementary transformations that makes $\sum_i f(g_i)$ strictly-smaller.
\end{enumerate}
Then $m_0,\ldots,m_n$ have the following properties.
\begin{enumerate}[(a),topsep=0pt,itemsep=-1ex,partopsep=1ex,parsep=1ex]
    \item For $i=1,\ldots,n-1$, $m_i\le \frac{l(g_i)+1}{2}$ and $m_i\le \frac{l(g_{i+1})+1}{2}$.
    \item For $i=1,\ldots,n$, $m_{i-1}+m_i\ge l(g_i)$ only if the reduced form of $g_i$ is $Q_i^{-1}a^{\epsilon_i}Q_i$ with $\epsilon_i=1,2$, $Q_i\in G$ and $l(Q_i)\ge 0$.
    \item If $m_{i-1}+m_i\le l(g_i)$ for each of $i=1,\ldots,n$, then $n=0$.
\end{enumerate}
\end{lemma}

\begin{proof}
(a) When both $g_i$ and $g_{i+1}$ are short, since $(g_i,g_{i+1})$ does not figure in Table \ref{table::short elements}, we check all possibilities and get that $m_i\le \frac{l(g_i)+1}{2}, \frac{l(g_{i+1})+1}{2}$.

When $g_i\in\mathcal{S}$ but $g_{i+1}\not\in\mathcal{S}$, say $g_{i+1}=Q_i^{-1}a^{\epsilon_i}Q_i$ or $Q_i^{-1}bQ_i$ or $Q_i^{-1}a^{\epsilon_i}ba^{\epsilon_i}Q_i$ with $\epsilon_i=1,2$ and $l(Q_i)\ge 1$, therefore $l(g_i)\le 3\le l(g_{i+1})$.
Assume that $m_i>\frac{l(g_i)+1}{2}$, then $m_i=l(g_i)$ and $l(g_i)\ge 2$.
If $l(g_i)=m_i=2$, as $g_{i+1}$ is long, then $l(g_ig_{i+1}g_i^{-1})\le l(g_{i+1})-2$, contradicting the hypothesis (2).
If $l(g_i)=m_i=l(g_{i+1})=3$, then the pair $(g_i,g_{i+1})$ is either $(s_1,a^2ba)$ or $(t_1,aba^2)$, which can be transformed into a pair of short elements by $R_1^2$, contradicting the hypothesis (2)
If $l(g_i)=m_i=3$ but $l(g_{i+1})\ge 5$, then again $l(g_ig_{i+1}g_i^{-1})\le l(g_{i+1})-2$, a contradiction.
Hence $m_i\le \frac{l(g_i)+1}{2}\le\frac{l(g_{i+1})+1}{2}$.

When $g_i\not\in\mathcal{S}$ but $g_{i+1}\in\mathcal{S}$, there is a similar argument.

When both $g_i$ and $g_{i+1}$ are long, suppose that their reduced forms are 
\begin{equation*}
Q_i^{-1}w_iQ_i \text{ and } Q_{i+1}^{-1}w_{i+1}Q_{i+1}
\end{equation*}
with $w_i,w_{i+1}\in\{a,a^2,b,aba,a^2ba^2\}$.
Assume that $l(Q_i)\le l(Q_{i+1})$ without loss of generality.
Assume that $m_i>\min\{\frac{l(g_i)+1}{2},\frac{l(g_{i+1})+1}{2}\}$.
Therefore $Q_{i+1}$ ends with $Q_i$.
Write $Q_{i+1}=\tilde{Q}Q_{i+1}$ and
\begin{equation*}
    (g_i,g_{i+1})=(Q_i^{-1}w_iQ_i,Q_i^{-1}\tilde{Q}^{-1}w_{i+1}\tilde{Q}Q_i).
\end{equation*}

Suppose that $l(\tilde{Q})=0$.
We further assume that $l(w_i)\le l(w_{i+1})$ without loss of generality.
Since $m_i>l(Q_i)+\frac{l(w_i)+1}{2}$, $w_iw_{i+1}\neq 1$ and one of $w_i$, $w_{i+1}$ is not a power of $a$, the pair $(w_i,w_{i+1})$ must be either $(a,a^2ba^2)$ or $(a^2,aba)$. Therefore, $l(g_ig_{i+1}g_i^{-1})\le l(g_{i+1})-2$, contradicting the hypothesis (2).

Suppose that the element $\tilde{Q}$ is of length at least $1$.
Therefore $l(w_i)\le 3\le l(\tilde{Q}^{-1}w_{i+1}\tilde{Q})$ and $m_i>\frac{l(g_i)+1}{2}$.
If $l(w_i)=1$, then $m_i>l(Q_i)+1$ and $l(g_ig_{i+1}g_i^{-1})\le l(g_{i+1})-2$, contradicting the hypothesis (2).
If $l(w_i)=3$ and $l(\tilde{Q}^{-1}w_{i+1}\tilde{Q})=3$, then $(w_i,\tilde{Q}^{-1}w_{i+1}\tilde{Q})$ is either $(s_1,a^2ba)$ or $(t_1,aba^2)$ and therefore $(g_i,g_{i+1})$ can be transformed into either $(Q_i^{-1}a^2bQ_i,Q_i^{-1}bQ_i)$ or $(Q_i^{-1}abQ_i,Q_i^{-1}bQ_i)$ by $R_1^2$, contradicting the hypothesis (2).
If $l(w_i)=3$ and $l(\tilde{Q}^{-1}w_{i+1}\tilde{Q})\ge 5$, then $m_i\ge l(Q_i)+3$ and $l(g_ig_{i+1}g_i^{-1})\le l(g_{i+1})-2$, contradicting the hypothesis (2).

(b) Suppose that $m_{i-1}+m_i\ge l(g_i)$ for some $i=1,\ldots,n-1$.

Suppose that $g_i\in\mathcal{S}$.
If $g_i$ is of length $2$ (i.e. the element $g_i$ is one of $s_0$, $s_2$, $t_0$ and $t_2$), then $m_{i-1}=m_i=1$. Therefore, one of $g_{i-1}$, $g_{i+1}$ is equal to $b$, contradicting Table \ref{table::short elements}.
If $g_i=b$, then one of $g_{i-1}$ and $g_{i+1}$ is long starting and ending with $b$. Therefore, either $l(g_i^{-1}g_{i-1}g_i)<l(g_{i-1})$ or $l(g_ig_{i+1}g_i^{-1})<l(g_{i+1})$, contradicting the hypothesis (2).
If $g_i=a^{\epsilon_i}ba^{\epsilon_i}$ with $\epsilon_i=1,2$, then one of $m_{i-1}$ and $m_i$ is equal to $2$ and thus one of $g_{i-1}$, $g_{i+1}$ is long. It is impossible as long elements are of length at least $3$.

Suppose that $g_i$ is long.
If $g_i=Q_i^{-1}bQ_i$, then either $g_{i-1}=bQ_i$ or $g_{i+1}=Q_i^{-1}b$, but thus either $g_{i-1}g_ig_{i-1}^{-1}=bQ_i(Q_i^{-1}bQ_i)Q_i^{-1}b=b$ or $g_{i+1}^{-1}g_ig_{i+1}=bQ_i(Q_i^{-1}bQ_i)Q_i^{-1}b=b$.
If $g_i=Q_i^{-1}a^{\epsilon_i}ba^{\epsilon_i}Q_i$ with $\epsilon_i=1,2$, then either $m_{i-1}=l(Q_i^{-1}a^{\epsilon_i}b)$ or $m_i=l(ba^{\epsilon_i}Q_i)$. It implies that either $g_{i-1}=ba^{-\epsilon_i}Q_i$ or $g_{i+1}=Q_i^{-1}a^{-\epsilon_i}b$, thus either $g_{i-1}g_ig_{i-1}^{-1}=a^{2\epsilon_i}b$ or $g_{i+1}^{-1}g_ig_{i+1}=ba^{2\epsilon_i}$. Both cases contradict the hypothesis (2).

We conclude that either $g_i=a^{\epsilon_i}$ or $g_i=Q_i^{-1}a^{\epsilon_i}Q_i$.

(c) Assume that $n\ge 1$.

By (2), when $m_{i-1}+m_i=l(g_i)$, then $g_i=Q_i^{-1}a^{\epsilon_i}Q_i$ with $\epsilon_i=1,2$ and $l(Q_i)\ge 0$.
If $g_i=a^{\epsilon_i}$ is short and assume that $m_{i-1}=0$, $m_i=1$ without loss of generality, then $g_{i+1}$ is either $a^{\epsilon_i}ba^{\epsilon_i}$ or a long element starting with $a^{\epsilon_i}$.
If $g_i=Q_i^{-1}a^{\epsilon_i}Q_i$ is long and $m_{i-1}=l(Q_i)$, to avoid $l(g_{i-1}g_ig_{i-1}^{-1})<l(g_i)$ then $g_{i-1}$ must be longer than $Q_i^{-1}$, contradicting the hypothesis that $m_{i-1}=l(Q_i)$. Therefore, neither $m_{i-1}$ nor $m_i$ is equal to $l(Q_i)$ and in particular, $m_{i-1}+m_i<l(g_i)$.

The proof of (2) and the above observation show that there is no possibility to fully reduce $g_i$ or $g_{i+1}$ in the product $g_i g_{i+1}$ and $m_{i-1}+m_i<l(g_i)$ if $g_i\neq a^{\epsilon_i}$. They imply a contradiction that $g_1\cdots g_n\neq 1$.
\end{proof}

Now we introduce the main result in this subsection.

\begin{theorem}
\label{theorem::Livne S->short}
Let $g_1,\ldots,g_n$ be such that each of them is conjugate to some element in $\mathcal{S}$ and $g_1\cdots g_n=1$. Then, the $n$-tuple $(g_1,\ldots,g_n)$ is Hurwitz equivalent to either
\begin{itemize}[-,topsep=0pt,itemsep=-1ex,partopsep=1ex,parsep=1ex]
    \item $(h_1,\ldots,h_{\mu})\bullet (s_0,t_0)^{m_{st}}\bullet (a,a^2)^{m_a}\bullet (b,b)^{m_b}\bullet (a,a,a)^{n_0}\bullet (a^2,a^2,a^2)^{n_1}$ with $\mu>0$, $m_{st}$, $m_a$, $m_b$, $n_0$, $n_1\ge 0$, $\mu+2(m_{st}+m_a+m_b)+3(n_0+n_1)=n$ such that $(h_1,\ldots,h_{\mu})$ is an inverse-free $\mu$-tuple of short elements, or
    \item $(k_1,k_1^{-1},\ldots,k_s,k_s^{-1},l_1,l_1,l_1,\ldots,l_t,l_t,l_t)$ with $s,t\ge 0$, $2s+3t=n$, $k_1,\ldots,k_s\in G$, $l_1,\ldots,l_t\in G$ and $l_j^3=1$ for each $j=1,\ldots,t$.
\end{itemize}
\end{theorem}

\begin{proof}
We first attempt to make the tuple inverse-free.
Applying any finite sequence of elementary transformations to $(g_1,\ldots,g_n)$, if we get a pair of mutually inverse elements or a triple of the form $(l,l,l)$ with $l\in G$ and $l^3=1$, then we move it to the rightmost position via elementary transformations and the resulting tuple is the concatenation of a shorter tuple and either a pair or a triple.
By induction on the length, we suppose that the $n$-tuple $(g_1,\ldots,g_n)$ is transformed into the concatenation of $(h_1,\ldots,h_{\mu})$ and $(k_1,k_1^{-1},\ldots,k_s,k_s^{-1},l_1,l_1,l_1,\ldots,l_t,l_t,l_t)$ with $\mu,s,t\ge 0$, $\mu+2s+3t=n$ such that $l_j^3=1$ for each $j=1,\ldots,t$ where $(h_1,\ldots,h_{\mu})$ is inverse-free.

We will always use the notation $m_i$ to indicate the length of the reduced part in $h_ih_{i+1}$ for $i=1,\ldots,\mu-1$ and set $m_0=m_{\mu}=0$ as before.

To prove the theorem for $(h_1,\ldots,h_{\mu})$, we use induction on
\begin{equation*}
\Big(\mu,\sum_{i=1}^{\mu} f(h_i), l(h_1), \ldots, l(h_{\mu})\Big)
\end{equation*}
and apply the following operations on $(h_1,\ldots,h_{\mu})$.
If there exists a pair of adjacent components which has the reduced form $(Q^{-1}\tau_i Q,Q^{-1}\tau_j Q)$ with $Q\in G$, $(\tau_i,\tau_j)$ in Table \ref{table::short elements} such that either $Q=1$ or $\tau_i\tau_j=1$ or both $\tau_i$, $\tau_j$ are powers of $a$, then we replace them with their product and reduce $(h_1,\ldots,h_{\mu})$ to a $(\mu-1)$-tuple.
If there exists an identical component, then we move it to the rightmost position and remove it.
If there exists a proper sequence of elementary transformations that can make $\sum_i f(h_i)$ strictly-smaller, then we apply it.

When each of the above operations fails, the resulting tuple, still denoted by $(h_1,\ldots,h_{\mu})$, satisfies all hypotheses in Lemma \ref{lemma::m_i-1_and_m_i}. Suppose that $\mu\ge 1$ and there exists some $i=2,\ldots,\mu-1$ such that $h_i=Q_i^{-1}a^{\epsilon_i}Q$ with $\epsilon_i=1,2$, $l(Q_i)\ge 0$ and $m_{i-1}=m_i=l(Q_i)+1$.
Then, either
\begin{itemize}[-]
    \item $(h_{i-1},h_i)=(a^{\epsilon_i}ba^{\epsilon_i},a^{\epsilon_i})$, or
    \item the previous component $h_{i-1}$ is long and $l(h_{i-1})\ge l(h_i)$.
\end{itemize}
In the second case, we first assume that $l(h_{i-1})=l(h_i)$.
Then $h_{i-1}=Q_i^{-1}a^{\epsilon_i}Q_i=h_i$, which is a contradiction.
Hence, $l(h_{i-1})>l(h_i)$ and, to avoid $l(h_i^{-1}h_{i-1}h_i)<l(h_{i-1})$, we claim that $h_{i-1}$ must end with $a^{\epsilon_i}Q_i$ and start with $Q_i^{-1}a^{-\epsilon_i}$.
In both cases, $l(h_i^{-1}h_{i-1}h_i)\le l(h_{i-1})$ and we are able to reduce $(h_1,\ldots,h_{\mu})$ to an $n$-tuple,
say $(\tilde{h}_1,\ldots,\tilde{h}_{\mu})$,
such that $\sum_j f(h_j)=\sum_j f(\tilde{h}_j)$, $l(h_j)=l(\tilde{h}_j)$ for $1\le j< i-1$
but $l(\tilde{h}_{i-1})=l(h_i)<l(h_{i-1})$ via the elementary transformation $R_{i-1}$.

The induction does not stop unless $\mu$ is equal to $0$.
Due to Proposition \ref{proposition::reduce_and_reloading}, by restoring the operations and applying more elementary transformations, we get a resulting $\mu$-tuple of short elements that can be obtained from the original $(h_1,\ldots,h_{\mu})$ via elementary transformations directly. Hence, the $n$-tuple $(g_1,\ldots,g_n)$ can be transformed into
\begin{equation*}
(h'_1,\ldots,h'_{\mu},k_1,k_1^{-1},\ldots,k_s,k_s^{-1},l_1,l_1,l_1,\ldots,l_t,l_t,l_t)
\end{equation*}
with $\mu,s,t\ge 0$, $\mu+2s+3t=n$, $l_j^3=1$ for each $j=1,\ldots,t$ such that each of $h'_i$, $i=1,\ldots,\mu$ is short and $(h'_1,\ldots,h'_{\mu})$ is inverse-free.

Suppose that $\mu>0$. There is always a pair $(h'_i,h'_j)$ of components with $1\le i\neq j\le \mu$ that is a generating set of $G$. By Lemma \ref{lemma::tuples containing a generating set}, each pair of the form $(k,k^{-1})=(Q^{-1}wQ,Q^{-1}w^{-1}Q)$ with $Q\in G$, $w,w^{-1}\in\{a,a^2,b,s_0\}$
in the resulting tuple
can be transformed into $(w,w^{-1})$.
There is a similar argument for each triple $(l,l,l)$ with $l^3=1$.
Hence, by elementary transformations, the $n$-tuple can be transformed into a tuple of short elements.
\end{proof}

Theorem \ref{theorem::Livne S->short} is surprising.
In fact, there are infinitely many pairs $(g_s,g_t)$ in $G$ up to Hurwitz equivalence such that $g_sg_t=1$ and $g_s$, $g_t$ are conjugates of $s_0$ and $t_0$ respectively.
However, all triples $(g_a,g_b,g_s)$ that $g_ag_bg_s=1$ and $g_a$, $g_b$, $g_s$ are conjugates of $a$, $b$, $s_0$ respectively, are mutually Hurwitz equivalent.
In particular, for any $Q\in G$, we have
\begin{equation*}
(Q^{-1}aQ,Q^{-1}aba^2Q,Q^{-1}abaQ)\sim (a,b,s_2).
\end{equation*}

%% file: Section4/classification.tex
\subsection{Classification of tuples up to Hurwitz equivalence}

Given $g_1,\ldots,g_n$ and $h_1,\ldots,h_n\in G$ conjugate to elements in $\mathcal{S}$ such that $g_1\cdots g_n=h_1\cdots h_n=1$, suppose that the $n$-tuples $(g_1,\ldots,g_n)$ and $(h_1,\ldots,h_n)$ have the same number of components in each conjugacy class.
In this subsection, we show that $(g_1,\ldots,g_n)$ is Hurwitz equivalent to $(h_1,\ldots,h_n)$ in most cases.
In particular, we introduce a normal form for tuples of elements conjugate to some short elements that only depends on the numbers of components in every conjugacy classes.

The following theorem is a partial result, which interprets the projective global monodromy of an achiral Lefschetz fibration. Matsumoto presented a slightly different theorem in \cite[Theorem 3.6]{matsumoto1985torus}.

\begin{theorem}
\label{theorem::Livne ST}
Let $g_1,\ldots,g_n \in G$ be such that $p$ of them are conjugates of $s_0$, $q=n-p$ of them are conjugates of $t_0$ and $g_1\cdots g_n=1$. Then,
\begin{enumerate}[topsep=0pt,itemsep=-1ex,partopsep=1ex,parsep=1ex]
    \item if $p>q$, then $p-q\equiv 0 \pmod{6}$ and the $n$-tuple $(g_1,\ldots,g_n)$ is Hurwitz equivalent to $(s_0,s_2)^{(p-q)/2}\bullet (s_0,t_0)^q$;
    \item if $p<q$, then $q-p\equiv 0 \pmod{6}$ and the $n$-tuple $(g_1,\ldots,g_n)$ is Hurwitz equivalent to $(t_0,t_2)^{(q-p)/2}\bullet (s_0,t_0)^p$;
    \item if $p=q$, then the $n$-tuple $(g_1,\ldots,g_n)$ is Hurwitz equivalent to $(k_1,k_1^{-1},\ldots,k_p,k_p^{-1})$ where each of $k_j$, $j=1,\ldots,p$, is conjugate to $s_0$.
\end{enumerate}
\end{theorem}

\begin{proof}
Theorem \ref{theorem::Livne S->short} reveals that, by elementary transformations, the $n$-tuple can be transformed into either $(k_1,k_1^{-1},\ldots,k_s,k_s^{-1})$ with $s=p=q$ or $(h_1,\ldots,h_{\mu})\bullet (s_0,t_0)^{m_{st}}$ with $\mu>0$, $m_{st}\ge 0$, $\mu+2m_{st}=n$ such that $(h_1,\ldots,h_{\mu})$ is an inverse-free $\mu$-tuple of short elements.
On the latter, by Lemma \ref{lemma::ST}, we get $\mu\equiv 0\pmod{6}$ and $(h_1,\ldots,h_{\mu})$ can be transformed into either $(s_0,s_2)^{\mu/2}$ or $(t_0,t_2)^{\mu/2}$ by elementary transformations. Hence, $m_{st}=\min\{p,q\}$ and $\mu=|p-q|$.
\end{proof}

In general, we have Theorem \ref{theorem::Livne mathcalS}, whose proof will be given at the end.

\begin{lemma}
\label{lemma::a^2,s}
Let $g_1,\ldots,g_n$ be $a^2$, $s_0$, $s_1$ or $s_2$ such that only one of them is equal to $a^2$ and $g_1\cdots g_n=1$. Then, $n\equiv 3\pmod{6}$ and the $n$-tuple $(g_1,\ldots,g_n)$ is Hurwitz equivalent to
\begin{equation*}
    (a^2,s_0,s_2)\bullet (s_0,s_2)^{(n-3)/2}.
\end{equation*}
\end{lemma}

\begin{proof}
Since a cyclic permutation of an $n$-tuple in $G$ can be obtained by a finite sequence of elementary transformations as in Lemma \ref{lemma::cyclic shift}, we may assume that $g_1=a^2$ without loss of generality. Since $g_1\cdots g_n=1$, then $n\ge 3$. If $n=3$, then the pair $(g_1g_2,g_3)$ must be equal to $(t_j,s_j)$ with some $j$ and therefore $(g_1,g_2,g_3)$ is given by $(a^2,s_{j+1},s_j)$ as $at_j=s_{j+1}$.
Otherwise, $n\ge 4$. We replace $g_1$ and $g_2$ with their product, which is one of $t_0$, $t_1$ and $t_2$. The $n$-tuple $(g_1,\ldots,g_n)$ is replaced by an $(n-1)$-tuple whose first component belongs to $\{t_0,t_1,t_2\}$ and the rest components are $s_0$, $s_1$ or $s_2$. By Theorem \ref{theorem::Livne ST}, $(n-1)-2\equiv 0\pmod{6}$ and the $(n-1)$-tuple can be transformed into $(s_0,s_2)^{(n-3)/2}\bullet (s_0,t_0)$ by successive application of elementary transformations. We note that the original $n$-tuple $(g_1,\ldots,g_n)$ must be inverse-free and, after combining $s_0$ and $t_0$ into $1$ and removing it, we make the result an inverse-free tuple of short elements. By Proposition \ref{proposition::reduce_and_reloading}, it implies a sequence of elementary transformations sending the $n$-tuple $(g_1,\ldots,g_n)$ into $(s_0,a^2,s_1)\bullet (s_0,s_2)^{(n-3)/2}$.
In any case, the substitutions of $(s_0,s_2)$, $(s_1,s_0)$ and $(s_2,s_1)$ transform the $n$-tuple $(g_1,\ldots,g_n)$ into $(a^2,s_0,s_2)\bullet (s_0,s_2)^{(n-3)/2}$.
\end{proof}

The following lemma can be proved similarly and we omit the details.
\begin{lemma}
\label{lemma::a,t}
Let $g_1,\ldots,g_n$ be $a$, $t_0$, $t_1$ or $t_2$ such that only one of them is equal to $a$ and $g_1\cdots g_n=1$.
Then, $n\equiv 3\pmod{6}$ and the $n$-tuple $(g_1,\ldots,g_n)$ is Hurwitz equivalent to
\begin{equation*}
    (a,t_2,t_0)\bullet (t_0,t_2)^{(n-3)/2}.
\end{equation*}
\end{lemma}

\begin{lemma}
\label{lemma::b,s}
Let $g_1,\ldots,g_n$ be $b$, $s_0$, $s_1$ or $s_2$ such that only one of them is equal to $b$ and $g_1\cdots g_n=1$.
Then, $n\equiv 4\pmod{6}$ and the $n$-tuple $(g_1,\ldots,g_n)$ is Hurwitz equivalent to
\begin{equation*}
    (b,s_0,s_2,s_0)\bullet (s_0,s_2)^{(n-4)/2}.
\end{equation*}
\end{lemma}

\begin{proof}
Without loss of generality, we assume that all the $n$-tuple in $\{b,s_0,s_1,s_2\}$ resulting from the successive application of elementary transformations on $(g_1,\ldots,g_n)$ contain no consecutive sub-tuples of the form $(s_0,s_2)^3$.

Take the $n$-tuple in $\{b,s_0,s_1,s_2\}$ that starts with $b$ and contains the minimal number of components equal to $s_1$ among all resulting tuples that we can get using elementary transformations on $(g_1,\ldots,g_n)$, still denoted by $(g_1,\ldots,g_n)$. We write it as
\begin{equation*}
    (b)\bullet
    \Big(\prod_{j=1}^{n_0} (s_2)^{u_{0,j}}\bullet (s_0)^{v_{0,j}}\Big)\bullet
    \Big(
    \prod_{i=1}^{\mu}
        (s_1)^{\lambda_{i}}
        \bullet
        \Big(\prod_{j=1}^{n_i} (s_2)^{u_{i,j}}\bullet (s_0)^{v_{i,j}}\Big)
    \Big)
\end{equation*}
with $\mu\ge 0$, $\lambda_1,\ldots,\lambda_{\mu}\ge 1$, $n_0,n_{\mu}\ge 0$, $n_1,\ldots,n_{\mu-1}\ge 1$ and $u_{i,j}$, $v_{i,j}\ge 0$ for $i=0,\ldots,\mu$, $j=1,\ldots,n_i$ where $u_{i,j}\ge 1$ for $j>1$ and $v_{i,j}\ge 1$ for $j<n_i$.
The minimality further requires that $u_{i,1}\ge 1$ for $i=1,\ldots,\mu$ and $v_{i,n_i}\ge 1$ for $i=0,\ldots,\mu-1$.

Assume that $(g_1,\ldots,g_n)$ does not start with $(b,s_2)$ or $(b)\bullet (s_0)^{v_{0,1}}\bullet (s_2,s_0)$, nor end with $s_0$ or $(s_2,s_0)\bullet (s_2)^{u_{\mu,n_{\mu}}}$. Then, $u_{0,1}=v_{\mu,n_{\mu}}=0$, either $n_0\le 1$ or $u_{0,2}\ge 2$ and either $n_{\mu}\le 1$ or $v_{\mu,n_{\mu}-1}\ge 2$.
Applying Proposition \ref{proposition::reduced_form_s0_s2} with the above restrictions, we obtain that the reduced form of $g_1\cdots g_n$ is not equal to $1$, which is a contradiction. Hence, using the substitution of $(s_0,s_2,s_0)$ and $(s_2,s_0,s_2)$ and a cyclic permutation if necessary, the $n$-tuple is transformed into an $n$-tuple in $\{b,s_0,s_1,s_2\}$, still denoted by $(g_1,\ldots,g_n)$, such that $(g_1,g_2)$ is equal to either $(b,s_2)$ or $(s_0,b)$. We combine $g_1$ and $g_2$ into their product and replace $(g_1,\ldots,g_n)$ with an $(n-1)$-tuple in $\{a^2,s_0,s_1,s_2\}$ starting with $a^2$.
By Lemma \ref{lemma::a^2,s} we get $n-1\equiv 3\pmod{6}$ and, by successive application of elementary transformations, the $(n-1)$-tuple can be transformed into $(a^2,s_0,s_2)\bullet (s_0,s_2)^{(n-4)/2}$.
By Proposition \ref{proposition::reduce_and_reloading}, we obtain an $n$-tuple of the form either $(b,s_2,s_0,s_2)\bullet (s_0,s_2)^{(n-4)/2}$ or $(s_0,b,s_0,s_2)\bullet (s_0,s_2)^{(n-4)/2}$ from $(g_1,\ldots,g_n)$ using elementary transformations. The substitution of $(s_0,s_2,s_0)$ and $(s_2,s_0,s_2)$ completes the proof.
\end{proof}

Again, the following lemma is similar and we omit the proof.
\begin{lemma}
\label{lemma::b,t}
Let $g_1,\ldots,g_n$ be $b$, $t_0$, $t_1$ or $t_2$ such that only one of them is equal to $b$ and $g_1\cdots g_n=1$.
Then, $n\equiv 4\pmod{6}$ and the $n$-tuple $(g_1,\ldots,g_n)$ is Hurwitz equivalent to
\begin{equation*}
    (b,t_0,t_2,t_0)\bullet (t_0,t_2)^{(n-4)/2}.
\end{equation*}
\end{lemma}

\begin{lemma}
\label{lemma::a,s}
Let $g_1,\ldots,g_n$ be $a$, $s_0$, $s_1$ or $s_2$ such that only one of them is equal to $a$ and $g_1\cdots g_n=1$.
Then, $n\equiv 5\pmod{6}$ and the $n$-tuple $(g_1,\ldots,g_n)$ is Hurwitz equivalent to 
\begin{equation*}
    (a,s_0,s_0,s_2,s_0)\bullet (s_0,s_2)^{(n-5)/2}.
\end{equation*}
\end{lemma}

\begin{proof}
Without loss of generality, we assume that each $n$-tuple in $\{a,s_0,s_1,s_2\}$ that results from the successive application of elementary transformations on $(g_1,\ldots,g_n)$ contains no consecutive sub-tuples of the form $(s_0,s_2)^3$.

Take the $n$-tuple in $\{a,s_0,s_1,s_2\}$ that starts with $a$ and contains the minimal number of components equal to $s_1$ among all resulting tuples that we can get using elementary transformations on $(g_1,\ldots,g_n)$, still denoted by $(g_1,\ldots,g_n)$. Assume that
\[
(g_1,g_2)\neq (a,s_0),
(g_n,g_1)\neq (s_2,a)
\text{ and }
(g_n,g_1,g_2)\neq (s_1,a,s_1).
\]
Then, the $n$-tuple $(g_1,\ldots,g_n)$ is written as
\begin{equation*}
    (a)\bullet
    \Big(\prod_{j=1}^{n_0} (s_2)^{u_{0,j}}\bullet (s_0)^{v_{0,j}}\Big)\bullet
    \Big(
    \prod_{i=1}^{\mu}
        (s_1)^{\lambda_{i}}
        \bullet
        \Big(\prod_{j=1}^{n_i} (s_2)^{u_{i,j}}\bullet (s_0)^{v_{i,j}}\Big)
    \Big)
\end{equation*}
with $\mu\ge 0$, $\lambda_1,\ldots,\lambda_{\mu}\ge 1$, $n_0,n_{\mu}\ge 0$ but $n_0+n_{\mu}\ge 1$, $n_1,\ldots,n_{\mu-1}\ge 1$ and $u_{i,j},v_{i,j}\ge 1$ for $i=0,\ldots,\mu$, $j=1,\ldots,n_i$. Applying Proposition \ref{proposition::reduced_form_s0_s2}, we notice that the reduced form of $g_1\cdots g_n$ is not equal to $1$, which is a contradiction. Hence, one of the above requirements cannot be fulfilled.

If either $(g_1,g_2)=(a,s_0)$ or $(g_n,g_1)=(s_2,a)$, using a cyclic permutation if necessary, then the pair $(g_1,g_2)$ is equal to either $(a,s_0)$ or $(s_2,a)$.
We combine $g_1$ and $g_2$ into a single $b$ and replace $(g_1,\ldots,g_n)$ with an $(n-1)$-tuple in $\{b,s_0,s_1,s_2\}$ starting with the only $b$.
By Lemma \ref{lemma::b,s}, we get $n-1\equiv 4\pmod{6}$ and there exists a finite sequence of elementary transformations that transforms the $(n-1)$-tuple into $(b,s_0,s_2,s_0)\bullet (s_0,s_2)^{(n-5)/2}$. Proposition \ref{proposition::reduce_and_reloading} implies that $(g_1,\ldots,g_n)$ can be transformed by elementary transformations into either
\begin{equation*}
(a,s_0,s_0,s_2,s_0)\bullet (s_0,s_2)^{(n-5)/2} \text{ or }
(s_2,a,s_0,s_2,s_0)\bullet (s_0,s_2)^{(n-5)/2}.
\end{equation*}
If $(g_n,g_1,g_2)=(s_1,a,s_1)$, by a cyclic permutation and an elementary transformation, then the $n$-tuple can be transformed into $(s_1,s_0,a,g_3,\ldots,g_{n-1})$. We combine $s_1$ and $s_0$ into $a$, further combine $a$ and $a$ into a single $a^2$ and replace the $n$-tuple with an $(n-2)$-tuple in $\{a^2,s_0,s_1,s_2\}$ starting with the only $a^2$.
By Lemma \ref{lemma::a^2,s}, we get $n-2\equiv 3\pmod{6}$ and the $(n-2)$-tuple can be transformed by elementary transformations into $(a^2,s_0,s_2)\bullet (s_0,s_2)^{(n-5)/2}$. By Proposition \ref{proposition::reduce_and_reloading}, the $n$-tuple $(g_1,\ldots,g_n)$ can be transformed into $(s_{j+1},s_j,a,s_0,s_2)\bullet (s_0,s_2)^{(n-5)/2}$ with some $j$.
The substitutions of $(s_0,s_2)$, $(s_1,s_0)$, $(s_2,s_1)$ and the substitution of $(s_0,s_2,s_0)$, $(s_2,s_0,s_2)$ conclude the lemma.
\end{proof}

Once again, the following lemma is similar and we omit the proof.
\begin{lemma}
\label{lemma::a^2,t}
Let $g_1,\ldots,g_n$ be $a^2$, $t_0$, $t_1$ or $t_2$ such that only one of them is equal to $a^2$ and $g_1\cdots g_n=1$.
Then, $n\equiv 5\pmod{6}$ and the $n$-tuple $(g_1,\ldots,g_n)$ is Hurwitz equivalent to
\begin{equation*}
    (a^2,t_0,t_2,t_0,t_0)\bullet (t_0,t_2)^{(n-5)/2}.
\end{equation*}
\end{lemma}

\begin{proof}[Proof of Theorem \ref{theorem::Livne mathcalS}]
By Theorem \ref{theorem::Livne S->short}, we are able to transform $(g_1,\ldots,g_n)$ into either
\begin{equation*}
    (k_1,k_1^{-1},\ldots,k_s,k_s^{-1},l_1,l_1,l_1,\ldots,l_t,l_t,l_t)
\end{equation*}
with $s,t\ge 0$, $2s+3t=n$ and $l_j^3=1$ for each $j=1,\ldots,t$, or
\begin{equation*}
    (h_1,\ldots,h_{\mu})\bullet (s_0,t_0)^{m_{st}}\bullet (a,a^2)^{m_a}\bullet (b,b)^{m_b}\bullet (a,a,a)^{n_0}\bullet (a^2,a^2,a^2)^{n_1}
\end{equation*}
with $\mu>0$, $m_{st}$, $m_a$, $m_b$, $n_0$, $n_1\ge 0$, $\mu+2(m_{st}+m_a+m_b)+3(n_0+n_1)=n$ such that $(h_1,\ldots,h_{\mu})$ is an inverse-free $\mu$-tuple of short elements.
The former case just so happens to be the first case of Theorem \ref{theorem::Livne mathcalS}, therefore we consider only the latter and suppose that $\mu>0$.
As $(h_1,\ldots,h_{\mu})$ is inverse-free, it contains at most two $a$'s, at most two $a^2$'s, at most one $b$ and it does not contain both $a$ and $a^2$.
Let $\mathcal{A}$ be the set of elements in $(h_1,\ldots,h_{\mu})$.
Let $I_a$ and $I_b$ be the numbers of components conjugate to some power of $a$ and $b$ respectively.
Take $\mathcal{A}_s=\mathcal{A}\cap\{s_0,s_1,s_2\}$ and
$\mathcal{A}_t=\mathcal{A}\cap\{t_0,t_1,t_2\}$. 

\textbf{Step 1.} Suppose that $I_a+I_b\le 1$.
Proposition \ref{proposition::ST} shows that either $\mathcal{A}_s$ or $\mathcal{A}_t$ is empty. Thus, by Lemma \ref{lemma::ST}, \ref{lemma::a,s}, \ref{lemma::a,t}, \ref{lemma::a^2,s}, \ref{lemma::a^2,t}, \ref{lemma::b,s} and \ref{lemma::b,t}, by elementary transformations the inverse-free tuple $(h_1,\ldots,h_{\mu})$ can be transformed into one of the following partial normal forms.
\begin{enumerate}[-,topsep=0pt,itemsep=-1ex,partopsep=1ex,parsep=1ex]
    \item $(s_0,s_2,s_0,s_2,s_0,s_2)^{\mu/6}$,
    $(t_0,t_2,t_0,t_2,t_0,t_2)^{\mu/6}$
    where $\mu\equiv 0\pmod{6}$;
    \item $(a^2,s_0,s_2)\bullet (s_0,s_2)^{(\mu-3)/2}$,
    $(a,t_2,t_0)\bullet (t_0,t_2)^{(\mu-3)/2}$
    where $\mu\equiv 3\pmod{6}$;
    \item $(b,s_0,s_2,s_0)\bullet (s_0,s_2)^{(\mu-4)/2}$,
    $(b,t_0,t_2,t_0)\bullet (t_0,t_2)^{(\mu-4)/2}$
    where $\mu\equiv 4\pmod{6}$;
    \item $(a,s_0,s_0,s_2,s_0)\bullet (s_0,s_2)^{(\mu-5)/2}$,
    $(a^2,t_0,t_2,t_0,t_0)\bullet (t_0,t_2)^{(\mu-5)/2}$
    where $\mu\equiv 5\pmod{6}$.
\end{enumerate}

\textbf{Step 2.} Suppose that $I_a=1=I_b$ and $a^{\epsilon}\in \mathcal{A}$ with $\epsilon=\pm 1$. It is clear that $\mu\ge 3$.

If there exists an element $h'\in \mathcal{A}$ equal to one of $ba^{\epsilon}$, $a^{\epsilon}b$, $a^{-\epsilon}ba^{-\epsilon}$ then, using elementary transformations, we place $a^{\epsilon}$ and $h'$ in adjacent positions that form a pair $(a^{\epsilon},h')$. The pair is further replaced by the product $a^{\epsilon} h'$ and we replace $(h_1,\ldots,h_{\mu})$ with a $(\mu-1)$-tuple, say $(y_1,\ldots, y_{\mu-1})$. Each of $y_1,\ldots,y_{\mu-1}$ is short, one of them is equal to $b$ and each of the rest is neither a power of $a$ nor $b$. By Theorem \ref{theorem::Livne S->short}, Proposition \ref{proposition::ST} and Lemma \ref{lemma::b,s}, \ref{lemma::b,t}, the $(\mu-1)$-tuple $(y_1,\ldots,y_{\mu-1})$ can be transformed into either
\begin{equation*}
    (b,s_0,s_2,s_0)\bullet (s_0,s_2,s_0,s_2,s_0,s_2)^u\bullet (s_0,t_0)^v
    \text{ or }
    (b,t_0,t_2,t_0)\bullet (t_0,t_2,t_0,t_2,t_0,t_2)^u\bullet (s_0,t_0)^v
\end{equation*}
with $u,v\ge 0$ and $5+6u+2v=\mu$. Proposition \ref{proposition::reduce_and_reloading} shows that $(h_1,\ldots,h_{\mu})$ can be transformed into one of them with exactly one of the following adjustments: replace an $s_0$ (resp. $s_2$, $t_0$, $t_2$) with $(a,t_2)$ (resp. $(a,t_1)$, $(a^2,s_1)$, $(a^2,s_0)$). The substitutions
\begin{align*}
    &(b,a,t_2,s_2,s_0)
    \rightarrow (b,t_2,s_2,a,s_0)
    \rightarrow (b,a,t_0,s_0,s_0)
    \rightarrow (b,a,s_0)\bullet (t_0,s_0)
    \rightarrow (a,b,s_2)\bullet (s_0,t_0),\\
    &(b,s_0,a,t_1,s_0)
    \rightarrow (b,s_0,t_0,a,s_0)
    \rightarrow (b,a,s_0)\bullet (s_0,t_0)
    \rightarrow (a,b,s_2)\bullet (s_0,t_0),\\
    &(b,s_0,s_2,a,t_2)
    \rightarrow (b,s_2,s_1,a,t_2)
    \rightarrow (b,s_2,s_1,t_1,a)
    \rightarrow (b,s_2,s_0,t_0,a)
    \rightarrow (a,b,s_2)\bullet (s_0,t_0),\\
    &(b,s_0,s_2,s_0)\bullet (a,t_2,s_2,s_0,s_2,s_0,s_2)
    \rightarrow (b,s_0,s_2,a,t_2)\bullet (s_0,s_2,s_0,s_2,s_0,s_2),\\
    &(b,s_0,s_2,s_0)\bullet (s_0,a,t_1,s_0,s_2,s_0,s_2)
    \rightarrow (s_0,b,s_0,a,t_1)\bullet (s_0,s_2,s_0,s_2,s_0,s_2),\\
    &(b,s_0,s_2,s_0)\bullet (a,t_2,t_0)
    \rightarrow (b,s_0,s_2,s_0)\bullet (t_0,a,t_2)
    \rightarrow (b,s_0,s_2,a,t_2)\bullet (s_0,t_0)
\end{align*}
and their symmetrical manners further transform the resulting $\mu$-tuple into one of the following partial normal forms.
\begin{enumerate}[-,topsep=0pt,itemsep=-1ex,partopsep=1ex,parsep=1ex]
    \item $(a,b,s_2)\bullet (s_0,s_2,s_0,s_2,s_0,s_2)^u\bullet (s_0,t_0)^{v+1}$ with $u,v\ge 0$;
    \item $(a,t_2,t_0)\bullet (b,t_0,t_2,t_0)\bullet (t_0,t_2,t_0,t_2,t_0,t_2)^u\bullet (s_0,t_0)^{v-1}$ with $u\ge 0$ and $v\ge 1$;
    \item $(a^2,b,t_0)\bullet (t_0,t_2,t_0,t_2,t_0,t_2)^u\bullet (s_0,t_0)^{v+1}$ with $u,v\ge 0$;
    \item $(a^2,s_0,s_2)\bullet (b,s_0,s_2,s_0)\bullet (s_0,s_2,s_0,s_2,s_0,s_2)^u\bullet (s_0,t_0)^{v-1}$ with $u\ge 0$ and $v\ge 1$.
\end{enumerate}

Otherwise, one of $\mathcal{A}_s$ and $\mathcal{A}_t$ is empty. If there exists an element $h'\in \mathcal{A}$ equal to either $ba^{-\epsilon}$ or $a^{-\epsilon}b$ then, using elementary transformations, we place $a^{\epsilon}$ and $h'$ in adjacent positions such that their product is equal to $b$.
The pair is further replaced by a single $b$.
Therefore, the resulting $(\mu-1)$-tuple has exactly two different components conjugate to $b$ and the rest are either conjugates of $s_0$ or conjugates of $t_0$. Applying Theorem \ref{theorem::Livne S->short}, we have shown in Step 1 that such an $(\mu-1)$-tuple can be transformed by elementary transformations into either $(s_0,s_2,s_0,s_2,s_0,s_2)^{(\mu-3)/6}\bullet (b,b)$ or $(t_0,t_2,t_0,t_2,t_0,t_2)^{(\mu-3)/6}\bullet (b,b)$. Proposition \ref{proposition::reduce_and_reloading} implies that $(h_1,\ldots,h_{\mu})$ can be transformed into a concatenation of either $(s_0,s_2)^{(\mu-3)/2}$ or $(t_0,t_2)^{(\mu-3)/2}$ and one of the following triples, which can be further transformed into a result consistent with the previous case.
\begin{align*}
    (a,s_0,b), (a^2,t_2,b), (s_2,a,b), (t_0,a^2,b),
    (b,a,s_0), (b,a^2,t_2), (b,s_2,a), (b,t_0,a^2).
\end{align*}

\textbf{Step 3.} We consider the last case left in Step 2 where $\mathcal{A}=\{a^{\epsilon},b,a^{\epsilon}ba^{\epsilon}\}$ and $I_a=1$.

In fact, we have $\mu\ge 4$. By elementary transformations we place $a^{\epsilon}$ and two different $a^{\epsilon}ba^{\epsilon}$'s in adjacent positions that form a triple of the form $(a^{\epsilon}ba^{\epsilon},a^{\epsilon},a^{\epsilon}ba^{\epsilon})$. The triple can be further transformed into $(a^{\epsilon}ba^{\epsilon},a^{-\epsilon}b,a^{\epsilon})$. We combine the first two components into $a^{\epsilon}$ and then rewrite the triple as a single $a^{-\epsilon}$. The resulting $(\mu-2)$-tuple is composed of $a^{-\epsilon},b$ and several $a^{\epsilon}ba^{\epsilon}$. Step 2 has shown that such a tuple can be transformed by elementary transformations into either
\begin{equation*}
    (a,t_2,t_0)\bullet (b,t_0,t_2,t_0)\bullet (t_0,t_2,t_0,t_2,t_0,t_2)^u
    \text{ or }
    (a^2,s_0,s_2)\bullet (b,s_0,s_2,s_0)\bullet (s_0,s_2,s_0,s_2,s_0,s_2)^u
\end{equation*}
with $u\ge 0$. By Proposition \ref{proposition::reduce_and_reloading}, elementary transformations can transform $(h_1,\ldots,h_{\mu})$ into either
\begin{align*}
    (t_j,t_{j+1},a^2,t_2,t_0)\bullet (b,t_0,t_2,t_0)\bullet (t_0,t_2,t_0,t_2,t_0,t_2)^u\\
    \text{ or }
    (s_j,s_{j-1},a,s_0,s_2)\bullet (b,s_0,s_2,s_0)\bullet (s_0,s_2,s_0,s_2,s_0,s_2)^u
\end{align*}
that can be further transformed into the result in Step 2 using the substitutions:
\begin{align*}
    (t_j,t_{j+1},a^2,t_2,t_0)\bullet (b,t_0,t_2,t_0)\rightarrow
    (a^2,t_2,t_0,t_2,t_0)\bullet (t_0,t_2,t_0,b)\rightarrow
    (a^2,b,t_0)\bullet (t_0,t_2,t_0,t_2,t_0,t_2),\\
    (s_j,s_{j-1},a,s_0,s_2)\bullet (b,s_0,s_2,s_0)\rightarrow
    (a,s_0,s_2,s_0,s_2)\bullet (s_0,s_2,s_0,b)\rightarrow
    (a,b,s_2)\bullet (s_0,s_2,s_0,s_2,s_0,s_2).
\end{align*}

\textbf{Step 4.} Suppose that $I_a=2$.

We place the powers of $a$ in adjacent positions and replace them with their product. The resulting $(\mu-1)$-tuple contains exactly one power of $a$ and can be transformed by elementary transformation into one of the eight partial normal forms introduced in Step 1 and 2. By Proposition \ref{proposition::reduce_and_reloading}, one can simply rewrite the powers of $a$ as pairs of powers of $a$ and obtain eight more partial normal forms. Replacing the inverse-free tuple $(h_1,\ldots,h_{\mu})$ of short elements by a partial normal form in the resulting tuple of the elementary transformations on $(g_1,\ldots,g_n)$, we finish the proof of the theorem.
\end{proof}

%% file: Section4/almost.tex
\subsection{Conjugates of almost short elements and tuples}

Suppose that $(g_1,\ldots, g_n)$ is an $n$-tuple with each $g_i$ conjugate to some almost short element (i.e. the component $g_i$ is conjugate to either $a$, $a^2$, $b$, $s_1$, $t_1$ or $ababa$).
In this subsection, we first show that by successive application of elementary transformations the $n$-tuple can be transformed into 
\begin{equation*}
    \prod_{i=1}^m (Q_i^{-1}\tau_{i,1}Q_i,\ldots,Q_i^{-1}\tau_{i,n_i}Q_i)
\end{equation*}
with $m\ge 1$, $\sum_{i=1}^m n_i=n$, $Q_i\in G$, $\tau_{i,j}\in \mathcal{S}_2$ such that $\tau_{i,1}\cdots\tau_{i,n_i}=1$ for $i=1,\ldots,m$ and $j=1,\ldots,n_i$.
For the concatenation of $(g_1,\ldots,g_n)$ and a fixed tuple, we further show a result extremely similar to Theorem \ref{theorem::Livne mathcalS}.

The first part of this subsection follows a similar line as in Subsection \ref{subsection::Conjugates of short elements and their tuples}.
Proposition \ref{proposition::reduce_and_reloading_hard} is an analog to Proposition \ref{proposition::reduce_and_reloading}.
Lemmata \ref{lemma::tau1_tau2_tau3_triple}, \ref{lemma::tau1tau2=a}, \ref{lemma::tau1tau2=b}, \ref{lemma::tau1tau2=bab}, \ref{lemma::tau1tau2=aaba}, \ref{lemma::tau1tau2=aba} and \ref{lemma::tau1tau2=ababa}, which have technicalities referring to Lemma \ref{lemma::aba,a^2->ab},
will be used to prove Proposition \ref{proposition::reduce_and_reloading_hard}.

\begin{table}[htb]
\centering
\begin{adjustbox}{width=\textwidth}
\begin{tabular}{ c|ccccccccccccccccccc| } 
 \diagbox{$g_i$}{$g_{i+1}$}

          &$a$    &$a^2$    &$b$   &$a^2b$  &$aba$   &$ba^2$   &$ba$  &$a^2ba^2$   &$ab$  &$bab$  &$ba^2b$  &$a^2ba$  &$aba^2$ &$a^2bab$&$ababa$ &$baba^2$ &$ba^2ba$&$a^2ba^2ba^2$ &$aba^2b$\\
\hline
     $a$    &$a^2$    &$1$   &$ab$    &$b$  &$a^2ba$  &$aba^2$  &$aba$   &$ba^2$   &$a^2b$       &       &   &$ba$  &$a^2ba^2$  &$bab$       &       &       &       &       &\\
     $a^2$    &$1$    &$a$   &$a^2b$   &$ab$   &$ba$  &$a^2ba^2$  &$a^2ba$  &$aba^2$    &$b$       &       &  &$aba$   &$ba^2$       &       &       &       &       &  &$ba^2b$\\
     $b$   &$ba$   &$ba^2$    &$1$  &$ba^2b$       &    &$a^2$    &$a$       &  &$bab$   &$ab$   &$a^2b$       &       &       &       &  &$aba^2$  &$a^2ba$       &       &\\
    $a^2b$  &$a^2ba$  &$a^2ba^2$    &$a^2$       &       &    &$a$    &$1$       &       &    &$b$   &$ab$       &       &       &       &   &$ba^2$  &$aba$       &       &\\
   $aba$  &$aba^2$   &$ab$       &    &$a$       &       &       &    &$1$ &$aba^2b$       &       &    &$a^2$       &   &$a^2b$       &       &       &   &$ba^2$       &\\
    $ba^2$    &$b$   &$ba$  &$ba^2b$  &$bab$    &$a$       &       & &$baba^2$    &$1$       &       &       &    &$a^2$       &  &$aba$       &       &       &   &$a^2b$\\
    $ba$   &$ba^2$    &$b$  &$bab$    &$1$ &$ba^2ba$       &       &    &$a^2$  &$ba^2b$       &       &    &$a$       &   &$ab$       &       &       &  &$a^2ba^2$       &\\
   $a^2ba^2$   &$a^2b$  &$a^2ba$       & &$a^2bab$    &$1$       &       &       &    &$a^2$       &       &       &    &$a$       &   &$ba$       &       &       &   &$ab$\\
    $ab$  &$aba$  &$aba^2$    &$a$       &       &    &$1$    &$a^2$       &       &   &$a^2b$    &$b$       &       &       &       &  &$a^2ba^2$   &$ba$       &       &\\
   $bab$       &       &   &$ba$       &       &    &$b$   &$ba^2$       &       &  &$ba^2b$    &$1$       &       &       &       &       &    &$a$       &       &\\
   $ba^2b$       &       &   &$ba^2$       &       &   &$ba$    &$b$       &       &    &$1$  &$bab$       &       &       &       &    &$a^2$       &       &       &\\
   $a^2ba$  &$a^2ba^2$   &$a^2b$       &    &$a^2$       &       &       &    &$a$       &       &       &    &$1$&$a^2ba^2ba^2$    &$b$       &       &       &  &$aba^2$       &\\
   $aba^2$   &$ab$  &$aba$       &       &    &$a^2$       &       &       &    &$a$       &       &&$ababa$    &$1$       &  &$a^2ba$       &       &       &    &$b$\\
  $a^2bab$       &       &  &$a^2ba$       &       &   &$a^2b$  &$a^2ba^2$       &       &       &    &$a^2$       &       &       &       &&$a^2ba^2ba^2$    &$1$       &       &\\
 $ababa$       &       &       &  &$aba$       &       &       &   &$ab$       &       &       &  &$aba^2$       & &$aba^2b$       &       &       &    &$1$       &\\
  $baba^2$  &$bab$       &       &       &   &$ba^2$       &       &       &   &$ba$       &       &       &    &$b$       & &$ba^2ba$       &       &       &    &$1$\\
  $ba^2ba$       &  &$ba^2b$       &   &$ba^2$       &       &       &   &$ba$       &       &       &    &$b$       &    &$1$       &       &       & &$baba^2$       &\\
 $a^2ba^2ba^2$       &       &       &       &   &$a^2b$       &       &       &  &$a^2ba^2$       &       &       &  &$a^2ba$       &    &$1$       &       &       & &$a^2bab$\\
  $aba^2b$       &       &  &$aba^2$       &       &  &$aba$   &$ab$       &       &    &$a$       &       &       &       &       &    &$1$&$ababa$       &       &\\
\hline
\end{tabular}
\end{adjustbox}
\caption{Some pairs $(g_i,g_{i+1})$ of almost short elements and the products $g_ig_{i+1}$.}
\label{table::almost short}
\end{table}

We introduce some pairs of almost short elements in Table \ref{table::almost short} as in Subsection \ref{subsection::Conjugates of short elements and their tuples}.
Broadly speaking, each pair of almost short elements in Table \ref{table::almost short} behaves well under the contraction operation introduced in Subsection \ref{subsection::contractions-restorations}, which is explained in lemmata \ref{lemma::tau1tau2=a}, \ref{lemma::tau1tau2=b}, \ref{lemma::tau1tau2=bab}, \ref{lemma::tau1tau2=aaba}, \ref{lemma::tau1tau2=aba} and \ref{lemma::tau1tau2=ababa}.
Besides, each pair of almost short elements not in Table \ref{table::almost short} satisfies the inequality $m_i\le \min\{\frac{l(g_i)+1}{2}, \frac{l(g_{i+1})+1}{2}\}$ which is the first step for Lemma \ref{lemma::m_i-1_and_m_i_hard} (a). (See Lemma \ref{lemma::m_i-1_and_m_i} for the precise definition of $m_i$.)
Furthermore, Table \ref{table::almost short} has to fulfil some irregular requirements which appear in the proofs of Lemma \ref{lemma::m_i-1_and_m_i_hard} and Theorem \ref{theorem::Livne S2->almost short}.
Unfortunately, we do not have high conviction in sifting out the pairs of almost short elements.
What is worse, Theorem \ref{theorem::Livne S2->almost short} needs a patch based on Lemma \ref{lemma::tau1_tau2_tau3_triple} which considers a triple of almost short elements.

\begin{lemma}
\label{lemma::tau1_tau2_tau3_triple}
Let $(\tau_1,\tau_2,\tau_2)$ be a triple of the form
\begin{equation*}
(a^{-\epsilon}ba^{\epsilon}b,ba^{\epsilon}b,ba^{\epsilon}ba^{-\epsilon})
\text{ or }
(ba^{-\epsilon}ba^{\epsilon},a^{\epsilon},a^{\epsilon}ba^{-\epsilon}b)
\end{equation*}
with $\epsilon=\pm 1$.
Set $(g_1,g_2,g_3)=(Q^{-1}\tau_1Q,Q^{-1}\tau_2Q,Q^{-1}\tau_3Q)$ with $Q\in G$ and suppose that $Q^{-1}\tau_1\tau_2\tau_3Q\in\mathcal{S}_2$.
Then $(g_1,g_2,g_3)$ is Hurwitz equivalent to a triple of almost short elements.
\end{lemma}
\begin{proof}
We only consider the triple $(\tau_1,\tau_2,\tau_2)=(a^{-\epsilon}ba^{\epsilon}b,ba^{\epsilon}b,ba^{\epsilon}ba^{-\epsilon})$.

Since $Q^{-1}\tau_1\tau_2\tau_3Q=Q^{-1}a^{\epsilon}Q$ is almost short, $Q$ is one of $1$, $a^{\pm\epsilon}$, $b$ and $a^{\pm\epsilon}b$. In the case that $Q=b$, the triple $(g_1,g_2,g_3)=(ba^{-\epsilon}ba^{\epsilon},a^{\epsilon},a^{\epsilon}ba^{-\epsilon}b)$ is already of almost short elements. In the cases that $Q=a^{\epsilon}$ or $a^{\epsilon}b$, the lemma follows from the following substitutions.
\begin{align*}
    (a^{-\epsilon}\tau_1a^{\epsilon},a^{-\epsilon}\tau_2a^{\epsilon},a^{-\epsilon}\tau_3a^{\epsilon})
    &=(a^{\epsilon}ba^{\epsilon}ba^{\epsilon},a^{-\epsilon}ba^{\epsilon}ba^{\epsilon},a^{-\epsilon}ba^{\epsilon}b)\\
    &\xrightarrow{R_1}
    (a^{-\epsilon}ba^{\epsilon}ba^{\epsilon},a^{-\epsilon}ba^{-\epsilon}ba^{-\epsilon}ba^{-\epsilon}ba^{\epsilon},a^{-\epsilon}ba^{\epsilon}b)\\
    &\xrightarrow{R_2}
    (a^{-\epsilon}ba^{\epsilon}ba^{\epsilon},a^{-\epsilon}ba^{\epsilon}b,ba^{\epsilon}ba^{-\epsilon})
    \xrightarrow{R_1}
    (a^{-\epsilon}ba^{\epsilon}b,ba^{\epsilon}b,ba^{\epsilon}ba^{-\epsilon}).\\
    (ba^{-\epsilon}\tau_1a^{\epsilon}b,ba^{-\epsilon}\tau_2a^{\epsilon}b,ba^{-\epsilon}\tau_3a^{\epsilon}b)
    &=(ba^{\epsilon}ba^{\epsilon}ba^{\epsilon}b,ba^{-\epsilon}ba^{\epsilon}ba^{\epsilon}b,ba^{-\epsilon}ba^{\epsilon})\\
    &\xrightarrow{R_2}
    (ba^{\epsilon}ba^{\epsilon}ba^{\epsilon}b,ba^{-\epsilon}ba^{\epsilon},a^{\epsilon})\\
    &\xrightarrow{R_1}
    (ba^{-\epsilon}ba^{\epsilon},a^{-\epsilon}ba^{-\epsilon}ba^{-\epsilon},a^{\epsilon})
    \xrightarrow{R_2}
    (ba^{-\epsilon}ba^{\epsilon},a^{\epsilon},a^{\epsilon}ba^{-\epsilon}b).
\end{align*}
For the rest two cases, the approach is similar. 
\end{proof}

\begin{lemma}
\label{lemma::tau1tau2=a}
Let $(\tau_1,\tau_2)$ be a pair of almost short elements in Table \ref{table::almost short} such that $\tau_1\tau_2$ is a power of $a$.
Set $(g_1,g_2)=(Q^{-1}\tau_1Q,Q^{-1}\tau_2Q)$ with $Q\in G$ and suppose that $Q^{-1}\tau_1\tau_2Q$ is almost short.
Then $(g_1,g_2)$ is Hurwitz equivalent to a pair of almost short elements. 
\end{lemma}
\begin{proof}
The pair $(\tau_1,\tau_2)$ must be one of
\begin{align*}
&(a^{-\epsilon},a^{-\epsilon}),\\
&(b,ba^{\epsilon}),
(a^{\epsilon}b,b),
(a^{-\epsilon}ba^{-\epsilon},a^{\epsilon}ba^{-\epsilon}),
(a^{-\epsilon}ba^{\epsilon},a^{-\epsilon}ba^{-\epsilon}),
(ba^{\epsilon},a^{-\epsilon}ba^{\epsilon}),
(a^{\epsilon}ba^{-\epsilon},a^{\epsilon}b),\\
&(a^{-\epsilon}b,ba^{-\epsilon}),
(a^{\epsilon}ba^{\epsilon},a^{-\epsilon}b),
(ba^{-\epsilon},a^{\epsilon}ba^{\epsilon}),\\
&(a^{\epsilon}ba^{-\epsilon}b,ba^{\epsilon}b),
(ba^{\epsilon}b,ba^{-\epsilon}ba^{\epsilon})
\end{align*}
and $Q\in\{1,a^{\epsilon},a^{-\epsilon},b,a^{\epsilon}b,a^{-\epsilon}b\}$ with $\epsilon=\pm 1$. Now we fix $\epsilon=\pm 1$.

When $(\tau_1,\tau_2)=(a^{-\epsilon},a^{-\epsilon})$, the pair $(g_1,g_2)$ is a pair of almost short elements. When one of $g_1$, $g_2$ is conjugate to $b$ and the other one is conjugate to $a^{\epsilon}b$, either $(g_1,g_2)$ is a pair of almost short elements or $(g_1,g_2)$ is equal to one of $(a^{\epsilon}b,ba^{-\epsilon}ba^{\epsilon}b)$, $(ba^{\epsilon}ba^{-\epsilon}b,ba^{\epsilon})$, $(ba^{-\epsilon}ba^{\epsilon}b,ba^{-\epsilon}ba^{-\epsilon}b)$ and $(ba^{-\epsilon}ba^{-\epsilon}b,ba^{\epsilon}ba^{-\epsilon}b)$. In this case, the substitutions given by the following graph show that $(g_1,g_2)$ can be transformed into a pair of almost short elements via elementary transformations.
\begin{center}
\begin{tikzpicture}[thick]
\node (11) {$(b,a^{\epsilon}b)$};
\node (12) [right of=11, node distance={40mm}] {$(a^{\epsilon}b,ba^{-\epsilon}ba^{\epsilon}b)$};
\node (13) [right of=12, node distance={40mm}] {$(ba^{-\epsilon}ba^{\epsilon}b,ba^{-\epsilon}ba^{-\epsilon}b)$};
\node (21) [below of=11, node distance={10mm}] {$(ba^{\epsilon},b)$};
\node (22) [right of=21, node distance={40mm}] {$(ba^{\epsilon}ba^{-\epsilon}b,ba^{\epsilon})$};
\node (23) [right of=22, node distance={40mm}] {$(ba^{-\epsilon}ba^{-\epsilon}b,ba^{\epsilon}ba^{-\epsilon}b)$};

\draw[->] (11) -- node[above] {$R_1$} (12);
\draw[->] (12) -- node[above] {$R_1$} (13);
\draw[->] (13) -- node[right] {$R_1$} (23);
\draw[->] (23) -- node[below] {$R_1$} (22);
\draw[->] (22) -- node[below] {$R_1$} (21);
\draw[->] (21) -- node[left] {$R_1$} (11);
\end{tikzpicture} 
\end{center}
When both of $g_1$, $g_2$ are conjugate to $a^{-\epsilon}b$, either $(g_1,g_2)$ is a pair of almost short elements or $(g_1,g_2)$ is one of $(ba^{\epsilon}ba^{\epsilon}b,ba^{-\epsilon})$, $(a^{-\epsilon}b,ba^{\epsilon}ba^{\epsilon}b)$ which can be transformed into $(ba^{-\epsilon},a^{-\epsilon}b)$ via $R_1$, $R_1^{-1}$ respectively. When one of $g_1$, $g_2$ is conjugate to $a^{\epsilon}$ and the other one is conjugate to $a^{-\epsilon}ba^{-\epsilon}ba^{-\epsilon}$, either $(g_1,g_2)$ is a pair of almost short elements or $(g_1,g_2)$ is one of
\begin{align*}
&(ba^{-\epsilon}ba^{\epsilon},a^{-\epsilon}ba^{\epsilon}ba^{\epsilon}),
(a^{\epsilon}ba^{\epsilon}ba^{-\epsilon},a^{\epsilon}ba^{-\epsilon}b),\\
&(a^{-\epsilon}ba^{\epsilon}ba^{\epsilon},a^{-\epsilon}ba^{-\epsilon}ba^{-\epsilon}),
(a^{-\epsilon}ba^{-\epsilon}ba^{-\epsilon},a^{\epsilon}ba^{\epsilon}ba^{-\epsilon}),\\
&(ba^{\epsilon}ba^{\epsilon}ba^{-\epsilon}b,ba^{\epsilon}ba^{-\epsilon}),
(a^{-\epsilon}ba^{\epsilon}b,ba^{-\epsilon}ba^{\epsilon}ba^{\epsilon}b),\\
&(ba^{-\epsilon}ba^{-\epsilon}ba^{-\epsilon}b,ba^{\epsilon}ba^{\epsilon}ba^{-\epsilon}b),
(ba^{-\epsilon}ba^{\epsilon}ba^{\epsilon}b,ba^{-\epsilon}ba^{-\epsilon}ba^{-\epsilon}b).
\end{align*}
In this case, the following graphs show that $(g_1,g_2)$ can be transformed into a pair of almost short elements via elementary transformations.
\begin{center}
\begin{tikzpicture}[thick]
\node (11) {$(ba^{\epsilon}b,ba^{-\epsilon}ba^{\epsilon})$};
\node (21) [below of=11, node distance={10mm}] {$(a^{\epsilon}ba^{-\epsilon}b,ba^{\epsilon}b)$};
\node (12) [right of=11, node distance={50mm}] {$(ba^{-\epsilon}ba^{\epsilon},a^{-\epsilon}ba^{\epsilon}ba^{\epsilon})$};
\node (13) [right of=12, node distance={50mm}] {$(a^{-\epsilon}ba^{\epsilon}ba^{\epsilon},a^{-\epsilon}ba^{-\epsilon}ba^{-\epsilon})$};
\node (22) [right of=21, node distance={50mm}] {$(a^{\epsilon}ba^{\epsilon}ba^{-\epsilon},a^{\epsilon}ba^{-\epsilon}b)$};
\node (23) [right of=22, node distance={50mm}] {$(a^{-\epsilon}ba^{-\epsilon}ba^{-\epsilon},a^{\epsilon}ba^{\epsilon}ba^{-\epsilon})$};

\draw[->] (11) -- node[above] {$R_1$} (12);
\draw[->] (12) -- node[above] {$R_1$} (13);
\draw[->] (13) -- node[right] {$R_1$} (23);
\draw[->] (23) -- node[below] {$R_1$} (22);
\draw[->] (22) -- node[below] {$R_1$} (21);
\draw[->] (21) -- node[left] {$R_1$} (11);
\end{tikzpicture} 
\end{center}
\begin{center}
\begin{tikzpicture}[thick]
\node (11) {$(ba^{\epsilon}ba^{-\epsilon},a^{\epsilon})$};
\node (21) [below of=11, node distance={10mm}] {$(a^{\epsilon},a^{-\epsilon}ba^{\epsilon}b)$};
\node (12) [right of=11, node distance={50mm}] {$(ba^{\epsilon}ba^{\epsilon}ba^{-\epsilon}b,ba^{\epsilon}ba^{-\epsilon})$};
\node (13) [right of=12, node distance={50mm}] {$(ba^{-\epsilon}ba^{-\epsilon}ba^{-\epsilon}b,ba^{\epsilon}ba^{\epsilon}ba^{-\epsilon}b)$};
\node (22) [right of=21, node distance={50mm}] {$(a^{-\epsilon}ba^{\epsilon}b,ba^{-\epsilon}ba^{\epsilon}ba^{\epsilon}b)$};
\node (23) [right of=22, node distance={50mm}] {$(ba^{-\epsilon}ba^{\epsilon}ba^{\epsilon}b,ba^{-\epsilon}ba^{-\epsilon}ba^{-\epsilon}b)$};

\draw[<-] (11) -- node[above] {$R_1$} (12);
\draw[<-] (12) -- node[above] {$R_1$} (13);
\draw[<-] (13) -- node[right] {$R_1$} (23);
\draw[<-] (23) -- node[below] {$R_1$} (22);
\draw[<-] (22) -- node[below] {$R_1$} (21);
\draw[<-] (21) -- node[left] {$R_1$} (11);
\end{tikzpicture} 
\end{center}
\end{proof}

\begin{lemma}
\label{lemma::tau1tau2=b}
Let $(\tau_1,\tau_2)$ be a pair of almost short elements in Table \ref{table::almost short} such that $\tau_1\tau_2=b$.
Set $(g_1,g_2)=(Q^{-1}\tau_1Q,Q^{-1}\tau_2Q)$ with $Q\in G$ and suppose that $Q^{-1}\tau_1\tau_2Q$ is almost short.
Then $(g_1,g_2)$ is Hurwitz equivalent to a pair of almost short elements.
\end{lemma}
\begin{proof}
The pair $(\tau_1,\tau_2)$ must be one of
\[
(a^{\epsilon},a^{-\epsilon}b), (ba^{-\epsilon},a^{\epsilon}), (a^{-\epsilon}b,ba^{\epsilon}b), (ba^{\epsilon}b,ba^{-\epsilon}),
(ba^{\epsilon}ba^{-\epsilon},a^{\epsilon}ba^{-\epsilon}),
(a^{-\epsilon}ba^{\epsilon},a^{-\epsilon}ba^{\epsilon}b)
\]
with $\epsilon=\pm 1$ and $Q\in\{1,a,a^2,b,ba,ba^2\}$. The lemma follows from the following graphs with $\epsilon=\pm 1$.
\begin{center}
\begin{tikzpicture}[thick]
\node (21) {$(a^{-\epsilon}b,a^{\epsilon})$};
\node (11) [above of=21, node distance={10mm}] {$(a^{\epsilon},a^{\epsilon}ba^{\epsilon})$};
\node (12) [right of=11, node distance={30mm}] {$(a^{\epsilon}ba^{\epsilon},a^{-\epsilon}ba^{\epsilon}ba^{\epsilon})$};
\node (22) [below of=12, node distance={10mm}] {$(a^{-\epsilon}ba^{\epsilon}ba^{\epsilon},a^{-\epsilon}b)$};

\draw[->] (11) -- node[above] {$R_1$} (12);
\draw[->] (12) -- node[right] {$R_1$} (22);
\draw[->] (22) -- node[below] {$R_1$} (21);
\draw[->] (21) -- node[left] {$R_1$} (11);

\node (23) [right of=22, node distance={40mm}] {$(a^{\epsilon}ba^{\epsilon}ba^{\epsilon},a^{-\epsilon}ba^{\epsilon})$};
\node (13) [above of=23, node distance={10mm}] {$(a^{-\epsilon}ba^{\epsilon},a^{-\epsilon}ba^{-\epsilon}ba^{-\epsilon})$};
\node (14) [right of=13, node distance={45mm}] {$(a^{-\epsilon}ba^{-\epsilon}ba^{-\epsilon},a^{\epsilon}ba^{\epsilon}ba^{-\epsilon}ba^{-\epsilon})$};
\node (24) [below of=14, node distance={10mm}] {$(a^{\epsilon}ba^{\epsilon}ba^{-\epsilon}ba^{-\epsilon},a^{\epsilon}ba^{\epsilon}ba^{\epsilon})$};

\draw[->] (13) -- node[above] {$R_1$} (14);
\draw[->] (14) -- node[right] {$R_1$} (24);
\draw[->] (24) -- node[below] {$R_1$} (23);
\draw[->] (23) -- node[left] {$R_1$} (13);

\node (31) [below of=21, node distance={10mm}] {$(b,ba^{\epsilon}ba^{-\epsilon})$};
\node (41) [below of=31, node distance={10mm}] {$(a^{\epsilon}ba^{-\epsilon}b,b)$};
\node (32) [right of=31, node distance={40mm}] {$(ba^{\epsilon}ba^{-\epsilon},a^{\epsilon}ba^{-\epsilon}ba^{\epsilon}ba^{-\epsilon})$};
\node (42) [below of=32, node distance={10mm}] {$(a^{\epsilon}ba^{-\epsilon}ba^{\epsilon}ba^{-\epsilon},a^{\epsilon}ba^{-\epsilon}b)$};

\draw[->] (31) -- node[above] {$R_1$} (32);
\draw[->] (32) -- node[right] {$R_1$} (42);
\draw[->] (42) -- node[below] {$R_1$} (41);
\draw[->] (41) -- node[left] {$R_1$} (31);

\node (43) [right of=42, node distance=40mm] {$(ba^{-\epsilon}ba^{\epsilon},a^{-\epsilon}ba^{\epsilon})$};
\node (33) [above of=43, node distance={10mm}] {$(a^{-\epsilon}ba^{\epsilon},a^{-\epsilon}ba^{\epsilon}b)$};
\node (34) [right of=33, node distance={40mm}] {$(a^{-\epsilon}ba^{\epsilon}b,a^{-\epsilon}ba^{\epsilon}b)$};
\node (44) [below of=34, node distance={10mm}] {$(ba^{-\epsilon}ba^{\epsilon}b,ba^{-\epsilon}ba^{\epsilon})$};

\draw[->] (33) -- node[above] {$R_1$} (34);
\draw[->] (34) -- node[right] {$R_1$} (44);
\draw[->] (44) -- node[below] {$R_1$} (43);
\draw[->] (43) -- node[left] {$R_1$} (33);
\end{tikzpicture} 
\end{center}
\end{proof}

\begin{lemma}
\label{lemma::tau1tau2=bab}
Let $(\tau_1,\tau_2)$ be a pair of almost short elements in Table \ref{table::almost short} such that $\tau_1\tau_2=ba^{\epsilon}b$ with $\epsilon=\pm 1$.
Set $(g_1,g_2)=(Q^{-1}\tau_1Q,Q^{-1}\tau_2Q)$ with $Q\in G$ and suppose that $Q^{-1}\tau_1\tau_2Q$ is almost short.
Then $(g_1,g_2)$ is Hurwitz equivalent to a pair of almost short elements.
\end{lemma}
\begin{proof}
The pair $(\tau_1,\tau_2)$ must be one of
\[
(a^{\epsilon},a^{-\epsilon}ba^{\epsilon}b),
(ba^{\epsilon}ba^{-\epsilon},a^{\epsilon}),
(b,a^{\epsilon}b),
(ba^{\epsilon},b),
(ba^{-\epsilon},a^{-\epsilon}b),
(ba^{-\epsilon}b,ba^{-\epsilon}b)
\]
with $\epsilon=\pm 1$ and $Q\in\{1,b,ba^{\epsilon},ba^{-\epsilon},ba^{\epsilon}b,ba^{-\epsilon}b\}$. The lemma follows from the following graphs.
\begin{center}
\begin{tikzpicture}[thick]
\node (11) {$(ba^{\epsilon}b,ba^{-\epsilon}ba^{\epsilon})$};
\node (12) [right of=11, node distance={40mm}] {$(ba^{-\epsilon}ba^{\epsilon},a^{-\epsilon}ba^{\epsilon}ba^{\epsilon})$};
\node (13) [right of=12, node distance={50mm}] {$(a^{-\epsilon}ba^{\epsilon}ba^{\epsilon},a^{-\epsilon}ba^{-\epsilon}ba^{-\epsilon})$};
\node (21) [below of=11, node distance={10mm}] {$(a^{\epsilon}ba^{-\epsilon}b,ba^{\epsilon}b)$};
\node (22) [right of=21, node distance={40mm}] {$(a^{\epsilon}ba^{\epsilon}ba^{-\epsilon},a^{\epsilon}ba^{-\epsilon}b)$};
\node (23) [right of=22, node distance={50mm}] {$(a^{-\epsilon}ba^{-\epsilon}ba^{-\epsilon},a^{\epsilon}ba^{\epsilon}ba^{-\epsilon})$};

\draw[->] (11) -- node[above] {$R_1$} (12);
\draw[->] (12) -- node[above] {$R_1$} (13);
\draw[->] (13) -- node[right] {$R_1$} (23);
\draw[->] (23) -- node[below] {$R_1$} (22);
\draw[->] (22) -- node[below] {$R_1$} (21);
\draw[->] (21) -- node[left] {$R_1$} (11);
\end{tikzpicture} 
\end{center}
\begin{center}
\begin{tikzpicture}[thick]
\node (11)                                     {$(a^{\epsilon},a^{-\epsilon}ba^{\epsilon}b)$};
\node (12) [right of=11, node distance={40mm}] {$(a^{-\epsilon}ba^{\epsilon}b,ba^{-\epsilon}ba^{\epsilon}ba^{\epsilon}b)$};
\node (13) [right of=12, node distance={50mm}] {$(ba^{-\epsilon}ba^{\epsilon}ba^{\epsilon}b,ba^{-\epsilon}ba^{-\epsilon}ba^{-\epsilon}b)$};
\node (21) [below of=11, node distance={10mm}] {$(ba^{\epsilon}ba^{-\epsilon},a^{\epsilon})$};
\node (22) [right of=21, node distance={40mm}] {$(ba^{\epsilon}ba^{\epsilon}ba^{-\epsilon}b,ba^{\epsilon}ba^{-\epsilon})$};
\node (23) [right of=22, node distance={50mm}] {$(ba^{-\epsilon}ba^{-\epsilon}ba^{-\epsilon}b,ba^{\epsilon}ba^{\epsilon}ba^{-\epsilon}b)$};

\draw[->] (11) -- node[above] {$R_1$} (12);
\draw[->] (12) -- node[above] {$R_1$} (13);
\draw[->] (13) -- node[right] {$R_1$} (23);
\draw[->] (23) -- node[below] {$R_1$} (22);
\draw[->] (22) -- node[below] {$R_1$} (21);
\draw[->] (21) -- node[left] {$R_1$} (11);
\end{tikzpicture} 
\end{center}
\begin{center}
\begin{tikzpicture}[thick]
\node (11)                                     {$(b,a^{\epsilon}b)$};
\node (12) [right of=11, node distance={25mm}] {$(a^{\epsilon}b,ba^{-\epsilon}ba^{\epsilon}b)$};
\node (13) [right of=12, node distance={35mm}] {$(ba^{-\epsilon}ba^{\epsilon}b,ba^{-\epsilon}ba^{-\epsilon}b)$};
\node (21) [below of=11, node distance={10mm}] {$(ba^{\epsilon},b)$};
\node (22) [right of=21, node distance={25mm}] {$(ba^{\epsilon}ba^{-\epsilon}b,ba^{\epsilon})$};
\node (23) [right of=22, node distance={35mm}] {$(ba^{-\epsilon}ba^{-\epsilon}b,ba^{\epsilon}ba^{-\epsilon}b)$};

\draw[->] (11) -- node[above] {$R_1$} (12);
\draw[->] (12) -- node[above] {$R_1$} (13);
\draw[->] (13) -- node[right] {$R_1$} (23);
\draw[->] (23) -- node[below] {$R_1$} (22);
\draw[->] (22) -- node[below] {$R_1$} (21);
\draw[->] (21) -- node[left] {$R_1$} (11);

\node (a) [right of=13, node distance={35mm}] {$(ba^{-\epsilon},a^{-\epsilon}b)$};
\node (b) [right of=a, node distance={30mm}] {$(a^{-\epsilon}b,ba^{\epsilon}ba^{\epsilon}b)$};
\node (c) [right of=23, node distance={50mm}] {$(ba^{\epsilon}ba^{\epsilon}b,ba^{-\epsilon})$};

\draw[->] (a) -- node[above] {$R_1$} (b);
\draw[->] (b) -- node[right] {$R_1$} (c);
\draw[->] (c) -- node[left] {$R_1$} (a);
\end{tikzpicture} 
\end{center}
\end{proof}

\begin{lemma}
\label{lemma::tau1tau2=aaba}
Let $(\tau_1,\tau_2)$ be a pair of almost short elements in Table \ref{table::almost short} such that $\tau_1\tau_2=a^{-\epsilon}ba^{\epsilon}$ with $\epsilon=\pm 1$.
Set $(g_1,g_2)=(Q^{-1}\tau_1Q,Q^{-1}\tau_2Q)$ with $Q\in G$ and suppose that $Q^{-1}\tau_1\tau_2Q$ is almost short.
Then $(g_1,g_2)$ is Hurwitz equivalent to a pair of almost short elements.
\end{lemma}
\begin{proof}
The pair $(\tau_1,\tau_2)$ must be one of
\begin{align*}
&(a^{\epsilon},a^{\epsilon}ba^{\epsilon}),
(a^{-\epsilon},ba^{\epsilon}),
(a^{-\epsilon}ba^{-\epsilon},a^{-\epsilon}),
(a^{-\epsilon}b,a^{\epsilon})\\
&(b,ba^{-\epsilon}ba^{\epsilon}),
(a^{-\epsilon}ba^{\epsilon}b,b),
(a^{-\epsilon}ba^{-\epsilon}ba^{-\epsilon},a^{\epsilon}ba^{-\epsilon}),
(a^{\epsilon}ba^{-\epsilon},a^{\epsilon}ba^{\epsilon}ba^{\epsilon})
\end{align*}
with $\epsilon=\pm 1$ and $Q\in\{1,a^{\epsilon},a^{-\epsilon},a^{-\epsilon}b,a^{-\epsilon}ba^{\epsilon},a^{-\epsilon}ba^{-\epsilon}\}$. The lemma follows from the following graphs.
\begin{center}
\begin{tikzpicture}[thick]
\node (11)                                     {$(a^{\epsilon},a^{\epsilon}ba^{\epsilon})$};
\node (12) [right of=11, node distance={30mm}] {$(a^{\epsilon}ba^{\epsilon},a^{-\epsilon}ba^{\epsilon}ba^{\epsilon})$};
\node (21) [below of=11, node distance={10mm}] {$(a^{-\epsilon}b,a^{\epsilon})$};
\node (22) [right of=21, node distance={30mm}] {$(a^{-\epsilon}ba^{\epsilon}ba^{\epsilon},a^{-\epsilon}b)$};

\draw[->] (11) -- node[above] {$R_1$} (12);
\draw[->] (12) -- node[right] {$R_1$} (22);
\draw[->] (22) -- node[below] {$R_1$} (21);
\draw[->] (21) -- node[left]  {$R_1$} (11);

\node (13) [right of=12, node distance={30mm}] {$(a^{\epsilon},ba^{-\epsilon})$};
\node (14) [right of=13, node distance={30mm}] {$(ba^{-\epsilon},a^{\epsilon}ba^{\epsilon}ba^{-\epsilon})$};
\node (23) [below of=13, node distance={10mm}] {$(a^{\epsilon}ba^{\epsilon},a^{\epsilon})$};
\node (24) [right of=23, node distance={30mm}] {$(a^{\epsilon}ba^{\epsilon}ba^{-\epsilon},a^{\epsilon}ba^{\epsilon})$};

\draw[->] (13) -- node[above] {$R_1$} (14);
\draw[->] (14) -- node[right] {$R_1$} (24);
\draw[->] (24) -- node[below] {$R_1$} (23);
\draw[->] (23) -- node[left]  {$R_1$} (13);
\end{tikzpicture} 
\end{center}
\begin{center}
\begin{tikzpicture}[thick]
\node (11)                                     {$(a^{\epsilon}ba^{-\epsilon},a^{\epsilon}ba^{-\epsilon}b)$};
\node (12) [right of=11, node distance={40mm}] {$(a^{\epsilon}ba^{-\epsilon}b,ba^{\epsilon}ba^{-\epsilon}b)$};
\node (21) [below of=11, node distance={10mm}] {$(ba^{\epsilon}ba^{-\epsilon},a^{\epsilon}ba^{-\epsilon})$};
\node (22) [right of=21, node distance={40mm}] {$(ba^{\epsilon}ba^{-\epsilon}b,ba^{\epsilon}ba^{-\epsilon})$};

\draw[->] (11) -- node[above] {$R_1$} (12);
\draw[->] (12) -- node[right] {$R_1$} (22);
\draw[->] (22) -- node[below] {$R_1$} (21);
\draw[->] (21) -- node[left]  {$R_1$} (11);

\node (13) [right of=12, node distance={30mm}] {$(b,ba^{-\epsilon}ba^{\epsilon})$};
\node (14) [right of=13, node distance={40mm}] {$(ba^{-\epsilon}ba^{\epsilon},a^{-\epsilon}ba^{\epsilon}ba^{-\epsilon}ba^{\epsilon})$};
\node (23) [below of=13, node distance={10mm}] {$(a^{-\epsilon}ba^{\epsilon}b,b)$};
\node (24) [right of=23, node distance={40mm}] {$(a^{-\epsilon}ba^{\epsilon}ba^{-\epsilon}ba^{\epsilon},a^{-\epsilon}ba^{\epsilon}b)$};

\draw[->] (13) -- node[above] {$R_1$} (14);
\draw[->] (14) -- node[right] {$R_1$} (24);
\draw[->] (24) -- node[below] {$R_1$} (23);
\draw[->] (23) -- node[left]  {$R_1$} (13);
\end{tikzpicture} 
\end{center}
\begin{center}
\begin{tikzpicture}[thick]
\node (11)                                     {$(a^{-\epsilon}ba^{\epsilon},a^{-\epsilon}ba^{-\epsilon}ba^{-\epsilon})$};
\node (12) [right of=11, node distance={50mm}] {$(a^{-\epsilon}ba^{-\epsilon}ba^{-\epsilon},a^{\epsilon}ba^{\epsilon}ba^{-\epsilon}ba^{-\epsilon})$};
\node (21) [below of=11, node distance={10mm}] {$(a^{\epsilon}ba^{\epsilon}ba^{\epsilon},a^{-\epsilon}ba^{\epsilon})$};
\node (22) [right of=21, node distance={50mm}] {$(a^{\epsilon}ba^{\epsilon}ba^{-\epsilon}ba^{-\epsilon},a^{\epsilon}ba^{\epsilon}ba^{\epsilon})$};

\draw[->] (11) -- node[above] {$R_1$} (12);
\draw[->] (12) -- node[right] {$R_1$} (22);
\draw[->] (22) -- node[below] {$R_1$} (21);
\draw[->] (21) -- node[left]  {$R_1$} (11);
\end{tikzpicture} 
\end{center}
\end{proof}

\begin{lemma}
\label{lemma::tau1tau2=aba}
Let $(\tau_1,\tau_2)$ be a pair of almost short elements in Table \ref{table::almost short} such that $\tau_1\tau_2$ is one of $a^{\epsilon}b$, $ba^{\epsilon}$ and $a^{-\epsilon}ba^{-\epsilon}$ with $\epsilon=\pm 1$.
Set $(g_1,g_2)=(Q^{-1}\tau_1Q,Q^{-1}\tau_2Q)$ with $Q\in G$ and suppose that $Q^{-1}\tau_1\tau_2Q$ is almost short.
Then $(g_1,g_2)$ is Hurwitz equivalent to a pair of almost short elements.
\end{lemma}
\begin{proof}
Since $Q^{-1}\tau_1\tau_2Q$ is almost short, the element $Q$ is either $(\tau_1\tau_2)^k a^{\zeta}$ or $(\tau_1\tau_2)^{-l}a^{\zeta}$ with $k,l\ge 0$ and $\zeta=0,1,2$.
When $Q=a^{\zeta}$, the only exceptional cases that at least one of $Q^{-1}\tau_1Q$, $Q^{-1}\tau_2Q$ is not almost short is that $(\tau_1,\tau_2)$ is equal to one of
\[
(b,ba^{\epsilon}b),
(ba^{\epsilon}b,b),
(a^{-\epsilon}b,ba^{-\epsilon}b),
(ba^{-\epsilon}b,ba^{-\epsilon})
\]
with $\epsilon=\pm 1$, where both $(a^{-\epsilon}\tau_1a^{\epsilon},a^{-\epsilon}\tau_2a^{\epsilon})$ and $(a^{\epsilon}\tau_1a^{-\epsilon},a^{\epsilon}\tau_2a^{-\epsilon})$ can be transformed into pairs of almost short elements by applying either $R_1$ or $R_1^{-1}$. In general, Lemma \ref{lemma::h=g1g2} shows that $(g_1,g_2)$ can be transformed into a pair of almost short elements.
\end{proof}

\begin{lemma}
\label{lemma::tau1tau2=ababa}
Let $(\tau_1,\tau_2)$ be a pair of almost short elements in Table \ref{table::almost short} such that $\tau_1\tau_2$ is almost short and conjugate to $ababa$.
Set $(g_1,g_2)=(Q^{-1}\tau_1Q,Q^{-1}\tau_2Q)$ with $Q\in G$ and suppose that $Q^{-1}\tau_1\tau_2Q$ is almost short.
Then $(g_1,g_2)$ is Hurwitz equivalent to a pair of almost short elements.
\end{lemma}
\begin{proof}
Since $\tau_1\tau_2$ is almost short and conjugate to $ababa$, it must be one of $a^{-\epsilon}ba^{\epsilon}b$, $ba^{-\epsilon}ba^{\epsilon}$ and $a^{\epsilon}ba^{\epsilon}ba^{\epsilon}$ with $\epsilon=\pm 1$.
When $\tau_1\tau_2=a^{-\epsilon}ba^{\epsilon}b$, since $Q^{-1}\tau_1\tau_2Q$ is almost short, the element $Q$ is either $(a^{-\epsilon}ba^{\epsilon}b)^ka^{\zeta}$ or $(a^{-\epsilon}ba^{\epsilon}b)^k(a^{\epsilon}b)a^{\zeta}$ with $k\in\mathbb{Z}$ and $\zeta=0,1,2$. Lemma \ref{lemma::h=g1g2} induces that it suffices to suppose that
\[
Q\in\{1,a^{\epsilon},a^{-\epsilon},a^{-\epsilon}b,a^{-\epsilon}ba^{\epsilon},a^{-\epsilon}ba^{-\epsilon}\}.
\]
Besides, $(\tau_1,\tau_2)$ is one of
\[
(a^{-\epsilon}ba^{-\epsilon},a^{-\epsilon}b),
(a^{-\epsilon}ba^{-\epsilon}ba^{-\epsilon},a^{\epsilon}ba^{-\epsilon}b).
\]
Each possible $(g_1,g_2)$ is either a pair of almost short elements or transformed into a pair of almost short elements by $R_1^{\pm 1}$. When $\tau_1\tau_2=ba^{-\epsilon}ba^{\epsilon}$ or $a^{\epsilon}ba^{\epsilon}ba^{\epsilon}$ we have similar arguments.
\end{proof}

%%%

We introduce the following operations and their restorations on an $n$-tuple $(g_1,\ldots,g_n)$ of elements in $G$ that are conjugate to some almost short element.

\begin{itemize}[-]
    \item \underline{Operation $1$:} For $i\in\{1,\ldots,n-1\}$, suppose that $g_i=Q^{-1}\tau_i Q$ and $g_{i+1}=Q^{-1}\tau_{i+1}Q$ with $Q\in G$ and $(\tau_i,\tau_{i+1})$ listed in Table \ref{table::almost short}.
    Then, the operation is a contraction as in Subsection \ref{subsection::contractions-restorations} that replaces $(g_i,g_{i+1})$ with $g_i g_{i+1}$.
    %\item \underline{Operation $1$:} For $i\in\{1,\ldots,n-1\}$, suppose that the reduced forms of $g_i$ and $g_{i+1}$ are expressed by $Q_i^{-1}\tau_i Q_i$ and $Q_{i+1}^{-1}\tau_{i+1}Q_{i+1}$ with $\tau_i,\tau_{i+1}\in\mathcal{S}_2$, $Q_i,Q_{i+1}\in G$ such that $Q_i=Q_{i+1}$ and $(\tau_i,\tau_{i+1})$ is listed in Table \ref{table::almost short}.
    %Then, the operation is a contraction as in Subsection \ref{subsection::contractions-restorations} that replaces $(g_i,g_{i+1})$ with $g_i g_{i+1}$.
    
    \item \underline{Operation $1'$:} For $i\in\{1,\ldots,n-2\}$, suppose that $g_i=Q^{-1}\tau_i Q$, $g_{i+1}=Q^{-1}\tau_{i+1}Q$ and $g_{i+2}=Q^{-1}\tau_{i+2}Q$ with $Q\in G$ and $(\tau_i,\tau_{i+1},\tau_{i+2})$ equal to either
    \[
    (a^2bab,bab,baba^2) \text{ or } (aba^2b,ba^2b,ba^2ba).
    \]
    Then, the operation is a contraction as in Subsection \ref{subsection::contractions-restorations} that replaces $(g_i,g_{i+1},g_{i+2})$ with $g_i g_{i+1} g_{i+2}$.

    \item \underline{Operation $2$:} For $i\in\{1,\ldots,n\}$, suppose that $g_i=1$. The operation moves the identical component to the rightmost position via elementary transformations, removes it and reduces $(g_1,\ldots,g_n)$ to an $(n-1)$-tuple.
\end{itemize}

Operation $1$ and $1'$ are contractions, whose restorations are introduced in Subsection \ref{subsection::elementary transformations}. The restoration of Operation $2$ will simply add an identical element on the right side of the tuple.

\begin{proposition}
\label{proposition::reduce_and_reloading_hard}
Let $(g_1,\ldots,g_n)$ be an $n$-tuple of elements in $G$ which are conjugate to some almost short element such that $g_1\cdots g_n=1$. Suppose that we apply the following operations successively on $(g_1,\ldots,g_n)$:
\begin{enumerate}[i),topsep=0pt,itemsep=-1ex,partopsep=1ex,parsep=1ex]
    \item elementary transformations;
    \item Operation $1$;
    \item Operation $1'$;
    \item Operation $2$;
\end{enumerate}
then apply the restorations of Operation $1$, $1'$ and $2$ in the reverse order.
If each component in the resulting tuple before restorations are almost short, then the initial tuple is Hurwitz equivalent to the following tuples:
\begin{enumerate}[a),topsep=0pt,itemsep=-1ex,partopsep=1ex,parsep=1ex]
    \item the resulting tuple after restorations;
    \item the concatenation of some tuples of the form
    $(Q^{-1}\tau_1 Q,Q^{-1}\tau_2 Q,\ldots,Q^{-1}\tau_m Q)$ with $m\ge 1$, $Q\in G$ and $\tau_1,\ldots,\tau_m\in\mathcal{S}_2$ such that $\tau_1\cdots\tau_m=1$.
\end{enumerate}
\end{proposition}

We emphasise that Proposition \ref{proposition::reduce_and_reloading_hard} does not require an inverse-free tuple $(g_1,\ldots,g_n)$ in $G$ as in Proposition \ref{proposition::reduce_and_reloading}.
Besides, an elementary transformation is allowed to transform a pair into such that has a bigger sum of $\mathcal{S}_2$-complexities.
That is why we cannot transform it into a tuple of almost short elements but get a concatenation of several tuples of almost short elements each with a diagonal conjugacy.

\begin{proof}
Lemma \ref{lemma::technique} shows that the initial tuple is Hurwitz equivalent to the resulting tuple after all operations and restorations.
We suppose that each component is almost short in the tuple before restorations.

We revisit the introduced operations. Operation $1$ may combine $Q^{-1}\tau_iQ$ and $Q^{-1}\tau_jQ$ into $Q^{-1}\tau_i\tau_j Q$ with $Q\in G$ and $\tau_i,\tau_j\in\mathcal{S}_2$. By elementary transformations, the product becomes a conjugate of the form $P^{-1}Q^{-1}\tau_i\tau_j QP$ with some $P\in G$. To restore the operation, we further rewrite it as $(P^{-1}Q^{-1}\tau_i QP,P^{-1}Q^{-1}\tau_j QP)$.
Operation $1'$ is similar.

If Operation $2$ has never been used, the proposition follows from lemmata \ref{lemma::tau1tau2=a}, \ref{lemma::tau1tau2=b}, \ref{lemma::tau1tau2=bab}, \ref{lemma::tau1tau2=aaba}, \ref{lemma::tau1tau2=aba}, \ref{lemma::tau1tau2=ababa} and \ref{lemma::tau1_tau2_tau3_triple}.
In general, suppose that $P^{-1}Q^{-1}\tau_i\tau_j QP=1$. Then $P=1$ and the restoration replaces the identical element with $(Q^{-1}\tau_iQ, Q^{-1}\tau_jQ)$. We consider the remaining restorations on $(\tau_i,\tau_j)$ instead.
\end{proof}

\begin{definition}
The $\mathcal{S}_2$-\emph{complexity} of an element $g$ conjugate to some element in $\mathcal{S}_2$ is defined as $f_2(g)$ such that
\begin{equation*}
    f_2(g) =
    \begin{cases}
        l(Q) &
        \begin{aligned}
        &\text{if $g=Q^{-1}wQ$ is almost long} \\
        &\text{with $Q\in G$ and $w\in\{ba^{\epsilon} b,a^\epsilon ba^{-\epsilon},a^\epsilon ba^\epsilon,a^\epsilon ba^\epsilon b a^\epsilon|\epsilon=1,2\}$;}
        \end{aligned}
        \\
        1/2 & \text{if $g\in\{ababa,a^2ba^2ba^2\}$;}\\
        0 & \text{otherwise.}
    \end{cases}
\end{equation*}
\end{definition}

\begin{definition}
Let $(g_1,\ldots,g_n)$ be an $n$-tuple in $G$ such that each of $g_i$, $i=1,\ldots,n$, is conjugate to some element in $\mathcal{S}_2$.
A sequence of elementary transformations $(R_{i_1}^{\epsilon_1},\ldots,R_{i_m}^{\epsilon_m})$, $\epsilon_1,\ldots,\epsilon_m\in\{1,-1\}$, is said to make the sum of $\mathcal{S}_2$-complexities of $(g_1,\ldots,g_n)$ \emph{smaller} if
$R_{i_m}^{\epsilon_m}\circ\cdots\circ R_{i_1}^{\epsilon_1}$ transforms $(g_1,\ldots,g_n)$ into a tuple with a smaller sum of $\mathcal{S}_2$-complexities.
\end{definition}

\begin{lemma}
\label{lemma::m_i-1_and_m_i_hard}
Let $(g_1,\ldots,g_n)$ be an $n$-tuple in $G$ such that each of $g_i$, $i=1,\ldots,n$, is conjugate to some element in $\mathcal{S}_2$ and $g_1\cdots g_n=1$. Let $m_i$ be the same as in Lemma \ref{lemma::m_i-1_and_m_i} and set $m_0=m_n=0$ for convenience. Suppose that
\begin{enumerate}[(1),topsep=0pt,itemsep=-1ex,partopsep=1ex,parsep=1ex]
    \item there is no pair of adjacent components $g_i$, $g_{i+1}$ of the reduced forms $Q^{-1}\tau_iQ$, $Q^{-1}\tau_{i+1}Q$ with $Q\in G$ and $(\tau_i, \tau_{i+1})$ in Table \ref{table::almost short},
    \item there is no sequence of elementary transformations that makes $\sum_i f_2(g_i)$ smaller.
\end{enumerate}
Then $m_i$, $i=0,\ldots,n$ have the following properties.
\begin{enumerate}[(a),topsep=0pt,itemsep=-1ex,partopsep=1ex,parsep=1ex]
    \item For $i=1,\ldots,n-1$, $m_i\le \frac{l(g_i)+1}{2}$ and $m_i\le \frac{l(g_{i+1})+1}{2}$.
    \item For $i=1,\ldots,n$, $m_{i-1}+m_i\ge l(g_i)$ only if the reduced form of $g_i$ is either $Q_i^{-1}a^{\epsilon_i}Q_i$ or $Q_i^{-1}a^{\epsilon_i}ba^{\epsilon_i}ba^{\epsilon_i}Q_i$ with $\epsilon_i=1,2$, $Q_i\in G$ and $l(Q_i)\ge 0$.
    \item If $m_{i-1}+m_i\le l(g_i)$ for each of $i=1,\ldots,n$, then $n=0$.
\end{enumerate}
\end{lemma}

\begin{proof}
(a) When both $g_i$ and $g_{i+1}$ are almost short, since $(g_i,g_{i+1})$ does not figure in Table \ref{table::almost short}, we check all possibilities and get that $m_i\le \frac{l(g_i)+1}{2}, \frac{l(g_{i+1})+1}{2}$.

When $g_i\in\mathcal{S}_2$ but $g_{i+1}\not\in\mathcal{S}_2$, we have concluded that $g_{i+1}=Q^{-1}wQ$ with
\begin{equation*}
    w\in\{ba^\epsilon b,a^\epsilon ba^{-\epsilon},a^\epsilon ba^\epsilon,a^\epsilon ba^\epsilon b a^\epsilon|\epsilon=1,2\}
\end{equation*}
and $Q\in G$ such that $l(Q)\ge 1$. In particular, $l(g_{i+1})\ge 5\ge l(g_i)$.
Assume that $m_i>\frac{l(g_i)+1}{2}$.
Suppose that $l(g_i)=2$. Then $m_i=2$ and the symmetry of $g_{i+1}$ implies the contradiction $l(g_i g_{i+1} g_i^{-1})\le l(g_{i+1})-2$.
Suppose that $l(g_i)=3$. Then $m_i=3$. The symmetry of $g_{i+1}$ and the fact that $l(w)\ge 3$ imply the contradiction $l(g_i g_{i+1} g_i^{-1})\le l(g_{i+1})-2$.
Suppose that $l(g_i)=4$. Then $g_i=a^{\epsilon_i}ba^{-\epsilon_i}b$ or $ba^{\epsilon_i}ba^{-\epsilon_i}$ with $\epsilon_i\in\{1,2\}$.
Therefore $m_i=3$ or $4$.
If $l(Q)=1$ and $g_i=a^{\epsilon_i}ba^{-\epsilon}b$, then $m_i=4$ and
$l(g_i g_{i+1} g_i^{-1})\le l(g_{i+1})-2$, contradicting the hypothesis (2).
If $l(Q)=1$ and $g_i=ba^{\epsilon_i}ba^{-\epsilon_i}$, then $g_{i+1}=a^{\epsilon_i}ba^{\epsilon_{i+1}}ba^{-\epsilon_i}$ with $\epsilon_{i+1}\in\{1,2\}$ and $l(g_i g_{i+1} g_i^{-1})\le l(g_{i+1})-2$, contradicting the hypothesis (2).
If $l(Q)\ge 2$, then we again get $l(g_i g_{i+1} g_i^{-1})< l(g_{i+1})$, contradicting the hypothesis (2).
Suppose that $l(g_i)=5$ and then $m_i=4$ or $5$. As $g_i=a^{\epsilon_i}ba^{\epsilon_i}ba^{\epsilon_i}$, $m_i$ must be $5$.
Therefore, if $l(Q)\ge 2$ then we get the contradiction $l(g_i g_{i+1} g_i^{-1})< l(g_{i+1})$. If $l(Q)=1$ then $g_{i+1}$ must be $a^{-\epsilon}ba^{-\epsilon}ba^{\epsilon}$ but the following substitution makes the sum of $\mathcal{S}_2$-complexities smaller and induces a contradiction.
\[
(g_i,g_{i+1})=(a^{\epsilon_i}ba^{\epsilon_i}ba^{\epsilon_i},a^{-\epsilon}ba^{-\epsilon}ba^{\epsilon})
    \longrightarrow (a^{-\epsilon}ba^{-\epsilon}ba^{\epsilon},a^{-\epsilon}ba^{\epsilon}b)
    \longrightarrow (a^{-\epsilon}ba^{\epsilon}b, ba^{-\epsilon}b).
\]

We have a similar argument when $g_{i+1}$ is almost short but $g_i$ not.

When both $g_i$ and $g_{i+1}$ are almost long, suppose that their reduced forms are $Q_i^{-1}w_iQ_i$ and $Q_{i+1}^{-1}w_{i+1}Q_{i+1}$ and assume that without loss of generality $l(Q_i)\le l(Q_{i+1})$.
Assume that $m_i>\min\{\frac{l(g_i)+1}{2},\frac{l(g_{i+1})+1}{2}\}$.
Therefore $Q_{i+1}$ must end with $Q_i$.
Write $Q_{i+1}=\tilde{Q}Q_i$ and
\begin{equation*}
    (g_i,g_{i+1})=(Q_i^{-1}w_iQ_i,Q_i^{-1}\tilde{Q}^{-1}w_{i+1}\tilde{Q}\tilde{Q}_i).
\end{equation*}
Suppose that $l(\tilde{Q})=0$. The assumption on $m_i$ contradicts Table \ref{table::almost short}.
Suppose that $l(\tilde{Q})\ge 1$. Therefore $l(w_i)\le 5\le l(\tilde{Q}^{-1}w_{i+1}\tilde{Q})$ and $m_i>\frac{l(g_i)+1}{2}$.
If $l(w_i)=3$ then $m_i>l(Q_i)+3$ and $l(g_i g_{i+1} g_i^{-1})\le l(g_{i+1})-2$, contradicting the hypothesis (2).
If $l(w_i)=5$ and $l(\tilde{Q}^{-1}w_{i+1}\tilde{Q})=5$, then
\begin{equation*}
    (g_i,g_{i+1})=(Q_i^{-1}w_iQ_i,Q_i^{-1}\tilde{Q}^{-1}w_{i+1}\tilde{Q}Q_i)=
    (Q_i^{-1} a^{\epsilon_i}ba^{\epsilon_i}ba^{\epsilon_i} Q_i,
     Q_i^{-1} a^{-\epsilon_i}ba^{-\epsilon_i}ba^{\epsilon_i} Q_i)
\end{equation*}
whose sum of $\mathcal{S}_2$-complexities can be smaller using elementary transformations.
If $l(w_i)=5$ and $l(\tilde{Q}^{-1}w_{i+1}\tilde{Q})\ge 7$, then
$l(g_i g_{i+1} g_i^{-1})\le l(g_{i+1})-2$, contradicting the hypothesis (2).

(b) Suppose that $m_{i-1}+m_i\ge l(g_i)$ for some $i=1,\ldots,n-1$.

Suppose that $g_i\in\mathcal{S}_2$.
If $g_i=b$ then either $m_{i-1}=0$, $m_i=1$ or $m_{i-1}=1$, $m_i=0$. Therefore either $g_{i-1}$ ends with $b$ or $g_{i+1}$ starts with $b$. Table \ref{table::almost short} shows that either $g_{i-1}$ or $g_{i+1}$ is almost long, starts and ends with $b$. Hence it implies the contradiction either $l(g_ig_{i+1}g_i^{-1})<l(g_{i+1})$ or $l(g_i^{-1}g_{i-1}g_i)<l(g_{i-1})$.
If $l(g_i)=2$, then one of $g_{i-1}$, $g_{i+1}$ must be $b$, which is impossible based on Table \ref{table::almost short}.
If $l(g_i)=4$, then either $g_{i-1}=a^{\epsilon_{i-1}}b$ or $g_{i+1}=ba^{\epsilon_{i+1}}$ with $\epsilon_{i-1},\epsilon_{i+1}\in\{1,2\}$, which is impossible based on Table \ref{table::almost short}.
If $l(g_i)=3$ and $g_i=a^{\epsilon_i}ba^{\epsilon_i}$ with $\epsilon_i\in\{1,2\}$, then either $g_{i-1}=ba^{-\epsilon_i}$ or $g_{i+1}=a^{-\epsilon_i}b$.
If $l(g_i)=3$ and $g_i=a^{\epsilon_i}ba^{-\epsilon_i}$ with $\epsilon_i\in\{1,2\}$, then either $g_{i-1}=ba^{-\epsilon_i}$ or $g_{i+1}=a^{\epsilon_i}b$.
Both are impossible again based on Table \ref{table::almost short}.
There are only two possibilities left: either $g_i$ is conjugate to a power of $a$ or $g_i=a^{\epsilon_i}ba^{\epsilon_i}ba^{\epsilon_i}$.

When $g_i$ is almost long, $g_i$ is one of
\begin{equation*}
    Q^{-1}ba^{\epsilon_i}bQ,
    Q^{-1}a^{\epsilon_i}ba^{-\epsilon_i}Q,
    Q^{-1}a^{\epsilon_i}ba^{\epsilon_i}Q,
    Q^{-1}a^{\epsilon_i}ba^{\epsilon_i}ba^{\epsilon_i}Q
\end{equation*}
with $\epsilon_i\in\{1,2\}$, $Q\in G$ and $l(Q)\ge 1$.
If $g_i=Q^{-1}a^{\epsilon_i}ba^{\epsilon_i}Q$ or $g_i=Q^{-1}a^{\epsilon_i}ba^{-\epsilon_i}Q$ then either $g_{i+1}=Q^{-1}a^{-\epsilon_i}b$ or $g_{i+1}=Q^{-1}a^{\epsilon_i}b$. Therefore $g_{i+1}^{-1}g_ig_{i+1}\in\{Q^{-1}ba^{2\epsilon_i}Q,Q^{-1}bQ\}$ which is a contradiction.

We conclude that $g_i$ is either $Q_i^{-1}a^{\epsilon_i}Q_i$ or $Q_i^{-1}a^{\epsilon_i}ba^{\epsilon_i}ba^{\epsilon_i}Q_i$ with $\epsilon_i=1,2$ and $l(Q_i)\ge 0$.

(c) We assume that $n\ge 1$ and suppose that $m_{i-1}+m_i=l(g_i)$ for some $2\le i\le n-1$. By (2), $g_i$ is either $Q_i^{-1}a^{\epsilon_i}Q_i$ or $Q_i^{-1}a^{\epsilon_i}ba^{\epsilon_i}ba^{\epsilon_i}Q_i$ with $\epsilon_i=1,2$ and $l(Q_i)\ge 0$.

If $g_i=a^{\epsilon_i}$ and suppose that $m_{i-1}=0$, $m_i=1$, then $g_{i+1}$ is either one of $a^{\epsilon_i}ba^{\epsilon_i}ba^{\epsilon_i}$, $a^{\epsilon_i}ba^{-\epsilon_i}b$, $a^{-\epsilon_i}ba^{-\epsilon_i}ba^{-\epsilon_i}$ or an almost long element starting with $a^{\epsilon_i}$ and ending with $a^{-\epsilon_i}$. However, $g_{i+1}$ cannot be $a^{-\epsilon_i}ba^{-\epsilon_i}ba^{-\epsilon_i}$ since the elementary transformation $R_i^{-1}$ makes the sum of $\mathcal{S}_2$-complexities smaller.

If $g_i=Q_i^{-1}a^{\epsilon_i}Q_i$ with $l(Q_i)\ge 1$, then either $g_{i-1}=Q_i$ or $g_{i+1}=Q_i^{-1}$, which implies the contradiction either $g_{i-1}g_ig_{i-1}^{-1}=a^{\epsilon_i}$ or $g_{i+1}^{-1}g_ig_{i+1}=a^{\epsilon_i}$.

If $g_i=Q^{-1}a^{\epsilon_i}ba^{\epsilon_i}ba^{\epsilon_i}Q$ with $l(Q)\ge 0$, then either $g_{i-1}=ba^{-\epsilon_i}Q$ or $g_{i+1}=Q^{-1}a^{-\epsilon_i}b$. Therefore Table \ref{table::almost short} denies the case of $Q=1$ and, when $Q\neq 1$, either $g_{i-1}g_ig_{i-1}^{-1}=a^{\epsilon_i}ba^{2\epsilon_i}b$ or $g_{i+1}^{-1}g_ig_{i+1}=ba^{2\epsilon_i}ba^{\epsilon_i}$, which is a contradiction.

The assertion (b) and the above observation show that $m_{i-1}+m_i<l(g_i)$ if $g_i\neq a^{\epsilon_i}$. They further imply a contradiction that $g_1\cdots g_n\neq 1$.
\end{proof}

Now we state the main result in this subsection.

\begin{theorem}
\label{theorem::Livne S2->almost short}
Let $g_1,\ldots,g_n$ be such that each of them is conjugate to some element in $\mathcal{S}_2$ and $g_1\cdots g_n=1$. Then, the $n$-tuple $(g_1,\ldots,g_n)$ is Hurwitz equivalent to
\begin{equation*}
    \prod_{i=1}^m (Q_i^{-1}\tau_{i,1}Q_i,\ldots,Q_i^{-1}\tau_{i,n_i}Q_i)
\end{equation*}
with $m\ge 1$, $\sum_{i=1}^m n_i=n$, $Q_i\in G$ and $\tau_{i,j}\in \mathcal{S}_2$ such that $\tau_{i,1}\cdots\tau_{i,n_i}=1$ for $i=1,\ldots,m$ and $j=1,\ldots,n_i$.
\end{theorem}

\begin{proof}
We will always use the notation $m_i$ to indicate the length of the reduced part in $h_ih_{i+1}$ for $i=1,\ldots,\mu-1$ and set $m_0=m_{\mu}=0$ as before.
To prove the theorem for $(g_1,\ldots,g_n)$, we use the induction on
\begin{equation*}
\Big(n,\sum_{i=1}^{n} f_2(g_i), l(g_1), \ldots, l(g_n)\Big).
\end{equation*}
and apply the following operations:
If there exists a pair of adjacent components of the form $(Q^{-1}\tau_1 Q,Q^{-1}\tau_2 Q)$ with $Q\in G$ and $(\tau_1,\tau_2)$ in Table \ref{table::almost short}, then we replace it with the product $Q^{-1}\tau_1\tau_2Q$ and reduce $(g_1,\ldots,g_n)$ to an $(n-1)$-tuple.
If there exists a triple of consecutive components of the form
\begin{equation*}
(Q^{-1}a^{-\epsilon}ba^{\epsilon}bQ,Q^{-1}ba^{\epsilon}bQ,Q^{-1}ba^{\epsilon}ba^{-\epsilon}Q)
\end{equation*}
with $Q\in G$, $\epsilon=\pm 1$ as introduced in Operation $1'$, then we replace it with $Q^{-1}a^{\epsilon}Q$ and reduce $(g_1,\ldots,g_n)$ to an $(n-2)$-tuple.
If there exists an identical component, then we move it to the rightmost position and remove it.
If there exists a sequence of elementary transformations that makes $\sum_i f_2(h_i)$ smaller, then we apply it.

When each of the above operations fails, the resulting tuple, still denoted by $(g_1,\ldots,g_n)$, satisfies all hypotheses in Lemma \ref{lemma::m_i-1_and_m_i_hard}. Suppose that $n\ge 1$ and there exists some $i=2,\ldots,n-1$ such that $m_{i-1}+m_i>l(g_i)$.

When $g_i=a^{\epsilon_i}$ with $\epsilon_i=1,2$, Table \ref{table::almost short} reveals that either $g_{i-1}$ is one of $ba^{-\epsilon_i}ba^{\epsilon_i}$, $a^{-\epsilon_i}ba^{-\epsilon_i}ba^{-\epsilon_i}$, $a^{\epsilon_i}ba^{\epsilon_i}ba^{\epsilon_i}$, or $g_{i-1}$ is an almost long element starting with $a^{-\epsilon_i}$ and ending with $a^{\epsilon_i}$.
Meanwhile, either $g_{i+1}$ is one of $a^{\epsilon_i}ba^{-\epsilon_i}b$, $a^{-\epsilon_i}ba^{-\epsilon_i}ba^{-\epsilon_i}$, $a^{\epsilon_i}ba^{\epsilon_i}ba^{\epsilon_i}$, or $g_{i+1}$ is an almost long element starting with $a^{\epsilon_i}$ and ending with $a^{-\epsilon_i}$.
The triple $(g_{i-1},g_i,g_{i+1})$ cannot be $(ba^{-\epsilon_i}ba^{\epsilon_i},a^{\epsilon_i},a^{\epsilon_i}ba^{-\epsilon_i}b)$ due to Operation $1'$.
Therefore, either $(g_{i-1},g_i)$ can be transformed into $(\tilde{g}_{i-1},\tilde{g}_i)=(a^{\epsilon_i},a^{-\epsilon_i}g_{i-1}a^{\epsilon})$ with $f_2(g_{i-1})\ge f_2(\tilde{g}_i)$ but $l(\tilde{g}_{i-1})<l(g_{i-1})$.

When $g_i=ba^{\epsilon_i}b$ with $\epsilon_i=1,2$, Table \ref{table::almost short} reveals that either $g_{i-1}=a^{-\epsilon_i}ba^{\epsilon_i}b$ or $g_{i-1}$ is almost long starting with $ba^{-\epsilon_i}$ and ending with $a^{\epsilon_i}b$.
Meanwhile, either $g_{i+1}=ba^{\epsilon}ba^{-\epsilon_i}$ or $g_{i+1}$ is almost long starting with $ba^{\epsilon}$ and ending with $a^{-\epsilon}b$.
The triple $(g_{i-1},g_i,g_{i+1})$ cannot be $(a^{-\epsilon_i}ba^{\epsilon_i}b,ba^{\epsilon_i}b,ba^{\epsilon}ba^{-\epsilon_i})$ due to Operation $1'$.
Therefore, $(g_{i-1},g_i,g_{i+1})$ can be transformed into $(\tilde{g}_{i-1},\tilde{g}_i,\tilde{g}_{i+1})$ with $f_2(g_{i-1})+f_2(g_i)+f_2(g_{i+1})\ge f_2(\tilde{g}_{i-1})+f_2(\tilde{g}_i)+f_2(\tilde{g}_{i+1})$, $l(g_{i-1})\ge l(\tilde{g}_{i-1})$, $l(g_i)\ge l(\tilde{g}_i)$, $l(g_{i+1})\ge l(\tilde{g}_{i+1})$ but either $l(\tilde{g}_{i-1})<l(g_{i-1})$ or $l(\tilde{g}_{i+1})<l(g_{i+1})$.

When $g_i=Q_i^{-1}a^{\epsilon_i}Q$ with $\epsilon_i=1,2$, $l(Q_i)\ge 2$, we have $m_{i-1}=m_i=l(Q_i)+1$ and $l(g_{i-1})>l(g_i)$. To avoid $l(g_i^{-1}g_{i-1}g_i)<l(g_{i-1})$, $g_{i-1}$ must end with $a^{\epsilon_i}Q_i$ and start with $Q_i^{-1}a^{-\epsilon_i}$. In this case, $l(g_i^{-1}g_{i-1}g_i)=l(g_{i-1})$ and, using the elementary transformation $R_{i-1}$, we are able to reduce $(g_1,\ldots,g_n)$ to a new $n$-tuple, say $(\tilde{g}_1,\ldots,\tilde{g}_n)$, such that $\sum_j f_2(g_j)=\sum_j f_2(\tilde{g}_j)$, $l(g_j)=l(\tilde{g}_j)$ for $1\le j\le n$ and $j\not\in\{i-1,i\}$ but $l(\tilde{g}_{i-1})=l(g_i)<l(g_{i-1})=l(\tilde{g}_i)$.

When $g_i=Q_i^{-1}a^{\epsilon_i}ba^{\epsilon_i}ba^{\epsilon_i}Q_i$ with $\epsilon_i=1,2$, $Q_i\in G$ and $l(Q_i)\ge 0$, then $m_{i-1}=m_i=l(Q_i)+3$. If $l(g_{i-1})=l(g_i)=l(g_{i+1})$, then
\begin{equation*}
    (g_{i-1},g_i,g_{i+1}) = (Q_i^{-1}a^{\epsilon_i}ba^{\epsilon_i}ba^{-\epsilon_i}Q_i, 
     Q_i^{-1}a^{\epsilon_i}ba^{\epsilon_i}ba^{\epsilon_i}Q_i,
     Q_i^{-1}a^{-\epsilon_i}ba^{\epsilon_i}ba^{\epsilon_i}Q_i)
\end{equation*}
that can be transformed into a triple with a smaller sum of $\mathcal{S}_2$-complexities via the following substitution.
\begin{align*}
(g_{i-1},g_i,g_{i+1})&\longrightarrow
(Q_i^{-1}a^{\epsilon_i}ba^{\epsilon_i}ba^{\epsilon_i}Q_i,Q_i^{-1}a^{-\epsilon_i}ba^{\epsilon_i}ba^{\epsilon_i}Q_i,Q_i^{-1}a^{-\epsilon_i}ba^{\epsilon_i}ba^{\epsilon_i}Q_i)\\
&\longrightarrow
(Q_i^{-1}a^{-\epsilon_i}ba^{\epsilon_i}ba^{\epsilon_i}Q_i,Q_i^{-1}a^{-\epsilon_i}ba^{-\epsilon_i}ba^{-\epsilon_i}ba^{-\epsilon_i}ba^{\epsilon_i}Q_i,Q_i^{-1}a^{-\epsilon_i}ba^{\epsilon_i}ba^{\epsilon_i}Q_i)\\
&\longrightarrow
(Q_i^{-1}a^{-\epsilon_i}ba^{\epsilon_i}ba^{\epsilon_i}Q_i,Q_i^{-1}a^{-\epsilon_i}ba^{\epsilon_i}ba^{\epsilon_i}Q_i,Q_i^{-1}a^{-\epsilon_i}ba^{\epsilon_i}bQ_i).
\end{align*}
If $l(g_{i-1})=l(g_i)<l(g_{i+1})$, then $g_{i+1}=Q_i^{-1}a^{-\epsilon_i}ba^{\epsilon_i}wa^{-\epsilon_i}ba^{\epsilon_i}Q_i$ with the word $w$ starts and ends with $b$. Therefore, by elementary transformations the triple can be transformed into
\begin{equation*}
    (Q_i^{-1}a^{\epsilon_i}bwba^{-\epsilon_i}Q,g_{i-1},g_i)
\end{equation*}
with a smaller sum of $\mathcal{S}_2$-complexities, which induces a contradiction.
If $l(g_{i-1})>l(g_i)$ then again using the elementary transformation $R_{i-1}$ we are able to reduce $(g_1,\ldots,g_n)$ to a new $n$-tuple, say $(\tilde{g}_1,\ldots,\tilde{g}_n)$, such that $\sum_j f_2(g_j)=\sum_j f_2(\tilde{g}_j)$, $l(g_j)=l(\tilde{g}_j)$ for $1\le j\le n$ and $j\not\in\{i-1,i\}$ but $l(\tilde{g}_{i-1})=l(g_i)<l(g_{i-1})=l(\tilde{g}_i)$.

The induction does not stop unless $n$ is equal to $0$.
Due to Proposition \ref{proposition::reduce_and_reloading_hard}, by restoring operations and applying more elementary transformations, we get a resulting $n$-tuple of almost short elements that can be obtained from the original $(g_1,\ldots,g_n)$ via elementary transformations directly.
\end{proof}

\begin{corollary}
\label{corollary::Livne S2->almost short}
Let $g_1,\ldots,g_n$ be such that each of them is conjugate to some element in $\mathcal{S}_2$ and $g_1\cdots g_n=1$. Let $(g'_1,\ldots,g'_m)$ be a tuple containing a generating set.
Then, the tuple $(g_1,\ldots,g_n)\bullet (g'_1,\ldots,g'_m)$ is Hurwitz equivalent to $(h_1,\ldots,h_n)\bullet (g'_1,\ldots,g'_m)$ where $(h_1,\ldots,h_n)$ is an $n$-tuple of almost short elements.
\end{corollary}
\begin{proof}
The corollary follows from Theorem \ref{theorem::Livne S2->almost short} and Lemma \ref{lemma::tuples containing a generating set}.
\end{proof}

Recall that $\mathcal{F}_{13}=(b,b)^2\bullet(a^2bab,t_0,s_1)^3$.
Theorem \ref{theorem::Livne mathcalS2} and Theorem \ref{theorem::Livne mathcalS2_modification} show that we are able to construct the normal form of $(g_1,\ldots,g_n)\bullet \mathcal{F}_{13}$ that depends only on the number of components in each conjugacy class as in Theorem \ref{theorem::Livne mathcalS}.
We prove them using the following.

\begin{theorem}
\label{theorem::Livne mathcalS2_step1}
Let $g_1,\ldots,g_n\in \PSL(2,\mathbb{Z})$ be conjugates of $a$, $a^2$, $b$, $aba$, $a^2ba^2$ or $ababa$ satisfying $g_1\cdots g_n=1$. Suppose that $m$ of them are conjugates of $ababa$.
Let $(v_1,\ldots,v_c)$ be a tuple containing a generating set,
and let $(v'_1,\ldots,v'_{c'})$ be a $(b,b,b,b)$-expanding tuple whose components are conjugate to $a,a^2,b,s_0$ or $t_0$.
Then,
\begin{equation*}
    (g_1,\ldots,g_n)\bullet (v'_1,\ldots,v'_{c'})\bullet (v_1,\ldots,v_c)
\end{equation*}
is Hurwitz equivalent to
\begin{equation*}
    (h_1,\ldots,h_{n'})\bullet
    (a^2bab,ba^2ba)^{(m-3+\mu)/2}\bullet
    (v_1,\ldots,v_c)
\end{equation*}
where all components of $(h_1,\ldots,h_{n'})$ are conjugate to $a$, $a^2$, $b$, $s_0$, $t_0$, $ababa$ and only $3-\mu$ of them are conjugate to $ababa$, where $\mu=3-m$ if $m\le 3$ and $\mu=(m+1) \mod{2}$ otherwise.
\end{theorem}

\begin{proof}
We assume that $(v'_1,\ldots,v'_{c'})=(b,b,b,b)$ without loss of generality. Rewrite $(g_1,\ldots,g_n)\bullet (b,b,b,b)\bullet (v_1,\ldots,v_c)$ as
\[
(h_1,\dots,h_k)\bullet 
(b,b)\bullet(v_1,\ldots,v_c)\bullet (a^2bab,ba^2ba)^l,
\]
with $(h_1,\ldots,h_k)=(g_1,\ldots,g_n,b,b)$ containing at least one component conjugate to $b$, with $k=n+2$ and $l=0$.
Following Corollary \ref{corollary::Livne S2->almost short}, we transform $(h_1,\ldots,h_k)$ into a tuple of almost short elements that contains at least one component of the form either $b$, $a^2ba$ or $aba^2$. 

We first apply the following inductions on $k$ when $m-2l>3$.

Suppose that there exist two components, say $h_i$ and $h_j$ with $i\neq j$, such that both of them are conjugate to $ababa$ and $h_i h_j=1$. We move $h_i$ and $h_j$ to the rightmost positions. Since $(v_1,\ldots,v_c)$ contains a generating set, by Lemma \ref{lemma::tuples containing a generating set}, they are further transformed into a pair of the form $(a^2bab,ba^2ba)$ by elementary transformations on $(g_1,\ldots,g_n)\bullet (b,b)\bullet (v_1,\ldots,v_c)$.
Therefore, we get the following tuple
\[
(\tilde{h}_1,\ldots,\tilde{h}_{k-2})\bullet (b,b)\bullet(v_1,\ldots,v_c)\bullet (a^2bab,ba^2ba)^{l+1}
\]
where $(\tilde{h}_1,\ldots,\tilde{h}_{k-2})$ is further transformed into a tuple of almost short elements.

Suppose that with a pair
\[
(\tau_b,\tau_{ababa})\in
\Bigg\{\begin{array}{c}
  (b,a^2bab),(b,ba^2ba),(b,baba^2),(b,aba^2b),\\
  (aba^2,ababa),(aba^2,a^2ba^2ba^2),(aba^2,baba^2),(aba^2,aba^2b),\\
  (a^2ba,ababa),(a^2ba,a^2ba^2ba^2),(a^2ba,a^2bab),(a^2ba,ba^2ba)
  \end{array}
\Bigg\}
\]
there exist some components of the form $\tau_b$ and at least two components of the form $\tau_{ababa}$. Using elementary transformations we gather them together and obtain a pair of mutually inverse elements via the following substitution.
\[
(\tau_b,\tau_{ababa},\tau_{ababa})\longrightarrow (\tau_{ababa}^{-1},b,\tau_{ababa})\longrightarrow (\tau_{ababa}^{-1},\tau_{ababa},\tau_{ababa}^{-1}b\tau_{ababa}).
\]
The pair of mutually inverse elements is further moved to the rightmost position and transformed into $(a^2bab,ba^2ba)$. The resulting tuple again has the expression with a lower $k$.

Once the above induction stops but $m-2l>3$, there is at most one almost short element conjugate to $ababa$, say $\tau_{abab}$, that appears more than once in $(h_1,\ldots,h_k)$.
Take a proper $\tau_b\in\{b,a^2ba,aba^2\}$ such that $(\tau_b,\tau_{ababa})$ belongs to the above set of pairs. Transform the extra pair $(b,b)$ into $(\tau_b,\tau_b)$ with the help of $(v_1,\ldots,v_c)$. Again we gather all components of the form $\tau_{ababa}$ in $(h_1,\ldots,h_k)$ together with an additional $\tau_b$ using elementary transformations and apply the following substitutions.
\[
(\tau_b,\tau_{ababa},\ldots,\tau_{ababa})\longrightarrow (\tau_{ababa}^{-1},b,\tau_{ababa},\ldots,\tau_{ababa})\longrightarrow
(\tau_{ababa}^{-1},\ldots,\tau_{ababa}^{-1},b,\tau_{ababa},\ldots,\tau_{ababa}).
\]
We make all mutually inverse elements within the above resulting tuple pairs of the form $(a^2bab,ba^2ba)$ and move them to the rightmost positions. By elementary transformations the tuple $(g_1,\ldots,g_n)\bullet (v'_1,\ldots,v'_{c'})\bullet (v_1,\ldots,v_c)$ has been finally transformed into
\[
(h_1,\ldots,h_k)\bullet
(v_1,\ldots,v_c)\bullet
(a^2bab,ba^2ba)^{m'/2}
\]
where $(h_1,\ldots,h_k)$ is a tuple of almost short elements containing at most three components conjugate to $ababa$ and $m-m'\le 3$.
\end{proof}

\begin{proof}[Proof of Theorem \ref{theorem::Livne mathcalS2} and Theorem \ref{theorem::Livne mathcalS2_modification}]
We only prove Theorem \ref{theorem::Livne mathcalS2} while the proof of Theorem \ref{theorem::Livne mathcalS2_modification} is similar.
Since $(a^2bab,t_0,s_1)$ contains a generating set, Theorem \ref{theorem::Livne mathcalS2_step1} has shown that the concatenation $(g_1,\ldots,g_n)\bullet \mathcal{F}_{13}$ can be transformed into
\begin{equation*}
    (h_1,\ldots,h_{n'})\bullet
    (a^2bab,ba^2ba)^{(m-3+\mu)/2}\bullet
    (a^2bab,t_0,s_1)^3
\end{equation*}
where $\mu$ is determined by $m$, only $3-\mu$ components of $(h_1,\ldots,h_{n'})$ are conjugate to $ababa$ and the rest are conjugates of $a$, $a^2$, $b$, $s_0$ or $t_0$.

Consider each of $i=1,2,3$ in turn.
Let $h_{\alpha}$ be the first component in $(h_1,\ldots,h_{n'})$ conjugate to $ababa$.
Since the first triple of the form $(a^2bab,t_0,s_1)$ within the concatenation contains a genearting set, by Lemma \ref{lemma::tuples containing a generating set} we can transform $(h_1,\ldots,h_{n'})$ into a tuple with a simultaneous conjugation such that $h_{\alpha}=ba^2ba$
Therefore we are able to make the $\alpha$-th component in $(h_1,\ldots,h_{n'})$ and the first component in the triple a pair of the form $(a^2bab,ba^2ba)$.
Hence $(h_1,\ldots,h_{n'})\bullet(a^2bab,t_0,s_1)^{4-i}$ is transformed into $(h'_1,\ldots,h'_{n''})\bullet(a^2bab,t_0,s_1)^{3-i}\bullet (a^2bab,ba^2ba)$ with $n''=n'+1$.
\end{proof}

%% file: Section5/fails.tex
\section{Stable classification vs. unstable classification}

\subsection{Hurwitz equivalence fails without stabilisation}

\label{subsection::fails}

We give some examples of global monodromies which are Hurwitz equivalent up to stabilisation,
as in Theorem \ref{thmx::B} and Theorem \ref{thmx::C},
but fail to be Hurwitz equivalent.
This illustrates why it is necessary to consider fibrations up to the fibre-connected sum.

\begin{example}
Let $f_1$ be an achiral Lefschetz fibration which has a global monodromy of the form $(-A^2B,-BA,-A^2B,-BA)$ and let $f_2$ be an achiral Lefschetz fibration which has a global monodromy of the form $(-A^2B,-BA,-ABA,A^2BA^2)$. Though $f_1$ and $f_2$ have the same type of singularities, these two global monodromies are not Hurwitz equivalent.

\begin{proof}
Indeed, the following graph shows all resulting tuples in $\PSL(2,\mathbb{Z})$ from $(s_0,t_0,s_0,t_0)$ using elementary transformations.
\begin{center}
\begin{tikzpicture}[node distance={16mm}, thick]

\node (1) {$(s_0,t_0,t_0,s_0)$};
\node (2) [below left=10mm of 1] {$(s_0,s_0,t_0,t_0)$};
\node (3) [below right=10mm of 1] {$(t_0,t_0,s_0,s_0)$};
\node (4) [below right=10mm of 2] {$(t_0,s_0,s_0,t_0)$};
\node (5) [left=10mm of 2] {$(s_0,t_0,s_0,t_0)$};
\node (6) [right=10mm of 3] {$(t_0,s_0,t_0,s_0)$};

\draw[<->] (1) to [out=180,in=90,looseness=0.5]
  node[above]
  {$R_3^{\pm 1}$}
(5);
\draw[<->] (1) to [out=0,in=90,looseness=0.5]
  node[above]
  {$R_1^{\pm 1}$}
(6);
\draw[<->] (4) to [out=180,in=270,looseness=0.5]
  node[below]
  {$R_1^{\pm 1}$}
(5);
\draw[<->] (4) to [out=0,in=270,looseness=0.5]
  node[below]
  {$R_3^{\pm 1}$}
(6);
\draw[<->] (2) to
  node[above]
  {$R_2^{\pm 1}$}
(5);
\draw[<->] (3) to
  node[above]
  {$R_2^{\pm 1}$}
(6);
\draw[->] (2) to [out=45,in=10,looseness=3.5]
  node[above]
  {$R_1^{\pm 1}$}
(2);
\draw[->] (2) to [out=315,in=350,looseness=3.5]
  node[below]
  {$R_3^{\pm 1}$}
(2);
\draw[->] (3) to [out=135,in=170,looseness=3.5]
  node[above]
  {$R_1^{\pm 1}$}
(3);
\draw[->] (3) to [out=225,in=190,looseness=3.5]
  node[below]
  {$R_3^{\pm 1}$}
(3);
\draw[->] (1) to [out=240,in=300,looseness=5.]
  node[below]
  {$R_2^{\pm 1}$}
(1);
\draw[->] (4) to [out=60,in=120,looseness=5.]
  node[above]
  {$R_2^{\pm 1}$}
(4);

\end{tikzpicture} 
\end{center}
In particular, one cannot transform $(s_0,t_0,s_0,t_0)$ into $(s_0,t_0,s_1,t_1)$.
\end{proof}
\end{example}

\begin{example}
Let $(b,b,a^2bab,ba^2ba)$ and $(aba^2,a^2ba,a^2bab,baba^2)$ be tuples in $\PSL(2,\mathbb{Z})$. We claim that for arbitrary positive integer $N$,
\begin{equation*}
    (b,b,a^2bab,ba^2ba)\bullet (b,b)^N
\end{equation*}
cannot be transformed into
\begin{equation*}
    (aba^2,a^2ba,a^2bab,baba^2)\bullet (b,b)^N
\end{equation*}
by elementary transformations.

\begin{proof}
Assume that $(b,b,a^2bab,ba^2ba)\bullet (b,b)^N$ can be transformed into $(aba^2,a^2ba,a^2bab,baba^2)\bullet (b,b)^N$ by elementary transformations for some $N$. Then there exists an element $g\in \PSL(2,\mathbb{Z})$ which is a product of $b$, $a^2bab$ and $ba^2ba$ such that $aba^2=g^{-1}bg$,
which implies $b=(ga)^{-1}b(ga)$.
Therefore, the element $g$ is either $a^2$ or $ba^2$,
but the number of occurrences of the letter $a$ in $g$ modulo $3$ is equal to $0$,
which is a contradiction.
\end{proof}
\end{example}

%% file: Section5/achiralLef.tex
\subsection{Unstable classification of achiral Lefschetz fibrations}

\label{subsection::achiralLef}

We consider torus achiral Lefschetz fibrations
having fixed cardinality of branch sets $|\mathcal{B}|=n\ge 1$.
A singular fibre is of type $I_1^{+}$ if its fibre monodromy is conjugate to $L=-ABA=\begin{bmatrix}1 & 0\\1 & 1\end{bmatrix}$
and a singular fibre is of type $I_1^{-}$ if its fibre monodromy is conjugate to $R=-AB=\begin{bmatrix}1 & 1\\0 & 1\end{bmatrix}$.

By Theorem \ref{thmx::B},
global monodromies of a pair of torus achiral Lefschetz fibrations are Hurwitz equivalent after performing direct sums with $f_{12}^L$
if and only if they have the same type of singularities.
However, the Hurwitz equivalence between global monodromies is more difficult to state, especially when singular fibres of type $I_1^+$ and $I_1^-$ occur in pairs.
In this subsection,
we enumerate all possible Hurwitz equivalent classes of global monodromies of torus achiral Lefschetz fibrations, without stabilisation.
This will prove Theorem \ref{thmx::achiralLef}.

\hfill

By a \emph{rooted tree} we mean a directed tree in which a specific vertex is called the \emph{root}, such that each directed edge indicates the parent-child relationship between two vertices.
A \emph{rooted forest} is a disjoint union of several rooted trees.
In general, given a (directed) graph $\Gamma$, we always use $V(\Gamma)$ to denote the set of vertices.

\begin{definition}
Given a rooted forest $T$ and a non-negative integer $k$,
we define $\Omega(T,k)$ to be the set of formal sums
$\sum_{v\in V(T)}  m_v\cdot v$ over vertices with $m_v\ge 0$ and $\sum m_v=k$ such that any two vertices $v_1\neq v_2$ with $m_{v_1}\ge 1$, $m_{v_2}\ge 1$
have no ancestor-descendant relationship (i.e. there does not exist a directed path joining $v_1$ to $v_2$).
\end{definition}

\begin{definition}
For $n=p+q$ with $p\ge 0$, $q\ge 0$,
we define
$\Hom_{p,q}^{\text{aL}}(\mathbb{F}_{n-1},\SL(2,\mathbb{Z}))$
to be the set consisting of all monodromy homomorphisms of torus achiral Lefschetz fibrations $f:M\rightarrow S^2$ with 
$\mathcal{O}(f)=[\underbrace{I_1^+,\ldots,I_1^+}_{\text{$p$ components}},\underbrace{I_1^-,\ldots,I_1^-}_{\text{$q$ components}}]$.
\end{definition}

\begin{theorem}
\label{theorem::achiralLef}
Let $n$, $p$ and $q$ be arbitrary integers such that $n\ge 1$, $p\ge 0$, $q\ge 0$ and $p+q=n$.
\begin{itemize}[-,topsep=0pt,itemsep=-1ex,partopsep=1ex,parsep=1ex]
\item If $p\neq q$, then the set
$B_n\backslash \Hom_{p,q}^{\text{aL}}(
\mathbb{F}_{n-1}, \SL(2,\mathbb{Z}))$
is a singleton.

\item If $p=q$, then there exists a one-to-one correspondence:
\begin{align*}
B_n\backslash& \Hom_{p,q}^{\text{aL}}(\mathbb{F}_{n-1},\SL(2,\mathbb{Z}))
\longleftrightarrow\\
&\{pt.\}
\sqcup \Big( \bigsqcup_{k=0}^{p-1} \Omega(T_{\infty},k) \Big)
\sqcup \Big( \bigsqcup_{k=0}^{p-1} \Omega(T_{\infty},k) \Big)
\sqcup \Big( \bigsqcup_{k=0}^{p-1} \Omega(T_{\infty},k) \Big)
\sqcup \Omega(T_{\infty}\sqcup T_{\infty}\sqcup T_{\infty},p)
\end{align*}
where $T_{\infty}$ is the rooted complete infinite binary tree.
\end{itemize}
\end{theorem}

Recall that each matrix $g$ with non-zero trace in $\SL(2,\mathbb{Z})$ is uniquely expressed by $\epsilon Q$ with $\epsilon=\pm I$ and $Q$ a word in $\{A,A^2,B\}$ in which $B$'s and powers of $A$ appear alternatively.
The \emph{length} of an element $g\in\SL(2,\mathbb{Z})$ is defined as the length of the word $Q$, denoted by $l(g)$.
Here we list all possibilities for fibre monodromies of a torus achiral Lefschetz fibration:
\begin{align*}
-A^2B, -ABA, -BA^2 \text{ and } \epsilon PABAQ,\\
-BA, A^2BA^2, -AB \text{ and } \epsilon PA^2BA^2Q
\end{align*}
where $P$, $Q$ are words in $\{A,A^2,B\}$ in which $B$'s and powers of $A$ appear alternatively, $PQ=\pm I$, $\epsilon=\pm I$ is uniquely determined by $l(Q)$ such that the trace is equal to $+2$.

Let $g_1$ and $g_2$ be matrices in $\SL(2,\mathbb{Z})$.
Suppose that $g_1$ and $g_{2}$ are expressed by 
$\epsilon t_k\ldots t_1$ and
$\tilde{\epsilon} \tilde{t}_1\ldots \tilde{t}_l$,
respectively,
with $\epsilon, \tilde{\epsilon}\in\{I,-I\}$, $k=l(g_1)$, $l=l(g_2)$, $t_j\in\{A,A^2,B\}$, $j=1,\ldots,k$ and $\tilde{t}_j\in\{A,A^2,B\}$, $j=1,\ldots,l$.
The product $g_1g_2$ is either
\[
\tau t_k\ldots t_m r\tilde{t}_{m+1}\ldots \tilde{t}_l
\text{ or }
\tau t_k\ldots t_m r
\text{ or }
\tau r \tilde{t}_{m+1}\ldots \tilde{t}_l
\]
for some $\tau=\tau(g_1,g_2) \in\{I,-I\}$,
$r=r(g_1,g_2) \in G$ such that $l(r)\le 1$ and $0\le m=m(g_1,g_2)\le k,l$.
%The coefficient $m$ is further indicated by $m=m(g_1,g_2)$.

Let $Q=BA^{k_1}BA^{k_2}\cdots BA^{k_m}B^{\lambda}$ be a matrix in $\SL(2,\mathbb{Z})$ with $m\ge 0$, $k_1,\ldots,k_m\in\{1,2\}$ and $\lambda\in\{0,1\}$.
We introduce the \emph{suffix tree} $T_Q$ which is a rooted binary tree with infinitely many vertices, whose each vertex is labelled with a pair of inverse elements in $\SL(2,\mathbb{ Z})$.
Set
\begin{equation*}
\widetilde{Q} = B^{\lambda}A^{3-k_m}B\cdots A^{3-k_2}BA^{3-k_1}B
\end{equation*}
so that $\widetilde{Q}Q=\pm I$
and let $\epsilon=\pm I$ be such that $\trace(\epsilon \widetilde{Q} ABA Q)=2$. 
The root of $T_Q$ is labelled with the pair
\begin{equation*}
(\epsilon \widetilde{Q} ABA Q,-\epsilon \widetilde{Q} A^2BA^2 Q)
\end{equation*}
and the suffix tree $T_Q$ is defined by the following form iteratively, where each directed edge indicates the parent-child relationship between a vertex and the root of a suffix tree.

\begin{center}
\begin{tikzpicture}[node distance={16mm}, thick]
\node (x) {$(\epsilon \widetilde{Q} ABA Q,-\epsilon \widetilde{Q} A^2BA^2 Q)$};
\node (l) [below left of=x] {$T_{BAQ}$};
\node (r) [below right of=x] {$T_{BA^2Q}$};
\draw[->] (x) to (l);
\draw[->] (x) to (r);
\end{tikzpicture} 
\end{center}

All conjugates of $L$ and $R$ occur in pairs. They are in one-to-one correspondence with the vertices in the following infinite directed graph $\Gamma$, where each directed edge again indicates the parent-child relationship between a vertex and the root of a suffix tree.

\begin{center}
\begin{tikzpicture}[node distance={16mm}, thick]
\node (x) {$(-A^2B,-BA)$};
\node (y) [right of=x, node distance={32mm}] {$(-ABA,A^2BA^2)$};
\node (z) [right of=y, node distance={32mm}] {$(-BA^2,-AB)$};
\node (tx) [below of=x] {$T_{BA^2}$};
\node (ty) [below of=y] {$T_{B}$};
\node (tz) [below of=z] {$T_{BA}$};
\draw[->] (x) to (ty);
\draw[->] (x) to (tz);
\draw[->] (y) to (tx);
\draw[->] (y) to (tz);
\draw[->] (z) to (tx);
\draw[->] (z) to (ty);
\end{tikzpicture} 
\end{center}

The vertices labelled by $(-A^2B,-BA)$, $(-ABA,A^2BA^2)$ or $(-BA^2,-AB)$ are called \emph{exceptional}.
The components of these labels project to short elements in $\PSL(2,\mathbb{Z})$, as in Section \ref{section::Livne}.

Let $(g_1,\ldots,g_n)$ be a global monodromy of torus achiral Lefschetz fibrations, which is an $n$-tuple of elements conjugate to either $L$ or $R$.
The \emph{complexity} of this tuple is defined to be
\begin{equation*}
\cxty(g_1,\ldots,g_n):=\sum_i \cxty(g_i),
\text{~~~~~}
\cxty(g_i) =
\begin{cases}
    l(Q) & \text{if $g_i=\epsilon \widetilde{Q}ABA Q$ or $g_i=\epsilon \widetilde{Q}A^2BA^2 Q$}\\
        0 & \text{otherwise}
\end{cases}.
\end{equation*}

\begin{lemma}
\label{lemma::a_pair_of_vertices_in_Gamma}
Let $(e_1,e_1^{-1})$, $(e_2,e_2^{-1})\in V(\Gamma)$ be distinct vertices.
\begin{enumerate}[(a)]
    \item If there exists an ancestor-descendant relationship between $(e_1,e_1^{-1})$ and $(e_2,e_2^{-1})$, then by a sequence of elementary transformations the quadruple
    \[(e_1,e_1^{-1},e_2,e_2^{-1})\]
    can be transformed into a quadruple of the form
    \[({e'}_1,{e'}_1^{-1},{e'}_2,{e'}_2^{-1})\]
    such that
    $\cxty(e_1,e_1^{-1},e_2,e_2^{-1})>\cxty({e'}_1,{e'}_1^{-1},{e'}_2,{e'}_2^{-1})$.
    
    \item If one of $(e_1,e_1^{-1})$, $(e_2,e_2^{-1})$ is not exceptional and they have no ancestor-descendant relationship, then $m(g_1,g_2)\le \min\{\frac{l(g_1)}{2}, \frac{l(g_2)}{2}\}$ for any $g_1, g_2\in \{e_1,e_1^{-1},e_2,e_2^{-1}\}$ unless $g_1 g_2= I$.
\end{enumerate}
\end{lemma}

\begin{proof}
(a) We first assume that $(e_1,e_1^{-1})$ is exceptional. Then $(e_2,e_2^{-1})$ cannot be exceptional and we suppose that $(e_2,e_2^{-1})=(\epsilon_2 P_2ABA Q_2,-\epsilon_2 P_2 A^2BA^2 Q_2)$.

When $(e_1,e_1^{-1})=(-A^2B,-BA)$ the pair $(e_2,e_2^{-1})$ is a vertex of either $T_B$ or $T_{BA}$.
If it is a vertex of $T_B$,
then both words $P_2ABAQ_2$ and $P_2A^2BA^2Q_2$ end with $AB$ and start with $BA$. Therefore the following sequence transforms $(e_1,e_1^{-1},e_2,e_2^{-1})$ into a desired quadruple with a smaller complexity.
\begin{equation*}
    (-A^2B,-BA,e_2,e_2^{-1})\longrightarrow
    (-A^2B,e_2,e_2^{-1},-BA)\longrightarrow
    (-A^2B,-BA,A^2Be_2BA,A^2Be_2^{-1}BA).
\end{equation*}
Otherwise, $(e_2,e_2^{-1})$ is a vertex of $T_{BA}$ and then both $P_2ABAQ_2$ and $P_2A^2BA^2Q_2$ end with $ABA$ and start with $A^2BA$. Therefore, the following substitution is desired.
\begin{equation*}
    (-A^2B,-BA,e_2,e_2^{-1})\longrightarrow
    (e_2,e_2^{-1},-A^2B,-BA)\longrightarrow
    (-A^2B,-BA,BAe_2A^2B,BAe_2^{-1}A^2B).
\end{equation*}

When $(e_1,e_1^{-1})=(-ABA,A^2BA^2)$ or $(e_1,e_1^{-1})=(-BA^2,-AB)$, we have similar arguments.

Now we assume that both $(e_1,e_1^{-1})$ and $(e_2,e_2^{-1})$ are unexceptional.
Suppose that
$(e_i,e_i^{-1})=(\epsilon_i P_i ABA  Q_i,-\epsilon_i P_i A^2BA^2 Q_i)$,
for $i=1,2$.
Then $Q_2$ is extended from $Q_1$ by a product of finitely many but at least one $BA$ or $BA^2$ on the left, say
\begin{equation*}
    Q_2=\big(\prod_{i=1}^{\mu} BA^{r_i}\big) Q_1
\end{equation*}
with $\mu\ge 1$ and $r_i\in\{1,2\}$, for each $i=1,\ldots,\mu$.
Therefore, the following substitution is desired for the case $r_{\mu}=1$.
\begin{align*}
    \big(&e_1,e_1^{-1},e_2,e_2^{-1}\big)\\
    =
    \big(&\epsilon_1 P_1 ABA  Q_1,-\epsilon_1 P_1 A^2BA^2 Q_1,\epsilon_2 P_2 ABA  Q_2,-\epsilon_2 P_2 A^2BA^2 Q_2\big)\\
    =
    \big(&\epsilon_1 P_1 ABA  Q_1,
    -\epsilon_1 P_1 A^2BA^2 Q_1,\\
    &\epsilon_2 P_1(\prod_{i=\mu}^{1} A^{3-r_i}B) ABA  (\prod_{i=1}^{\mu} BA^{r_i}) Q_1 ,-\epsilon_2 P_1(\prod_{i=\mu}^{1} A^{3-r_i}B) A^2BA^2 (\prod_{i=1}^{\mu} BA^{r_i}) Q_1\big)\\
    \rightarrow 
    \big(&\epsilon_1 P_1 ABA  Q_1,
    \\
    &\epsilon_2 P_1(\prod_{i=\mu}^{1} A^{3-r_i}B) ABA  (\prod_{i=1}^{\mu} BA^{r_i}) Q_1 ,-\epsilon_2 P_1(\prod_{i=\mu}^{1} A^{3-r_i}B) A^2BA^2 (\prod_{i=1}^{\mu} BA^{r_i}) Q_1,\\
    &-\epsilon_1 P_1 A^2BA^2 Q_1\big)\\
    \rightarrow
    \big(&\epsilon_1 P_1 ABA  Q_1,
    -\epsilon_1 P_1 A^2BA^2 Q_1,\\
    &\epsilon_2 P_1A(\prod_{i=\mu-1}^{1} A^{3-r_i}B) ABA  (\prod_{i=1}^{\mu-1} BA^{r_i}) A^2Q_1 ,-\epsilon_2 P_1A(\prod_{i=\mu-1}^{1} A^{3-r_i}B) A^2BA^2 (\prod_{i=1}^{\mu-1} BA^{r_i}) A^2 Q_1\big).
\end{align*}
Besides, the following substitution is desired for the case $r_{\mu}=2$.
\begin{align*}
    \big(&e_1,e_1^{-1},e_2,e_2^{-1}\big)\\
    =
    \big(&\epsilon_1 P_1 ABA  Q_1,-\epsilon_1 P_1 A^2BA^2 Q_1,\epsilon_2 P_2 ABA  Q_2,-\epsilon_2 P_2 A^2BA^2 Q_2\big)\\
    =
    \big(&\epsilon_1 P_1 ABA  Q_1,
    -\epsilon_1 P_1 A^2BA^2 Q_1,\\
    &\epsilon_2 P_1(\prod_{i=\mu}^{1} A^{3-r_i}B) ABA  (\prod_{i=1}^{\mu} BA^{r_i}) Q_1 ,-\epsilon_2 P_1(\prod_{i=\mu}^{1} A^{3-r_i}B) A^2BA^2 (\prod_{i=1}^{\mu} BA^{r_i}) Q_1\big)\\
    \rightarrow 
    \big(&\epsilon_2 P_1(\prod_{i=\mu}^{1} A^{3-r_i}B) ABA  (\prod_{i=1}^{\mu} BA^{r_i}) Q_1 ,-\epsilon_2 P_1(\prod_{i=\mu}^{1} A^{3-r_i}B) A^2BA^2 (\prod_{i=1}^{\mu} BA^{r_i}) Q_1,\\
    &\epsilon_1 P_1 ABA  Q_1,-\epsilon_1 P_1 A^2BA^2 Q_1\big)\\
    \rightarrow
    \big(&\epsilon_1 P_1 ABA  Q_1,
    -\epsilon_1 P_1 A^2BA^2 Q_1,\\
    &\epsilon_2 P_1A^2(\prod_{i=\mu-1}^{1} A^{3-r_i}B) ABA  (\prod_{i=1}^{\mu-1} BA^{r_i}) AQ_1 ,-\epsilon_2 P_1A^2(\prod_{i=\mu-1}^{1} A^{3-r_i}B) A^2BA^2 (\prod_{i=1}^{\mu-1} BA^{r_i}) A Q_1\big).
\end{align*}

(b) When one of $(e_1,e_1^{-1})$ and $(e_2,e_2^{-1})$ is exceptional, the other belongs to the unique sub-tree either $T_{BA^2}$, $T_B$ or $T_{BA}$.
Therefore, $m(g_1,g_2)\le 1$ unless $g_1g_2=I$.
When $(e_1,e_1^{-1})$ and $(e_2,e_2^{-1})$ belong to different sub-trees of $T_{BA^2}$, $T_B$ and $T_{BA}$, again we have $m(g_1,g_2)\le 1$ unless $g_1g_2=I$.

Now we assume that $(e_1,e_1^{-1})$ and $(e_2,e_2^{-1})$ are vertices of the same sub-tree either $T_{BA^2}$, $T_B$ or $T_{BA}$.
Let $(\epsilon P ABA Q, -\epsilon P A^2BA^2 Q)$ be the lowest common ancestor of $(e_1,e_1^{-1})$ and $(e_2,e_2^{-1})$.
Therefore, there exist the reduced forms of $e_1$, $e_1^{-1}$, $e_2$ and $e_2^{-1}$ such that
\begin{equation*}
    (e_i,e_i^{-1}) = (\epsilon_i P (A^{3-r_i}B) \omega_i (BA^{r_i}) Q, -\epsilon_i P (A^{3-r_i}B) \omega_i^{-1} (BA^{r_i}) Q)
\end{equation*}
with $r_i\in\{1,2\}$, $\epsilon_i=\pm I$, $\omega_i\in \SL(2,\mathbb{Z})$ and $r_1\neq r_2$, for $i=1,2$.
Hence, $m(g_1,g_2)\le l(Q)+1$ unless $g_1g_1=I$. 
\end{proof}

Suppose that a tuple of the form $(e_1,e_1^{-1},\ldots,e_p,e_p^{-1})$ in $\SL(2,\mathbb{Z})$
is a global monodromy of torus achiral Lefschetz fibrations.
One can write it as a formal sum $\sum_{v\in V(\Gamma)} m_v\cdot v$
such that $\sum m_v = p$.
By Lemma \ref{lemma::move subtuples},
different tuples which can be written as the same formal sum are Hurwitz equivalent.

\begin{lemma}
\label{lemma::no-ancestor-descendant-relationship}
Let $\sum_{v\in V(\Gamma)} m_v\cdot v$ be a formal sum over vertices of $\Gamma$.
Let
$(e_1,e_1^{-1},\ldots,e_p,e_p^{-1})$ be a $(2p)$-tuple in $\SL(2,\mathbb{Z})$ expressed by $\sum_{v\in V(\Gamma)} m_v\cdot v$.
Suppose that there exist distinct vertices $v_1$, $v_1$ such that $m_{v_1}\ge 1$, $m_{v_2}\ge 1$ and there exists an ancestor-descendant relationship between $v_1$ and $v_2$.
Then, by a sequence of elementary transformations the $(2p)$-tuple can be transformed into a tuple of the form $({e'}_1,{e'}_1^{-1},\ldots,{e'}_p,{e'}_p^{-1})$ with a smaller complexity.
\end{lemma}

\begin{proof}
The lemma follows from Lemma \ref{lemma::a_pair_of_vertices_in_Gamma} (a).
\end{proof}

On the other hand, we have the following lemma.

\begin{lemma}
\label{lemma::2p tuple without vertices joined by a directed path}
Let $\sum_{v\in V(\Gamma)} m_v\cdot v$ be a formal sum over vertices of $\Gamma$.
Let
$(e_1,e_1^{-1},\ldots,e_p,e_p^{-1})$ be a $(2p)$-tuple in $\SL(2,\mathbb{Z})$
expressed by $\sum_{v\in V(\Gamma)} m_v\cdot v$.
Suppose that any two distinct vertices $v_1$, $v_2$ with $v_1\ge 1$, $v_2\ge 1$
have no ancestor-descendant relationship.
Then, either
\begin{enumerate}[(i)]
    \item there exist at least two distinct exceptional vertices with $m_v\ge 1$, or
    \item there exists at most one of the three exceptional vertices satisfying $m_v \ge 1$.
\end{enumerate}

In Case (\RNum{1}), all components of the $(2p)$-tuple are short (i.e. $\cxty(e_i)=0$ for $i=1,\ldots,p$) and by a sequence of elementary transformations the $(2p)$-tuple can be transformed into
\begin{equation*}
    (-A^2B,-BA)\bullet (-ABA,A^2BA^2)^{p-1}.
\end{equation*}

In Case (\RNum{2}),
the tuple 
$(e_1,e_1^{-1},\ldots,e_p,e_p^{-1})$ is minimal according to the complexity among tuples obtained from $(e_1,e_1^{-1},\ldots,e_p,e_p^{-1})$ using a sequence of elementary transformations.
Besides, all minimal tuples according to the complexity among them of the form $({e'}_1,{e'}_1^{-1},\ldots,{e'}_p,{e'}_p^{-1})$ are expressed by the formal sum
$\sum_{v\in V(\Gamma)} m_v\cdot v$.
\end{lemma}

\begin{proof}
As in Case (\RNum{1}), when there exist at least two distinct exceptional vertices occurring in the product form,
by an elementary transformation,
the corresponding quadruple can be transformed into
\[(-A^2B,-BA,-ABA,A^2BA^2).\]
Besides, the resulting quadruple contains a generating set of $\SL(2,\mathbb{Z})$.
Therefore, by Lemma \ref{lemma::tuples containing a generating set}, the $(2p)$-tuple can be transformed into $(-A^2B,-BA)\bullet (-ABA,A^2BA^2)^{p-1}$, as desired.

In Case (\RNum{2}), we assume that there exists a sequence of elementary transformations that transforms $(e_1,e_1^{-1},\ldots,e_p,e_p^{-1})$ into a new tuple with a smaller complexity or a new tuple of the form $({e'}_1,{e'}_1^{-1},\ldots,{e'}_p,{e'}_p^{-1})$ with the same complexity but expressed by a different formal sum of vertices in $V(\Gamma)$.
Therefore, there exists at least one component of the new tuple, say $Q^{-1}\omega Q$, where $\omega$ is equal to some component of $(e_1,e_1^{-1},\ldots,e_p,e_p^{-1})$ and $Q$ is a product of $e_1,\ldots,e_p$ and their inverses, such that $\cxty(Q^{-1}\omega Q)<\cxty(\omega)$.
It contradicts Lemma \ref{lemma::a_pair_of_vertices_in_Gamma} (b).
\end{proof}

\begin{proof}[Proof of Theorem \ref{theorem::achiralLef}]
Since each component $g\in\SL(2,\mathbb{Z})$ in a global monodromy of torus achiral Lefschetz fibrations is uniquely determined by $\iota(g)\in \PSL(2,\mathbb{Z})$,
by Theorem \ref{theorem::Livne ST},
torus achiral Lefschetz fibrations of type 
$\mathcal{O}=[\underbrace{I_1^+,\ldots,I_1^+}_{\text{$p$ components}},\underbrace{I_1^-,\ldots,I_1^-}_{\text{$q$ components}}]$
have pairwise Hurwitz equivalent global monodromies when $p\neq q$.

When $p=q$, each global monodromy is Hurwitz equivalent to a tuple in $\SL(2,\mathbb{Z})$ of the form 
\[(e_1,e_1^{-1},\ldots,e_p,e_p^{-1})\]
and hence can be written as a formal sum $\sum_{v\in V(\Gamma)} m_v\cdot v$.
We enumerate all possible formal sums of vertices that express minimal tuples according to the complexity among all Hurwitz equivalent tuples.
By Lemma \ref{lemma::no-ancestor-descendant-relationship}, there is no ancestor-descendant relationship between any two distinct vertices in such a formal sum of vertices.

If there are at least two distinct exceptional vertices $v_1$ and $v_2$ such that $m_{v_1}\ge 1$ and $m_{v_2}\ge 1$, 
by Case (\RNum{1}) in Lemma \ref{lemma::2p tuple without vertices joined by a directed path},
then all possible formal sums like this are associated with the same tuple up to Hurwitz equivalence.

If there exists the unique exceptional vertex $v$ occurring in the formal sum, then other vertices belong to the same sub-tree $T$ that is either $T_{BA^2}$, $T_{B}$ or $T_{BA}$.
Therefore, by Case (\RNum{2}) in Lemma \ref{lemma::2p tuple without vertices joined by a directed path}, all possible formal sums are in one-to-one correspondence with elements in $\Omega(T,p-m_v)$.

Otherwise, there is no exceptional vertex in the formal sum.
Again by Case (\RNum{2}) in Lemma \ref{lemma::2p tuple without vertices joined by a directed path}, all possible formal sums are in one-to-one correspondence with elements in $\Omega(T_{BA^2}\sqcup T_B\sqcup T_{BA},p)$.
\end{proof}

%% file: computability.tex
\section{Computability}
\label{section::computability}

We have included this section to demonstrate that all Hurwitz equivalences occurring in our results (including Theorem \ref{thmx::A}, Theorem \ref{thmx::B}, Theorem \ref{thmx::C}) are computable.
For Theorem \ref{thmx::A}, an algorithm exists to provide a sequence of elementary transformations that transforms one tuple to the other.
Its time complexity is
\begin{equation*}
    O\Big(n^5+n^3\sum_{i\in\{1,2\}} \sum_{j=1}^n l(g_j^{(i)})+n\big(\sum_{i\in\{1,2\}} \sum_{j=1}^n l(g_j^{(i)})\big)^2\Big).
\end{equation*}
We implement the algorithm in C\texttt{++} and make our code available on GitHub: \url{https://github.com/AHdoc/monodromy_normalisation}.

The main goal is to analyse the computability of Theorem \ref{theorem::Livne mathcalS}.
In particular, in Theorem \ref{theorem::Livne S->short}, we need an algorithm to make a tuple inverse-free by elementary transformations, but we cannot use Theorem \ref{theorem::Livne mathcalS} directly.

\textbf{Step 1.}
Suppose that $(g_1,\ldots,g_n)$ is a tuple in $\PSL(2,\mathbb{Z})$ whose components are conjugate to short elements.
Recall that Theorem \ref{theorem::Livne S->short} shortens the tuple by removing some pairs of the form $(x,x^{-1})$ and triples of the form $(l,l,l)$ with $l^3=1$;
the resulting tuple is an inverse-free tuple of short elements.
The proof uses an induction on an inverse-free tuple, within which one operation seeks to make the sum of $\mathcal{S}$-complexities strictly-smaller by a sequence of elementary transformations.
In fact, it is sufficient to check all transformations of the form $(R_i)^t$ with $t\in\{-2,-1,+1,+2\}$.

However, if we throw out the inverse-freeness, the induction still works well and ends with $\mu=0$, but the restoration operations (see Subsection \ref{subsection::contractions-restorations}) cannot result in a tuple of short elements.
Indeed, a pair of long elements of the form $(x,x^{-1})$ could have been combined into a single $1$, which was a contradiction in Proposition \ref{proposition::reduce_and_reloading}.
Therefore, restorations result in a tuple $(h_1,\ldots,h_m)$ that probably contains sub-pairs $(h_i,h_{i+1})=(x,x^{-1})$ or/and sub-triples $(h_i,h_{i+1},h_{i+2})=(l,l,l)$ with $l^3=1$.
Using cyclic permutations, we move these pairs and triples, if exist, to the rightmost positions.
Hence, we get a tuple of short elements, still denoted by $(h_1,\ldots,h_m)$.

Proposition \ref{proposition::reduce_and_reloading} asks us to handle each restoration $(h_1 h_2)\dashrightarrow (h_1,h_2)$ carefully.
We repeat the search for $(R_i)^t$ with $t\in\{-2,-1,+1,+2\}$ that makes $f(h_i)+f(h_{i+1})$ strictly-smaller.
In conclusion, Step 1 calls the following procedure.

\begin{breakablealgorithm}
\begin{algorithmic}[1]
\Procedure{Shorten}{$(g_1,\ldots,g_n)$}
\State $(h_1,\ldots,h_m) \gets (g_1,\ldots,g_n)$, $(k_1,\ldots,k_l) \gets$ empty tuple
\While{True} \Comment{see Subsection \ref{subsection::Conjugates of short elements and their tuples}}
    \If{$\exists i$ s.t. Operation $1$ is available on $(h_1,\ldots,h_m)$ for $i$}
    \State combine $(h_i,h_{i+1})$ into $h_ih_{i+1}$
    
    \ElsIf{$\exists i$ s.t. $h_i=1$}
    \State $(h_1,\ldots,h_m) \gets (1,h_1,\ldots,h_{i-1},h_{i+1},\ldots,h_m)$
    \Comment{via a cyclic permutation}
    
    \ElsIf{$\exists i$ and $t\in\{-2,-1,+1,+2\}$ s.t. $(R_i)^t$ makes $\sum_j f(h_j)$ strictly-smaller}
    \State $(h_1,\ldots,h_m) \gets (R_i)^t(h_1,\ldots,h_m)$
    
    \ElsIf{$\exists i$ s.t. $R_i$ keeps $\sum_j f(h_j)$ unchanged but makes $l(h_i)$ smaller}
    \State $(h_1,\ldots,h_m) \gets R_{i}(h_1,\ldots,h_m)$
    
    \Else
    \State \textbf{break while}
    \EndIf
\EndWhile

\While{$\exists$ a restoration on $h_i=\widetilde{h_1}\widetilde{h_2}$}
    \State $(h_1,\ldots,h_m)\gets (h_1,\ldots,h_{i-1},\widetilde{h_1},\widetilde{h_2},h_{i+1},\ldots,h_m)$

    \While{True}
    \If{$\exists t\in\{-2,-1,+1,+2\}$ s.t. $(R_i)^t$ makes $\sum_j f(h_j)$ strictly-smaller}
    \State $(h_1,\ldots,h_m) \gets (R_i)^t(h_1,\ldots,h_m)$
    \Else
    \State \textbf{break while}
    \EndIf
    \EndWhile
\EndWhile
\While{$\exists i$ s.t. $h_{i}h_{i+1}=1$}
    \State $(k_1,\ldots,k_l) \gets (h_i,h_{i+1})\bullet (k_1,\ldots,k_l),
    (h_1,\ldots,h_m) \gets (h_1,\ldots,h_{i-1},h_{i+2}\ldots,h_m)$
\EndWhile
\While{$\exists i$ s.t. $h_i=h_{i+1}=h_{i+2}$ and $h_{i}h_{i+1}h_{i+2}=1$}
    \State $(k_1,\ldots,k_l) \gets (h_i,h_{i+1},h_{i+2})\bullet (k_1,\ldots,k_l),
    (h_1,\ldots,h_m) \gets (h_1,\ldots,h_{i-1},h_{i+3},\ldots,h_m)$
\EndWhile
\State \textbf{return} $(h_1,\ldots,h_m)$ and $(k_1,\ldots,h_l)$
\EndProcedure
\end{algorithmic}
\end{breakablealgorithm}

The input of $\text{SHORTEN}$ is an arbitrary tuple $(g_1,\ldots,g_n)$ in $\PSL(2,\mathbb{Z})$ of conjugates of short elements.
The output is
the concatenation of a tuple
$(h_1,\ldots,h_m)$ of short elements and some pairs of the form $(x,x^{-1})$ and some triples of the form $(l,l,l)$, $l^3=1$, say $(h_1,\ldots,h_m)\bullet (k_1,\ldots,k_l)$.
In general, the tuple $(h_1,\ldots,h_m)$ is not inverse-free.
Therefore the difficulty is inherited to the next step.

\emph{
Time complexity: A step of the induction in $\text{SHORTEN}$ either decreases $\sum_i f(h_i)$ or decreases the number of the pairs $(i,j)$ such that $1\le i<j\le m$ but $l(h_i)>l(h_j)$.
Therefore, the time complexity of $\text{SHORTEN($(g_1,\ldots,g_n)$)}$ is $O((n^2+\sum_i l(g_i))n\sum_i l(g_i))$.
}

\textbf{Step 2.}
The tuple $(h_1,\ldots,h_m)$ of short elements probably has two components (resp. three components) that form a tuple of mutually inverse elements (resp. a triple of the form $(l,l,l)$ with $l^3=1$).
In this case, we move these components to the rightmost positions using cyclic permutations so that $(h_1,\ldots,h_m)$ is transformed into the concatenation of a shorter tuple, still denoted by $(h_1,\ldots,h_m)$, and a pair (resp. a triple).
However, cyclic permutations do not keep components of $(h_1,\ldots,h_m)$ short.
We end up with this reduction in an extra call on $\text{SHORTEN}((h_1,\ldots,h_m))$ and then repeat it.

\emph{
Time complexity: Using two/three cyclic permutations, we transform a tuple of short elements into a tuple, denoted by $(h_1,\ldots,h_m)$, such that $\sum_i l(h_i)=O(m)$.
Therefore, the above reduction is $O(m(m^2+m)m^2)=O(m^5)$.
}

From now on, we can assume that $(h_1,\ldots,h_m)$ contains at most $2$ components equal to $a$, at most $2$ components 
equal to $a^2$, at most $1$ component equal to $b$ and $a,a^2$ cannot appear together within this tuple.
We mark a tuple of short elements with $c_a$ components equal to $a$, $c_{a^2}$ components equal to $a^2$ and $c_b$ components equal to $b$ with the signature $[c_a,c_{a^2},c_b]$.
The following diagram shows a method to simplify such a tuple into a tuple of signature $[c_a,c_{a^2},c_b]$ with $c_a+c_{a^2}+c_b\le 1$ (c.f. Step 2, Step 3 and Step 4 in the proof of Theorem \ref{theorem::Livne mathcalS}).

\begin{center}
\begin{tikzpicture}[node distance={16mm}, thick]
\node (200) {$[2,0,0]$};
\node (010) [right=20mm of 200] {$[0,1,0]$};

\draw[->] (200) to
  node[below]
  {$(a,a)\dashrightarrow a^2$}
(010);

\node (020) [below=4mm of 200] {$[0,2,0]$};
\node (100) [right=20mm of 020] {$[1,0,0]$};

\draw[->] (020) to
  node[below]
  {$(a^2,a^2)\dashrightarrow a$}
(100);

\node (101) [below=4mm of 020] {$[1,0,1]$};
\node (021) [left=20mm of 101] {$[0,2,1]$};
\node (002) [below=10mm of 101] {$[0,0,2]$};
\node (011) [below=10mm of 002] {$[0,1,1]$};
\node (201) [left=20mm of 011] {$[2,0,1]$};
\node (001) [right=20mm of 002] {$[0,0,1]$};
\node (000) [right=20mm of 001] {$[0,0,0]$};

\draw[->] (021) to
  node[above]
  {$(a^2,a^2)\dashrightarrow a$}
(101);
\draw[->] (201) to
  node[below]
  {$(a,a)\dashrightarrow a^2$}
(011);
\draw[->] (101) to
  node[right]
  {
  \begin{tabular}{cc}
  $(a,s_0)\dashrightarrow b$ or \\$(s_2,a)\dashrightarrow b$
  \end{tabular}
  }
(002);
\draw[->] (101) to [out=0,in=90,looseness=1.4]
  node[right]
  {
  \begin{tabular}{cc}
  $(a,t_0)\dashrightarrow s_1$ or \\
  $(a,t_1)\dashrightarrow s_2$ or \\
  ~~$(a,t_2)\dashrightarrow s_0$
  \end{tabular}
  }
(001);
\draw[->] (011) to
  node[right]
  {
  \begin{tabular}{cc}
  $(a^2,t_2)\dashrightarrow b$ or \\$(t_0,a^2)\dashrightarrow b$
  \end{tabular}
  }
(002);
\draw[->] (011) to [out=0,in=270,looseness=1.4]
  node[right]
  {
  \begin{tabular}{cc}
  ~~$(a^2,s_0)\dashrightarrow t_2$ \\
  or $(a^2,s_1)\dashrightarrow t_0$ \\
  or $(a^2,s_2)\dashrightarrow t_1$
  \end{tabular}
  }
(001);

\draw[<->] (101) to [out=190,in=170,looseness=1.4]
  node[left]
  {
  \begin{tabular}{cc}
  $(s_1,a,s_1)\rightarrow (s_1,s_0,a)$\\
  $\dashrightarrow (a,a)\dashrightarrow a^2$ \\
  \\
  $(t_1,a,t_1)\rightarrow (t_1,t_2,a^2)$\\
  $\dashrightarrow (a^2,a^2)\dashrightarrow a$
  \end{tabular}
  }
(011);

\draw[->] (002) to [out=10,in=90,looseness=0.5]
  node[below]
  {~~~~~~~~~~~~~~~$(b,b)\dashrightarrow b$}
(000);

\end{tikzpicture} 
\end{center}

In the diagram, a reduction from a tuple of signature $[c_a,c_{a^2},c_b]$ to a tuple of signature $[c'_a,c'_{a^2},c'_b]$ is a directed edge endowed with some elementary transformations and contractions on a pair or a triple.
The reduction starts with cyclic permutations that create a sub-pair or a sub-triple with which the edge is first endowed.
It ends with a call on $\text{SHORTEN}$.

Signature $[0,0,2]$ is the only exception that does not satisfy the hypotheses on the tuple.
However, with at most $4$ contractions, any tuple of short elements satisfying hypotheses can be transformed into a tuple of signature $[c_a,c_{a^2},c_b]$ with $c_a+c_{a^2}+c_b\le 1$.
Indeed, if a tuple of signature $[1,0,1]$ has to be aimed at a tuple of signature $[0,1,1]$, then it is a tuple of $a,b,s_1$ and it is transformed into a tuple of $a^2,b,s_0,s_1,s_2$ of signature $[0,1,1]$, which is further transformed into a tuple of signature $[0,0,1]$.

\emph{
Time complexity:
Both cyclic permutation and contraction are linear.
We have shown that a cyclic permutation on a tuple of short elements results in a tuple such that $\sum_i l(h_i)=O(m)$.
The simplification along the diagram calls $\text{SHORTEN}$ at most $4$ times, therefore its time complexity is $O((m^2+m)m^2)=O(m^4)$.
}

From now on, we can further assume that $(h_1,\ldots, h_m)$ contains at most $1$ component equal to either $a$, $a^2$ or $b$.
Proposition \ref{proposition::ST} claims that, if $(h_1,\ldots,h_m)$ is inverse-free, then it is either a tuple of $a,a^2,b,s_0,s_1,s_2$ or a tuple of $a,a^2,b,t_0,t_1,t_2$.
The proof is a rearrangement of $s_0,s_1,s_2$ and $t_0,t_1,t_2$.
To introduce a similar reduction, we provide the following procedure.

\begin{breakablealgorithm}
\begin{algorithmic}[1]
\Procedure{\text{st-Rearrangement}}{$(\kappa_1,\ldots,\kappa_l)$}
\If{$\exists i$ s.t. $\kappa_i\in\{a,a^2,b\}$}
    \State $(\kappa_1,\ldots,\kappa_l) \gets (\kappa_i,\kappa_{i+1},\ldots,\kappa_l,\kappa_1,\ldots,\kappa_{i-1})$ \Comment{via a cyclic permutation}
\EndIf
\While{$\exists i$ s.t. $\kappa_i\in\{t_0,t_1,t_2\}$ and $\kappa_{i+1}\in\{s_0,s_1,s_2\}$}
    \If{$(\kappa_i,\kappa_{i+1})\in\{(t_0,s_0),(t_1,s_1),(t_2,s_2)\}$}
        \State \textbf{return} $(\kappa_1,\ldots,\kappa_{i-1},\kappa_{i+2},\ldots,\kappa_l)$
    \ElsIf{$(\kappa_i,\kappa_{i+1})\in\{(t_0,s_1),(t_1,s_2),(t_2,s_0)\}$}
        \State $(\kappa_1,\ldots,\kappa_l) \gets R_i^{-1}(\kappa_1,\ldots,\kappa_l)$
    \ElsIf{$(\kappa_i,\kappa_{i+1})\in\{(t_0,s_2),(t_1,s_0),(t_2,s_1)\}$}
        \State $(\kappa_1,\ldots,\kappa_l) \gets R_i(\kappa_1,\ldots,\kappa_l)$
    \EndIf
\EndWhile
\If{$\exists i<j$ s.t. $(\kappa_i,\kappa_j)\in\{(s_0,t_0),(s_1,t_1),(s_2,t_2)\}$}
    \State \textbf{return} $(\kappa_1,\ldots,\kappa_{i-1},\kappa_{i+1}^{\kappa_j},\ldots,\kappa_{j-1}^{\kappa_j},\kappa_{j+1},\ldots,\kappa_l)$
\ElsIf{$\exists i<j<k$ s.t. $(\kappa_i,\kappa_j,\kappa_k)\in\{(s_0,t_1,t_2),(s_1,t_2,t_0),(s_2,t_0,t_1)\}$}
    \State $(\kappa_1,\ldots,\kappa_l) \gets (R_2)^2 (\kappa_i,\kappa_j,\kappa_k,\kappa_1^{\kappa_i\kappa_j\kappa_k},\ldots,\kappa_{i+1}^{\kappa_j\kappa_k},\ldots,\kappa_{j+1}^{\kappa_k},\ldots,\kappa_{k+1},\ldots)$
    \State \textbf{return} $(\kappa_3,\ldots,\kappa_l)$
\ElsIf{$\exists i<j<k$ s.t. $(\kappa_i,\kappa_j,\kappa_k)\in\{(s_2,s_1,t_0),(s_0,s_2,t_1),(s_1,s_0,t_2)\}$}
    \State $(\kappa_1,\ldots,\kappa_l) \gets (R_2)^{-1} (\kappa_i,\kappa_j,\kappa_k,\kappa_1^{\kappa_i\kappa_j\kappa_k},\ldots,\kappa_{i+1}^{\kappa_j\kappa_k},\ldots,\kappa_{j+1}^{\kappa_k},\ldots,\kappa_{k+1},\ldots)$
    \State \textbf{return} $(\kappa_3,\ldots,\kappa_m)$
\Else
\State \textbf{return}
$(\kappa_1,\ldots,\kappa_m)$
\EndIf
\EndProcedure
\end{algorithmic}
\end{breakablealgorithm}

The input of $\text{ST-REARRANGEMENT}$ is a tuple of short elements, say $(\kappa_1,\ldots,\kappa_l)$, that contains at most $1$ component equal to $a,a^2$ or $b$. The output is either a tuple of short elements of length $l$, meaning that $(\kappa_1,\ldots,\kappa_l)$ is inverse-free, or a tuple $(\widetilde{\kappa_1},\ldots,\widetilde{\kappa_{l-2}})$ of length $l-2$, meaning that $(\kappa_1,\ldots,\kappa_l)$ can be transformed into $(\widetilde{\kappa_1},\ldots,\widetilde{\kappa_{l-2}})\bullet (s_i,t_i)$ with some $i$ by elementary transformations.

We call $\text{ST-REARRANGEMENT}$ and $\text{SHORTEN}$ with $(h_1,\ldots,h_m)$ repeatedly unless the tuple is inverse-free.
In conclusion, Step $2$ calls a procedure, named as $\text{INVERSE-FREE}$, whose input is a tuple $(h_1,\ldots,h_m)$ of short elements and output is an inverse-free tuple $(\kappa_1,\ldots,\kappa_l)$ of short elements such that at most $1$ component is equal to $a,a^2$ or $b$.

\emph{
Time complexity:
The procedure $\text{ST-REARRANGEMENT}$ decreases the length of the tuple and transform a tuple of short elements into a tuple, denoted by $(\kappa_1,\ldots,\kappa_l)$, such that $\sum_i l(\kappa_i)=O(l)$.
The time complexity of $\text{ST-REARRANGEMENT}$ is $O(l^4)$.
}

\emph{
Meanwhile, $\text{SHORTEN}$ transforms the tuple back to a tuple of short elements.
In conclusion, the time complexity of $\text{INVERSE-FREE($h_1,\ldots,h_m$)}$ is
$O(m^5+m^4+m(m^4+((m^2+m)m^2)))=O(m^5)$.
}

\textbf{Step 3.}
To slightly improve the complement to Theorem of R. Livn\'e introduced in \cites[p.180-187]{moishezon1977complex} to a tuple of $a,b,s_0,s_1,s_2$ that at most $1$ component is equal to $a$ or $b$, we first introduce the following procedure named as $\text{MOISHEZON}$ (c.f. Proposition \ref{proposition::reduced_form_s0_s2}).

\begin{breakablealgorithm}
\begin{algorithmic}[1]
\Procedure{\text{MOISHEZON}}{$(\kappa_1,\ldots,\kappa_l)$}
\If{$\exists i$ s.t. $\kappa_i\in\{a,b\}$}
    \State $(\kappa_1,\ldots,\kappa_l) \gets (\kappa_i,\kappa_{i+1},\ldots,\kappa_l,\kappa_1,\ldots,\kappa_{i-1})$ \Comment{via a cyclic permutation}
\EndIf
\While{True}
    \If{$\exists i$ s.t. $(\kappa_i,\kappa_{i+1})=(s_1,s_0)$} \Comment{decrease $\#$ of $s_1$}
        \State $(\kappa_1,\ldots,\kappa_l) \gets R_i (\kappa_1,\ldots,\kappa_l)$
    \ElsIf{$\exists i$ s.t. $(\kappa_i,\kappa_{i+1})=(s_2,s_1)$} \Comment{decrease $\#$ of $s_1$}
        \State $(\kappa_1,\ldots,\kappa_l) \gets R_i^{-1} (\kappa_1,\ldots,\kappa_l)$
    \ElsIf{$\exists i$ s.t. $(\kappa_i,\kappa_{i+1},\kappa_{i+2})=(s_0,s_2,s_0)$} \Comment{Claim $1$}
        \State $(\kappa_1,\ldots,\kappa_l) \gets (\kappa_1,\ldots,\kappa_{i-1},s_2,s_0,s_2,\kappa_{i+3},\ldots,\kappa_l)$
    \ElsIf{$\exists i+1< j$ s.t. $(\kappa_i,\ldots,\kappa_{j+1})=(s_1,s_2,\ldots,s_2,s_0,s_2)$} \Comment{Claim $2$}
        \State $(\kappa_1,\ldots,\kappa_l) \gets (\kappa_1,\ldots,\kappa_{i-1},s_1,s_0,s_2,s_0\ldots,s_0,\kappa_{j+2},\ldots,\kappa_l)$
    \ElsIf{$\exists i+2<j$ s.t $(\kappa_{i-1},\ldots,\kappa_{j+1})=(s_2,s_0,s_2,\ldots,s_2,s_0,s_2)$} \Comment{Claim $3$}
        \State $(\kappa_1,\ldots,\kappa_l) \gets (\kappa_1,\ldots,\kappa_{i-2},s_2,s_0,s_2,s_0,s_2,s_0,\ldots,s_0,\kappa_{j+2},\ldots,\kappa_l)$
        \State $(\kappa_1,\ldots,\kappa_l) \gets (\kappa_1,\ldots,\kappa_{i-2},s_0,\ldots,s_0,\kappa_{j+2},\ldots,\kappa_l)$
    \ElsIf{$\exists i$ s.t. $(\kappa_i,\ldots,\kappa_{i+5})=(s_0,s_2,s_2,s_0,s_2,s_2)$} \Comment{Claim 4}
        \State $(\kappa_1,\ldots,\kappa_l) \gets (\kappa_1,\ldots,\kappa_{i-1},\kappa_{i+6},\ldots,\kappa_l)$
    \ElsIf{$\exists i$ s.t. $(\kappa_i,\ldots,\kappa_{i+5})=(s_2,s_2,s_0,s_2,s_2,s_0)$} \Comment{Claim 4}
        \State $(\kappa_1,\ldots,\kappa_l) \gets (\kappa_1,\ldots,\kappa_{i-1},\kappa_{i+6},\ldots,\kappa_l)$
    \Else
        \State \textbf{break while}
    \EndIf
\EndWhile
\EndProcedure
\end{algorithmic}
\end{breakablealgorithm}

The input of $\text{MOISHEZON}$ is an inverse-free tuple $(\kappa_1,\ldots,\kappa_l)$ of $a,b,s_0,s_1,s_2$ that at most $1$ component is equal to $a$ or $b$.
The output, denoted by $(\kappa'_1,\ldots,\kappa'_{l'})$, is again a tuple of $a,b,s_0,s_1,s_2$ and shows that $(\kappa_1,\ldots,\kappa_l)$ can be transformed into the concatenation of $(\kappa'_1,\ldots,\kappa'_{l'})$ and some sextuples of the form $(s_0,s_2,s_0,s_2,s_0,s_2)$ by elementary transformations.
If $(\kappa_1,\ldots,\kappa_l)$ is a tuple of $s_0,s_1,s_2$, then $l'=0$;
otherwise, by Lemma \ref{lemma::b,s} and Lemma \ref{lemma::a,s}, either
\begin{itemize}[-]
    \item $(\kappa'_1,\kappa'_2)=(a,s_0)$ or $(\kappa'_l,\kappa'_1)=(s_2,a)$ or $(\kappa'_l,\kappa'_1,\kappa'_2)=(s_1,a,s_1)$, or
    \item $(\kappa'_1,\ldots,\kappa'_l)$ starts with $(b,s_2)$ or $(b)\bullet (s_0)^{v_{0,1}}\bullet (s_2,s_0)$ with $v_{0,1}\ge 1$, or
    \item $(\kappa'_2,\ldots,\kappa'_l,\kappa'_1)$ ends with $(s_0,b)$ or $(s_2,s_0)\bullet (s_2)^{u_{\mu,n_{\mu}}}\bullet (b)$ with $u_{\mu,n_{\mu}}\ge 1$.
\end{itemize}
By elementary transformations and at most $2$ contractions, the tuple $(\kappa'_1,\ldots,\kappa'_{l'})$ is further transformed into a tuple of $a,a^2,b,s_0,s_1,s_2$ that at most $1$ component is equal to $a,a^2$ or $b$.

\emph{
Time complexity:
The procedure $\text{MOISHEZON($(\kappa_1,\ldots,\kappa_l)$)}$ is looping, seeks the minimal number of components equal to $s_1$ and seeks the minimal according to the lexicographical order given by $s_0<s_2$.
Therefore, the number of times that the loop loops is related to the number of reverse pairs, i.e. $i<j$ but $\kappa_i>\kappa_j$ according to the lexicographical order, which is $O(l^2)$.
The time complexity of $\text{MOISHEZON}$ is $O(l^5)$.
}

In Step $3$, we consider an inverse-free tuple $(\kappa_1,\ldots,\kappa_l)$ of short elements that contains at most $1$ component equal to $a,a^2$ or $b$.
Let $\mathcal{A}$ be the set of elements in $(\kappa_1,\ldots,\kappa_l)$.
We follow the diagram below to reduce the tuple using elementary transformations and at most $3$ contractions.

\begin{center}
\begin{tikzpicture}[node distance={16mm}, thick]
\node (1) {
  \begin{tabular}{cc}
  $\mathcal{A}=$\\
  $\{a,s_0,s_1,s_2\}$
  \end{tabular}
};

\node (2) [right=20mm of 1] {$\cdot$};
\draw[->] (1) to node[below]
  {$\text{MOISHEZON}$}
(2);

\node (3) [right=20mm of 2] 
{
  \begin{tabular}{cc}
  $\mathcal{A}=$\\
  $\{a^2,s_0,s_1,s_2\}$
  \end{tabular}
};
\draw[->] (2) to node[above]
  {
  \begin{tabular}{cc}
  $(s_1,a,s_1)\rightarrow (s_1,s_0,a)$\\
  $\dashrightarrow (a,a)\dashrightarrow a^2$
  \end{tabular}
  }
(3);

\node (4) [right=20mm of 3]
{
  \begin{tabular}{cc}
  $\mathcal{A}=$\\
  $\{s_0,s_1,s_2\}$
  \end{tabular}
};
\draw[->] (3) to node[below]
  {
  \begin{tabular}{cc}
  $(a^2,s_i)\dashrightarrow t_{i-1}$,\\
  $\text{INVERSE-FREE}$
  \end{tabular}
  }
(4);

\node (5) [below=20mm of 2]
{
  \begin{tabular}{cc}
  $\mathcal{A}=$\\
  $\{b,s_0,s_1,s_2\}$
  \end{tabular}
};
\draw[->] (2) to
[out=270,in=110,looseness=1.4]
node[left]
  {
  \begin{tabular}{cc}
  $(a,s_0)\dashrightarrow b$\\\\
  \end{tabular}
  }
(5);
\draw[->] (2) to
[out=270,in=70,looseness=1.4]
node[right]
  {
  \begin{tabular}{cc}
  $(s_2,a)\dashrightarrow b$\\\\
  \end{tabular}
  }
(5);

\node (6) [right=20mm of 5] {$\cdot$};
\draw[->] (5) to node[below]
  {$\text{MOISHEZON}$}
(6);
\draw[->] (6) to
[out=90,in=250,looseness=1.4]
node[left]
  {
  \begin{tabular}{cc}
  \\\\
  $(b,s_2)\dashrightarrow a^2$
  \end{tabular}
  }
(3);
\draw[->] (6) to
[out=90,in=290,looseness=1.4]
node[right]
  {
  \begin{tabular}{cc}
  \\\\
  $(s_0,b)\dashrightarrow a^2$
  \end{tabular}
  }
(3);

\node (7) [right=20mm of 6] {$(s_0,s_2,s_0,s_2,s_0,s_2)^{m_s}$};
\draw[->] (4) to
node[right]
  {\text{MOISHEZON}}
(7);
\end{tikzpicture} 
\end{center}

A symmetric procedure, named as $\text{MOISHEZON}^{-1}$, can handle a tuple of $a^2,b,t_0,t_1,t_2$ that at most $1$ component is equal to $a^2$ or $b$.
Therefore, we have a symmetric diagram for the rest of the cases.

\emph{
Time complexity:
In conclusion, Step $3$ contracts the tuple at most $3$ times, calls $\text{MOISHEZON}$ several times and calls $\text{INVERSE-FREE}$ at most once.
Its time complexity is $O(l^5)$.
}

\textbf{Step 4.}
We have shown that by elementary transformations and at most $7$ contractions, the initial tuple $(g_1,\ldots,g_n)$ is transformed into a concatenation of the following tuples.
\begin{align*}
&(s_1,t_1)^Q,
(t_1,s_1)^Q,
(a,a^2)^Q,
(a^2,a)^Q,
(b,b)^Q,
(a,a,a)^Q,
(a^2,a^2,a^2)^Q,\\
&(s_0,s_2,s_0,s_2,s_0,s_2),
(t_0,t_2,t_0,t_2,t_0,t_2)
\end{align*}
where $Q\in \PSL(2,\mathbb{Z})$ is arbitrary, such that $(s_0,s_2,s_0,s_2,s_0,s_2)$ and $(t_0,t_2,t_0,t_2,t_0,t_2)$ cannot appear at the same time.
By Lemma \ref{lemma::move subtuples}, the concatenation can be transformed into such with a desired order by elementary transformations.
Besides, a pair of the form $(x,x^{-1})$ can be transformed into $(x^{-1},x)$ by an elementary transformation.
Therefore, we handle each restoration carefully and obtain
\begin{equation*}
(h_1,\ldots,h_m)\bullet
(s_0,s_2)^{3m_s}\bullet
(t_0,t_2)^{3m_t}\bullet
\prod_{i=1}^{m_{st}} (s_1,t_1)^{X_i}\bullet
\prod_{i=1}^{m_a} (a,a^2)^{Y_i}\bullet
\prod_{i=1}^{m_b} (b,b)^{Z_i}\bullet
\prod_{\epsilon=\pm 1}
\prod_{i=1}^{n_{\epsilon}} (a^{\epsilon},a^{\epsilon},a^{\epsilon})^{P_{\epsilon,i}}
\end{equation*}
with $m_sm_t=0$, $X_i,Y_i,Z_i,P_{\epsilon,i}\in \PSL(2,\mathbb{Z})$,
which is Hurwitz equivalent to the initial tuple.

The tuple $(h_1,\ldots,h_m)$, called the exceptional part of the resulting tuple, is a tuple of short elements.
In fact, if $p=q$, $|p_a-q_a|\equiv 0\pmod{3}$ and $n_b$ is even in Theorem \ref{theorem::Livne mathcalS}, the exceptional part does not exist anymore, i.e. $m=0$, thus we have already finished the computation.
Otherwise, the exceptional tuple $(h_1,\ldots,h_m)$ contains a generating set;
by Lemma \ref{lemma::tuples containing a generating set}, we obtain
\begin{equation*}
(h_1,\ldots,h_m)\bullet
(s_0,s_2)^{3m_s}\bullet
(t_0,t_2)^{3m_t}\bullet
(s_1,t_1)^{m_{st}}\bullet
(a,a^2)^{m_a}\bullet
(b,b)^{m_b}\bullet
(a,a,a)^{n_1}\bullet
(a^2,a^2,a^2)^{n_{-1}}.
\end{equation*}
With a slight adjustment using cyclic permutations, we may further assume that $n_1 n_{-1}=0$.

The length of the exceptional tuple is bounded by a constant.
In fact, we claim that the exceptional tuple $(h_1,\ldots,h_m)$ satisfies $m\le 13$ without further discussion.
The proof of Theorem \ref{theorem::Livne mathcalS} has revealed that a partial normal form can be transformed into the desired normal form by cyclic transformations and elementary transformations that keep each component short.
The whole computation ends with a brute-force search.

\emph{
Time complexity:
The brute-force search is $O(1)$ as the length of the exceptional tuple is bounded by a constant.
The time complexity of Step $4$ is $O(n\sum_i l(g_i)+n^3+1)$.
Hence, the computation of Theorem \ref{theorem::Livne mathcalS} has the time complexity $O(n^5+n^3\sum_i l(g_i)+n(\sum_i l(g_i))^2)$.
}